\documentclass[english,refpage,smallextended,envcountsect,envcountsame]{svjour3}
\usepackage{mathptmx}
\usepackage[T1]{fontenc}
\usepackage[latin9]{inputenc}
\usepackage{babel}
\usepackage{url}
\usepackage{amsmath}
\usepackage{amssymb}
\usepackage{makeidx}
\makeindex
\usepackage{graphicx}
\usepackage[all]{xy}
\usepackage{nomencl}
\providecommand{\printnomenclature}{\printglossary}
\providecommand{\makenomenclature}{\makeglossary}
\makenomenclature
\usepackage[unicode=true,
 bookmarks=true,bookmarksnumbered=true,bookmarksopen=true,bookmarksopenlevel=3,
 breaklinks=false,pdfborder={0 0 1},backref=false,colorlinks=false]
 {hyperref}
\hypersetup{pdftitle={Nonexistence of Wandering Domains for Infinitely Renormalizable Hénon Maps},
 pdfauthor={Dyi-Shing Ou},
 pdfkeywords={Hénon Maps, Wandering Domain}}

\makeatletter
\numberwithin{equation}{section}
\numberwithin{figure}{section}

\usepackage{tikz}
\newcommand{\frmu}{\frm{^\}}}
\newcommand{\frmd}{\frm{_\}}}

\journalname{Inventiones mathematicae}

\spnewtheorem*{theorem*}{Theorem}{\bf}{\it}

\AtBeginDocument{%
  \paperwidth=\dimexpr
    1in + \oddsidemargin
    + \textwidth
    + 1in + \oddsidemargin
  \relax
  \paperheight=\dimexpr
    1in + \topmargin
    + \headheight + \headsep
    + \textheight
    + 1in + \topmargin
  \relax
  \usepackage[pass]{geometry}\relax
}

\smartqed  
\spnewtheorem*{xproof}{}{\itshape}{\rmfamily}

\renewenvironment{proof}[1][\proofname]
 {\renewcommand\xproofname{#1}\xproof}
 {\qed\endxproof}

\@ifundefined{showcaptionsetup}{}{%
 \PassOptionsToPackage{caption=false}{subfig}}
\usepackage{subfig}
\makeatother

  \providecommand{\proofname}{Proof}

\begin{document}

\title{Nonexistence of Wandering Domains for Infinitely Renormalizable
Hénon Maps}

\author{Dyi-Shing Ou}

\institute{Dyi-Shing Ou \at Mathematics Department, Stony Brook University,
Stony Brook NY, 11794-3651, USA \\
\email{dsou@math.stonybrook.edu}}

\date{\today}
\maketitle
\begin{abstract}
This article extends the theorem of the absence of wandering domains
from unimodal maps to infinitely period-doubling renormalizable Hénon-like
maps in the strongly dissipative (area contracting) regime. The theorem
solves an open problem proposed by several authors \cite{van2010one,lyubich2011renormalization},
and covers a class of maps in the nonhyperbolic higher dimensional
setting. The classical proof for unimodal maps breaks down in the
Hénon settings, and two techniques, ``the area argument'' and ``the
good region and the bad region'', are introduced to resolve the main
difficulty.

The theorem also helps to understand the topological structure of
the heteroclinic web for such kind of maps: the union of the stable
manifolds for all periodic points is dense.

\end{abstract}

\begin{acknowledgements}
The author thanks Marco Martens for many discussions on Hénon maps,
careful reading the manuscript, and useful comments on the results.
The author also thanks Department of Mathematics, Stony Brook University
for financial support.
\end{acknowledgements}

\section{Introduction }

This article studies the question of the existence of wandering domains
for Hénon-like maps. A Hénon-like map\index{Hénon-like map} is a
real two-dimensional continuous map that has the form 
\begin{equation}
F(x,y)=(f(x)-\epsilon(x,y),x)\label{eq:Henon-like map}
\end{equation}
where $f$ is a unimodal map\index{unimodal map} (will be defined
later) and $\epsilon$ is a small perturbation. For renormalization
purposes, the Hénon-like maps in consideration are all real analytic
and strongly dissipative (the Jacobian $\left|\frac{\partial\epsilon}{\partial y}\right|$
is small\footnote{This article assumes $\epsilon$ is small. This implies that $\left|\frac{\partial\epsilon}{\partial y}\right|$
is also small in the analytic setting.}). One can see from the definition, Hénon-like maps are a generalization
of classical Hénon maps \cite{henon1976two} (two-parameters polynomial
maps) to the analytic settings and extension of unimodal maps to higher
dimensions.

Strongly dissipative Hénon-like maps are the maps that are close to
unimodal maps. They share some dynamical properties with unimodal
maps. For example, the tool of unimodal renormalization can be adopted
to Hénon-like maps \cite{de2005renormalization,lyubich2011renormalization,hazard2011henon},
the renormalization operator is hyperbolic \cite[Theorem 4.1]{de2005renormalization},
and an infinitely renormalizable Hénon-like map has an attracting
Cantor set\index{Cantor set} \cite[Section 5.2]{de2005renormalization}.
However, there are also some properties that make Hénon-like maps
distinct from unimodal maps. For example, the Cantor set\index{Cantor set}
for infinitely renormalizable Hénon-like maps is not rigid \cite[Theorem 10.1]{de2005renormalization}
and a universal model can not be presented by a finite dimensional
family of Hénon-like maps \cite{hazard2014zero}. In the degenerate
case, unimodal maps do not have wandering intervals \cite{guckenheimer1979sensitive,de1988one,de1989structure,martens1992julia}.
It is natural to ask whether this property can be promoted to Hénon-like
maps.

The study of wandering domains has a broad interest in the field of
dynamics. In one-dimension, the problem has been widely studied and
there are many important consequences due to the absence of wandering
intervals/domains. However, only a few systems in higher dimensions
were known not having wandering domains.

In real one-dimension, showing the absence of \index{wandering interval}wandering
intervals in a system is important to solve the classification problem.
For circle homeomorphisms, a sequence of works \cite{hall1981ac,yoccoz1984n,hu1997topological,norton1999denjoy}
follows after Denjoy \cite{denjoy1932courbes} showed that a circle
homeomorphism with irrational rotation number (the average rotation
angle is irrational) does not have a wandering domain if the map is
smooth enough. Those maps are conjugated to the rigid rotation with
the same rotation number by a classical theorem from Poincare \cite{poincare1881memoire}.
For multimodal maps, a full family is a family of multimodal maps
that exhibits all relevant dynamical behavior. A multimodal map that
does not have a wandering interval is conjugated to an element in
a full family \cite{guckenheimer1979sensitive,de1988one,de1989structure,lyubich1989non,blokh1989non,martens1992julia}.

In complex dimension one, Sullivan's no-wandering-domain theorem \cite{sullivan1985}
fully solves the problem for rational maps. The theorem says that
a rational map on the Riemann sphere does not have wandering Fatou
components. As a consequence, this theorem completes the last puzzle
for the classification of Fatou components \cite{fatou1919equations,julia1918memoire}.
Thus, the main interest turns to transcendental maps. In general,
there are transcendental maps that have wandering domains \cite{baker1976entire,baker1984wandering,herman1984exemples,sullivan1985,eremenko1987examples,bishop2015constructing,fagella2015wandering}.
There are also some types of transcendental maps that do not have
wandering domains \cite{goldberg1986finiteness,eremenko1992dynamical,bergweiler1993limit,mihaljevic-brandt2013}. 

In real higher dimensions, the problem for wandering domains is still
wide open. There is no reason to expect the absence of wandering domains
\cite{yoccoz1980}, especially when the regularity is not enough as
in one-dimension \cite{mcswiggen1993,mcswiggen1995diffeomorphisms,bonatti1994wandering}.
The classification problem fails between any two different levels
of differentiability for diffeomorphisms on $d$-manifold with $d\neq1,4$
\cite{Harrison1975,Harrison1979}. Examples are found in polynomial
skew-product maps having wandering domains \cite{astorg2016two}.
Non-hyperbolic phenomena also play a role in building counterexamples
\cite{colli2001non,kiriki2015takens,kiriki2016non}. A relevant work
by Kiriki and Soma \cite{kiriki2015takens} found Hénon-like maps
having wandering domains by using a homoclinic tangency of some saddle
fixed point \cite{kiriki2010coexistence,kiriki2013existence}. On
the other hand, there are studies \cite{norton1996wandering,kwakkel2009,kwakkel2010topological,navas_wandering_2017}
suggests that some types of systems may not have wandering domains.
However, only a few \cite{norton1991area,bonatti1994wandering} were
discovered not having wandering domains.

In complex higher dimension, counterexamples in transcendental maps
can be constructed from one-dimensional examples \cite{fornaess1998fatou}
by taking direct products. For polynomial maps, very little was known
about the existence of wandering Fatou components until recent developments
on polynomial skew-products \cite{lilov2004fatou,astorg2016two,peters2015fatou,peters2016fatou,peters2016},
which are the maps of the form
\[
F(z,w)=(f(z,w),g(w)).
\]
The first example was given by Astorg, Buff, Dujardin, Peters, and
Raissy \cite{astorg2016two}, who found a polynomial skew-product
possessing a wandering Fatou component as the quasi-conformal methods
break down. The reader can refer to the survey \cite{raissy2016}
for more details about other relevant work on polynomial skew-product
\cite{lilov2004fatou,peters2015fatou,peters2016fatou,peters2016}.
For complex Hénon maps\footnote{The study of complex Hénon maps covers a broader class of functions
motivated by the classification of polynomial automorphisms \cite{friedland1989dynamical}.
It allows the map $f$ in (\ref{eq:Henon-like map}) to be any polynomial
or analytic map.}, a recent paper by Leandro Arosio, Anna Miriam Benini, John Erik
Fornaess, and Han Peters \cite{arosio2017dynamics} found a transcendental
Hénon map exhibiting a wandering domain. Nevertheless, the problem
is still unsolved \cite{bedford2015dynamics} for complex polynomial
Hénon maps \cite{hubbard1986henon,hubbard1995henon}.

In this paper, a wandering domain is a nonempty open set that does
not intersect the stable manifold of any saddle periodic points. This
definition is weaker than the classical notion because it excludes
the condition having a disjoint orbit. The reason to drop this condition
is to allow the usage study the topological structure of attractors
which will be discussed later. Dropping the condition also makes the
conclusion of the theorem stronger compared to the classical definition.
In fact, this condition is redundant in the unimodal setting (See
Remark \ref{rem:Wandering domain compare with the classical definition}).

The main result of this article, Theorem \ref{thm:nonexistence of wandering domain},
is stated as follow. 
\begin{theorem*}
A strongly dissipative infinitely period-doubling renormalizable Hénon-like
map does not have wandering domains.
\end{theorem*}
The theorem covers a class of maps in the higher dimensional nonhyperbolic
setting \cite[Corollary 6.2]{lyubich2011renormalization} and solves
an open problem proposed by van Strien \cite{lyubich2011renormalization},
Lyubich, and Martens \cite{lyubich2011renormalization}. The result
does not overlap with the previous work by Kiriki and Soma \cite{kiriki2015takens}.
The Hénon-like maps in this article are real analytic and the fixed
points are far away from having a homoclinic tangency, while the examples
they found having wandering domains only have finite differentiability
and their construction relies on the existence of a homoclinic tangency
of a fixed point \cite{kiriki2010coexistence,kiriki2013existence}. 

The condition of being infinitely renormalizable is imposed to the
theorem to gain a self-similarity between different scale. A map is
renormalizable means that a higher iterate of the map has a similar
topological structure on a smaller scale. Several papers \cite{bonatti1994wandering,martens2014hyperbolicity,martens2016coexistence}
in different contexts show that this condition will ensure the absence
of wandering domains. In this paper, we center on infinitely renormalizable
maps of period-doubling combinatorics type which is one of the most
fundamental types of maps. The reader will see later in the proof
that the condition infinitely renormalizable is essential because
that the area where bad things (nonhyperbolic phenomena) happens,
called the bad region, becomes smaller when a map gets renormalized
more times.

The theorem is important because it helps us to understand the structure
of attractors. An attracting set is a closed set such that many points
evolve toward the set. Hénon maps are famous for its chaotic limiting
behavior since Hénon first discovered the strange attractor in the
classical Hénon family \cite{henon1976two}. The $\omega$-limit set
of a point is an attractor which characterizes the long-time behavior
of a single orbit. For an infinitely period-doubling renormalizable
Hénon-like map, the map has only two types of $\omega$-limit set
\cite{gambaudo1989henon,lyubich2011renormalization}: a saddle periodic
orbit and the renormalization Cantor attractor. From this dichotomy,
a wandering domain is equivalently a non-empty open subset of the
basin of the Cantor attractor. The theorem of the absence of wandering
domains implies that the union of the stable manifolds is dense. In
other words, the basin of the Cantor set has no interior even though
it has full Lebesgue measure.

Two tools are introduced to prove the theorem: the bad region\index{bad region}
and the area argument (thickness\index{thickness}). The bad region
is a set in the domain where the length expansion argument from unimodal
maps breaks down. This is where the main difficulty of extending the
theorem occurs. The solution to this is the area argument which is
also a dimension two feature. These two concepts make the Hénon-like
maps different from the unimodal maps. 

The tools may be used to prove the nonexistence of wandering domains
in other contexts. One is infinitely renormalizable Hénon-like maps
with arbitrary combinatorics \cite{hazard2011henon}. The definition
of the bad region carries over to the arbitrary combinatorics case
directly. It is also possible to generalize the area argument because
the tip of those maps also has a universal shape \cite[Theorem 6.1]{hazard2011henon}.
However, the expansion argument breaks down for other combinatorics.
This may be solved by studying the hyperbolic length instead of the
Euclidean length. 

\subsection*{Outline of the article}

In this article, chapters, sections, or statements marked with a star
sign ``{*}'' means that the main theorem, Theorem \ref{thm:nonexistence of wandering domain},
does not depend on them. Terminologies in the outline will be defined
precisely in later chapters.

Chapters \ref{sec:Notations}, \ref{sec:Unimodal-Maps}, \ref{sec:Renormalization of Henon Like Map},
and \ref{sec:Infinite Renormalizable H=0000E9non-Like Maps} are the
preliminaries of the theorem. The chapters include basic knowledge
and conventions that will be used in the proof. Most of the theorems
in Chapter \ref{sec:Renormalization of Henon Like Map} and Section
\ref{subsec:Rescaling-levels} can be found in \cite{de2005renormalization,lyubich2011renormalization}.

The proof for the nonexistence of wandering domains is motivated by
the proof of the degenerate case. A Hénon-like map is degenerate means
that $\epsilon=0$ in (\ref{eq:Henon-like map}). In this case, the
dynamics of the map degenerates to the unimodal dynamics. In Chapter
\ref{sec:degenerate case}, a short proof for the nonexistence of
wandering intervals\index{wandering interval} for infinitely renormalizable
unimodal maps\index{unimodal map} is presented by identifying a unimodal
map as a degenerate Hénon-like map. The proof assumes the contrapositive,
there exists a wandering interval $J$. Then we apply the Hénon renormalization
instead of the standard unimodal renormalization to study the dynamics
of the rescaled orbit of $J$ that closest approaches the critical
value. The rescaled orbit is called the $J$-closest approach (Definition
\ref{def:Wandering domain}). The proof argues that the length of
the elements in the rescaled orbit approaches infinity by a length
expansion argument which leads to a contradiction. The expansion argument
motivates the proof for the Hénon case.

The proof of the main theorem is covered by Chapters \ref{sec:Framework},
\ref{sec:Good and Bad region}, \ref{sec:Good region}, and \ref{sec:Bad region}.
The structure is explained as follows. 

Assume the contrapositive, a Hénon-like map has a wandering domain
$J$. In Chapter \ref{sec:Framework}, we study the rescaled orbit
$\left\{ J_{n}\right\} _{n\geq0}$ of $J$ that closest approaches
\index{closest approach} to the tip\index{tip}, called the $J$-closest
approach. Each element $J_{n}$ belongs to some appropriate renormalization
scale (the domain of the $r(n)$-th renormalization $R^{r(n)}F$ for
some nonnegative integer $r(n)$). The transition between two constitutive
sequence elements $J_{n}\rightarrow J_{n+1}$ is called one step.
Motivated by the expansion argument from the degenerate case, we estimate
the change rate of the horizontal size\index{horizontal size} $l_{n}$
in each step. The horizontal size of a set is the size of its projection
to the first coordinate (Definition \ref{def:l,h}). Our final goal
is to show that the horizontal size of the sequence elements approaches
infinity to obtain a contradiction.

In the degenerate case, the expansion argument says that the horizontal
size expands at a uniform rate and hence the horizontal size of the
sequence elements approaches infinity. Unfortunately, the argument
breaks down in the non-degenerate case. There are two features that
make the non-degenerate case special: 
\begin{enumerate}
\item The good region and the bad region.\index{good region}\index{bad region}
\item Thickness\index{thickness}. 
\end{enumerate}
The good region and the bad region, introduced in Chapter \ref{sec:Good and Bad region},
divide the phase space of a Hénon-like map into two regions by how
similar the Hénon-like map behaves like unimodal maps. Each renormalization
scale (domain of the $n$-th renormalization $R^{n}F$ for some $n$)
has its own good region and bad region, and the size of the bad regions
contract super-exponentially as the renormalization applies to the
map more times (\cite[Theorem 4.1]{de2005renormalization} and Definition
\ref{def:Good and bad region}). When the elements in a closest approach
stay in the good regions of some appropriate scale, we show that the
expansion argument can be generalized to the Hénon-like maps. Thus,
the horizontal size expands at a uniform rate (Proposition \ref{prop:Good region estimates}).
However, when an element $J_{n}$ enters the bad region of the renormalization
scale of the set, the expansion argument breaks down. At this moment,
another quantity, called the thickness, offers a way to estimate the
horizontal size of the next element $J_{n+1}$ (Definition \ref{def:w}).
The reader should imagine the thickness of a set is the same as its
area. We will show that the thickness has a uniform contraction rate
proportion to the Jacobian of the map (Proposition \ref{prop:w_n Bound}).
For a strongly dissipative Hénon-like map, the Jacobian is small and
hence the contraction is strong. This strong contraction yields the
main obstruction toward our final goal.

The breakthrough is the discovery that the elements in a closest approach
can at most enter the bad regions finitely many times (Proposition
\ref{prop:Finite number of chain}). When an element $J_{n}$ enters
the bad region, the horizontal size contracts and the following element
$J_{n+1}$ belongs to a deeper renormalization scale. But the size
of the bad region in the deeper scale (scale of $J_{n+1}$) is much
smaller than the bad region of the original scale (scale of $J_{n}$).
Roughly speaking, we found that the contraction of the size of the
bad region is faster than the contraction of the horizontal size so
that the elements cannot enter the bad region infinitely times. The
actual proof is more delicate because another quantity, the time span
in the good regions\index{time span in the good regions} (Definition
\ref{def:J double sequence}), also involves in the competition. The
two-row lemma (Lemma \ref{lem:n recurrent relation}) is the key lemma
that gives an estimate for the competitions between the contraction
of the thickness, the expansion of the horizontal size in the good
region, the time span in the good region, and the size of the bad
region when the closest approach enters the bad region twice.  The
conclusion follows after applying the two-row lemma inductively (Lemma
\ref{lem:n lower bound}). 

In summary, the horizontal size of the elements in a closest approach
expands in the good regions, while contracts in the bad regions. However,
the contraction happens only finitely many times. This shows that
the horizontal size approaches infinity which is a contradiction.
Therefore, wandering domains cannot exist.

\section{\label{sec:Notations}Notations}

Common terminologies from Dynamical Systems will be adopted in this
article. The reader can refer to standard textbook (e.g. \cite{brin2002introduction})
for more information.

Let $I$ be an interval on the real line. The complex $\epsilon$-neighborhood
\nomenclature[I]{$I(\epsilon)$}{Complex $\epsilon$-neighborhood of the interval $I$}$I(\epsilon)$
of $I$ is defined to be the open set $I(\epsilon)=\left\{ z\in\mathbb{C};\left|z-z'\right|<\epsilon\text{ for some }z'\in I\right\} $.
The length\index{length|textbf} of $I$ is defined to be $\left|I\right|=\sup\left\{ \left|b-a\right|;a,b\in I\right\} $.

Assume that $X\subset\mathbb{R}^{2}$ is open and $F:X\rightarrow\mathbb{R}^{2}$
is differentiable. The Jacobian\index{Jacobian} of $F$ is the function
$\det DF$. 

The projections\index{projection} $\pi_{x}$ and $\pi_{y}$ are the
maps $\pi_{x}(x,y)=x$ and $\pi_{y}(x,y)=y$.

\global\long\def\interior#1{\text{int}(#1)}

\subsection{Functions}

Assume that $S$ is a set and $f$ is a real- or complex-valued function
on $S$. The sup norm\index{sup norm|textbf} on $S'\subset S$ is
denoted by 
\[
\left\Vert f\right\Vert _{S'}=\sup\left\{ \left|f(x)\right|;x\in S'\right\} .
\]
The subscript is neglected whenever the context is clear.

Assume that $V$ and $W$ are Banach spaces. A function $f:V\rightarrow W$
is called Lipschitz continuous\index{Lipschitz continuous|textbf}
with constant $L$, or $L$-Lipschitz, if 
\[
\left|f(y)-f(x)\right|\leq L\left|y-x\right|
\]
for all $x,y\in V$. The space of $L$-Lipschitz functions equipped
with the sup norm is complete. The space of $C^{n}(I)$ functions
on a closed interval $I$ is the collection of functions $f:I\rightarrow\mathbb{R}$
that are $n$-times differentiable with continuous derivatives. 

For a holomorphic function, the size of its derivatives can be estimate
by the sup norm of the map from the Cauchy integration formula.
\begin{lemma}
\label{lem:Derivative Bound}Assume that $U$ is open in $\mathbb{C}$
and $K\Subset U$. For each integer $n\geq1$, there exists a constant
$c>0$ such that 
\[
\left\Vert f^{(n)}\right\Vert _{K}\leq c\left\Vert f\right\Vert _{U}.
\]
for all holomorphic maps $f$ on $U$ where $f^{(n)}$ means the $n$-th
derivative of $f$. A similar estimate holds for multi-variable holomorphic
maps and partial derivatives.
\end{lemma}
This estimate will be frequently used in the setting where $K$ is
a real interval $I$ and $U$ is the complex $\delta$-neighborhood
$I(\delta)$.

\subsection{Schwarzian derivative}

In this section, we recall the definition and the properties of Schwarzian
derivative. The proof for the properties stated in this section can
be found in \cite{de2012one}. These properties will be used only
in Chapter \ref{sec:Unimodal-Maps}.
\begin{definition}[Schwarzian Derivative]
\nomenclature[S]{$S$}{Schwarzian derivative}Assume that $f$ is
a $C^{3}$ real valued function on an interval. The Schwarzian derivative\index{Schwarzian derivative|textbf}
of $f$ is defined by 
\[
(Sf)(x)=\left(\frac{f''(x)}{f'(x)}\right)'-\frac{1}{2}\left(\frac{f''(x)}{f'(x)}\right)^{2}=\frac{f'''(x)}{f'(x)}-\frac{3}{2}\left(\frac{f''(x)}{f'(x)}\right)^{2}
\]
whenever $f'(x)\neq0$. The map $f$ is said to have \index{Schwarzian derivative!negative|textbf}negative
Schwarizan derivative if $Sf(x)<0$ for all $x\in I$ with $f'(x)\neq0$.
\end{definition}
Negative Schwarzian derivative is preserved under iteration.
\begin{proposition}
If $f$ has negative Schwarzian derivative, then $f^{n}$ also has
negative Schwarzian derivative for all $n>0$.
\end{proposition}

\begin{proposition}[Minimal Principle]
\label{prop:minimal principle}Assume that $J$ is a bounded closed
interval and $f:J\rightarrow\mathbb{R}$ is a $C^{3}$ map with negative
Schwarzian derivative. If $f'(x)\neq0$ for all $x\in J$, then $\left|f'(x)\right|$
does not attain a local minimum in the interior of $J$.
\end{proposition}

\section{\label{sec:Unimodal-Maps}Unimodal Maps}

In this chapter, we give a short review over the procedure for unimodal
renormalization. The goal is to introduce the hyperbolic fixed point
for the renormalization operator (Proposition \ref{prop:Functional Equation})
and establish the estimations for its derivative (Subsection \ref{subsec:Expansion for the Feigenbaum map}).
\begin{definition}[Unimodal Map]
Let $I=[-1,1]$. A unimodal map\index{unimodal map|textbf} in this
paper is a smooth map $f:I\rightarrow I$ such that 
\begin{enumerate}
\item the point $-1$ is the unique fixed point with a positive multiplier,
\item $f(1)=-1$, and
\item the map $f$ has a unique maximum at $c\in\interior I$ and the point
$c$ is a non-degenerate critical point\index{critical point}, i.e.
$f'(c)=0$ and $f''(c)\neq0$. 
\end{enumerate}
The class of analytic unimodal maps $f:I\rightarrow I$ is denoted
as $\mathcal{U}$. \nomenclature[U]{$\mathcal{U}$}{Class of unimodal maps}
\end{definition}

\begin{definition}[Critical Orbit]
\nomenclature[c^(n)]{$c^{(n)}$}{Critcal point and its orbit}For
a unimodal map $f\in\mathcal{U}$, let $c^{(0)}=c^{(0)}(f)\in I$
be the critical point\index{critical point!orbit|textbf} of $f$.
The critical orbit is denoted as $c^{(n)}=f^{n}(c^{(0)})$ for all
$n>0$.
\end{definition}
\begin{definition}[Reflection]
Assume that $f\in\mathcal{U}$ and $x\in I$. If $x\neq c^{(0)}$,
define the reflection of $x$ to be the point \nomenclature[hat]{$\hat{\cdot}$}{Reflection point}$\hat{x}\in I$
such that $f(\hat{x})=f(x)$ and $\hat{x}\neq x$. If $x=c^{(0)}$,
define $\hat{x}=c^{(0)}$.
\end{definition}

\subsection{The renormalization of a unimodal map}

To define the period-doubling renormalization operator for unimodal
maps, we introduce a partition on $I$ that allows us to define the
first return map for a renormalizable unimodal map.
\begin{definition}
\label{def:Partition of I}\nomenclature[A, B, C]{$A$, $B$, $C$}{Partition of the domain $D$ for unimodal map}Assume
that $f\in\mathcal{U}$ has a unique fixed point $p(0)\in I$ with
a negative multiplier. Let $p^{(1)}=\hat{p}(0)$ and $p^{(2)}$ be
the point such that $f(p^{(2)})=p^{(1)}$ and $p^{(2)}>c^{(0)}$.
Define $A=(-1,p^{(1)})\cup(p^{(2)},1)$, $B=(p^{(1)},p(0))$, and
$C=(p(0),p^{(2)})$. The sets $A=A(f)$, $B=B(f)$, and $C=C(f)$
form a partition\index{partition!unimodal map|textbf} of the domain
$D\equiv I$. See Figure \ref{fig:Partition for unimodal map} for
an illustration.

\begin{figure}
\begin{centering}
\includegraphics[scale=0.5]{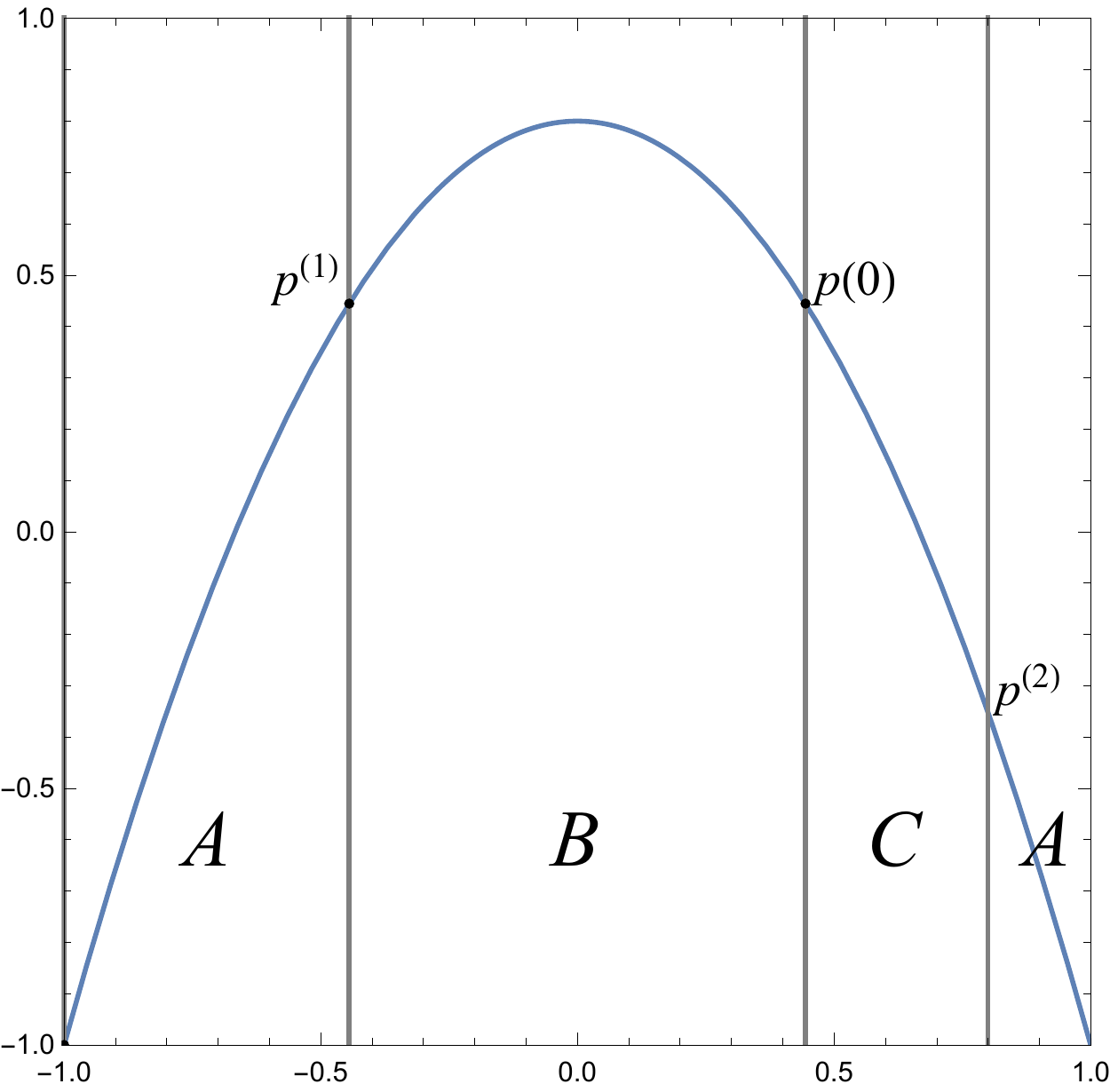}
\par\end{centering}
\caption{\label{fig:Partition for unimodal map}The partition $\left\{ A,B,C\right\} $
of a unimodal map. The parabola is the graph of a unimodal map. The
points $p(0)$, $p^{(1)}$, and $p^{(2)}$ are defined as in Definition
\ref{def:Partition of I}.}
\end{figure}
\end{definition}
The property ``renormalizable'' is defined by using the partition
elements.
\begin{definition}[Renormalizable]
A unimodal map $f\in\mathcal{U}$ is (period-doubling) renormalizable\index{renormalizable!unimodal map|textbf}
if it has a fixed point $p(0)$ with a negative multiplier and $f(B)\subset C$.
The class of renormalizable unimodal maps is denoted as \nomenclature[U^r]{$\mathcal{U}^r$}{Class of renormalizable unimodal maps}$\mathcal{U}^{r}$.
\end{definition}
\begin{remark}
Most of the articles define the unimodal renormalization by using
the critical orbit. However, here we choose to use an orbit that maps
to the fixed point with a negative multiplier instead. The purpose
of doing this is to make the partition consistent with the partition
defined for Hénon-like maps (Definition \ref{def:Domain A,B,C}) because
Hénon-like maps do not have a critical point. 
\end{remark}
For a renormalizable unimodal map, an orbit that is not eventually
periodic follows the paths in the following diagram. 
\[
\xymatrix{A\ar[r]\ar@(ul,dl) & B\ar@<.5ex>[r] & C\ar@<.5ex>[l]}
\]
  This allows us to define the first return map on $B$ and the
period-doubling renormalization.

\global\long\def\uRenormalizeC{R}
\global\long\def\uRescaleC{s}
\begin{definition}[Renormalization]
Assume that $f\in\mathcal{U}^{r}$. \nomenclature[R_c]{$R_{c}$}{Renormalization operator about the critical point}\nomenclature[s_c]{$s_{c}$}{Affine rescaling about the critical point}The
renormalization \index{renormalization operator!unimodal map} of
$f$ is the map $\uRenormalizeC f=\uRescaleC\circ f^{2}\circ\uRescaleC^{-1}$
where $\uRescaleC$ is the orientation-reversing affine rescaling
such that $\uRescaleC(p(0))=-1$ and $\uRescaleC(p^{(1)})=1$.

The renormalization operator is a map $\uRenormalizeC:\mathcal{U}^{r}\rightarrow\mathcal{U}$.
If the procedure of renormalization can be done recurrently infinitely
many times, then the map is called infinitely (period-doubling) renormalizable.
The class of infinitely renormalizable unimodal maps is denoted as
\nomenclature[U^r]{$\mathcal{U}^r$}{Class of renormalizable unimodal maps}$\mathcal{I}$.
\end{definition}

\subsection{\label{subsec:Feigenbaum map}The fixed point of the renormalization
operator}

In this section, we study the fixed point $g$ of the renormalization
operator. The map $g$ is also important for the Hénon case because
it also defines the hyperbolic fixed point of the Hénon renormalization
operator \cite[Theorem 4.1]{de2005renormalization}.

The existence and uniqueness of the fixed point was proved in \cite{epstein1981analyticity,campanino1982feigenbaum}.
Here, the theorem is restated in the coordinate system used in this
paper.
\begin{proposition}
\label{prop:Functional Equation}\index{renormalization!fixed point}There
exists a unique constant $\lambda=2.5029...$ and a unique solution
$g\in\mathcal{I}$ \nomenclature[g]{$g$}{Fixed point for $R_{c}$}
of the Cvitanovi\'{c}-Feigenbaum-Coullet-Tresser functional equation\index{Feigenbaum-Cvitanovic functional equation@Feigenbaum-Cvitanovi\'{c}
functional equation}
\begin{equation}
g(x)=-\lambda g^{2}\left(-\frac{x}{\lambda}\right)\label{eq:Functional Equation}
\end{equation}
for $-1\leq x\leq1$ with the following properties:
\begin{enumerate}
\item $g$ is analytic in a complex neighborhood of $[-1,1]$.
\item $g$ is even.
\item $g$ is concave on $[-c^{(1)},c^{(1)}]$.
\item $g(c^{(1)})=-\frac{1}{\lambda}c^{(1)}$ and $g'(c^{(1)})=-\lambda$.
\item $g$ has negative Schwarzian derivative.
\end{enumerate}
\end{proposition}
\begin{corollary}
The map $g$ satisfies the following property
\begin{equation}
g^{2^{n}}\left(\frac{1}{\left(-\lambda\right)^{n}}x\right)=\frac{1}{\left(-\lambda\right)^{n}}g\left(x\right)\label{eq:Functional equation general}
\end{equation}
for all $n\geq0$ and all $x\in I$.
\end{corollary}
\begin{proof}
The proof follows from the functional equation (\ref{eq:Functional Equation}).
\end{proof}

In the remaining part of the section, the notations for the unimodal
maps will be applied to the map $g$. For example, $\left\{ c^{(j)}=c^{(j)}(g)\right\} _{j\geq0}$
is the critical orbit and the sets $A=A(g)$, $B=B(g)$, and $C=C(g)$
form a partition of the domain $D=I$.

\subsubsection{\index{critical point!orbit of renormalization fixed point}A backward orbit of the critical point}

In this section, we establish a backward orbit $b^{(2)}\rightarrow b^{(1)}\rightarrow c^{(0)}$
of the critical point $c^{(0)}$. Let $b^{(1)}\in[0,c^{(1)}]$ be the point such that $g(b^{(1)})=0$.
Set $b^{(2)}=\frac{1}{\lambda}b^{(1)}$.
\begin{lemma}
We have 
\[
g\left(b^{(2)}\right)=b^{(1)}.
\]
\end{lemma}
\begin{proof}
Since $g$ is even, the only two roots of $g$ are $-b^{(1)}$ and
$b^{(1)}$. By the functional equation (\ref{eq:Functional equation general}),
we have
\[
g^{2}\left(b^{(2)}\right)=-\frac{1}{\lambda}g\left(-b^{(1)}\right)=0.
\]
Thus, $g\left(b^{(2)}\right)=-b^{(1)}$ or $b^{(1)}$. Also, $g\left(b^{(2)}\right)\neq-b^{(1)}$
because $b^{(2)}\in(0,b^{(1)})$ and $g\left(x\right)>0$ on $(0,b^{(1)})$.
Therefore, $g\left(b^{(2)}\right)=b^{(1)}$.
\end{proof}

\subsubsection{\label{subsec:Expansion for the Feigenbaum map}Estimations for the
derivative}

Apply the chain rule to the functional equation (\ref{eq:Functional Equation}),
we have
\begin{equation}
g'\left(x\right)=g'\left(-\frac{x}{\lambda}\right)g'\circ g\left(-\frac{x}{\lambda}\right)\label{eq:Derivative of Functional Equation}
\end{equation}
for $x\in I$. We will use this formula to derive the values for the
derivative of $g$ at some particular values.

\begin{lemma}
The slope at $b^{(2)}$ is 
\begin{equation}
g'(b^{(2)})=-1.\label{eq:Slope 1}
\end{equation}
\end{lemma}
\begin{proof}
From (\ref{eq:Derivative of Functional Equation}) and $g$ is even,
we have
\[
g'(b^{(1)})=g'\left(-b^{(2)}\right)g'\circ g\left(-b^{(2)}\right)=-g'\left(b^{(2)}\right)g'\left(b^{(1)}\right).
\]
We solve $g'(b^{(2)})=-1$.
\end{proof}

Let $q(-1)=-1$ (the fixed point with a positive multiplier) and $q(0)$
be the fixed point with a negative multiplier. From the functional
equation (\ref{eq:Functional Equation}), we get $q(0)=\frac{1}{\lambda}$.

\begin{lemma}
The slopes at the fixed points satisfy the relation
\begin{equation}
g'\left(q(-1)\right)=\left[g'\left(q(0)\right)\right]^{2}.\label{eq:slope for fixed point}
\end{equation}
\end{lemma}
\begin{proof}
From (\ref{eq:Derivative of Functional Equation}), compute
\[
g'\left(q(-1)\right)=g'\left(-1\right)=g'\left(\frac{1}{\lambda}\right)g'\circ g\left(\frac{1}{\lambda}\right)=g'\left(q(0)\right)g'\circ g\left(q(0)\right)=\left[g'\left(q(0)\right)\right]^{2}.
\]
\end{proof}

Finally, we prove that the map $g$ is expanding on $A$ and $C$.
\begin{proposition}
\label{prop:Derivative of g on A}The slope of $g$ is bounded below
by
\[
\left|g'(x)\right|\geq\left|g'(q(0))\right|>1
\]
for all $x\in\left[q(-1),\hat{q}(0)\right]\cup\left[q(0),\hat{q}(-1)\right]$.
\end{proposition}
\begin{proof}
It is enough to prove the case when $x\in\left[q(0),\hat{q}(-1)\right]$
since $g$ is even.

First, we consider the interval $[b^{(2)},c^{(1)}]$. We have $b^{(2)}<q(0)<c^{(1)}$.
By (\ref{eq:Slope 1}) and Proposition \ref{prop:Functional Equation},
the derivatives of the boundaries are $g'(b^{(2)})=-1$ and $g'(c^{(1)})=-\lambda$.
We get $\left|g'(q(0))\right|>1$ by the minimal principle (Proposition
\ref{prop:minimal principle}). 

Next, we consider the interval $[q(0),\hat{q}(-1)]$. From (\ref{eq:slope for fixed point}),
we also get $\left|g'(\hat{q}(-1))\right|>1$. Therefore, the proposition
follows from the minimal principle (Proposition \ref{prop:minimal principle}).

\end{proof}

\section{\label{sec:Renormalization of Henon Like Map}Hénon-like Maps }

In this chapter, we give an introduction to the theory of Hénon renormalization
in the strongly dissipative regime developed by \cite{de2005renormalization,lyubich2011renormalization}.
Their theorems are adopted to fit the notations and the coordinate
system used in this article.

\subsection{The class of unimodal maps}
\begin{definition}[Class of unimodal maps]
Assume that $\delta>0$, $\kappa>0$, and $I^{h}\Supset I\equiv[-1,1]$.
Let \nomenclature[U_delta]{$\mathcal{U}_{\delta}$}{Class of unimodal maps with holomorphic extension on a $\delta$-neighborhood}$\mathcal{U}_{\delta,\kappa}(I^{h})\subset\mathcal{U}$
be the class of analytic unimodal maps\index{unimodal map} $f:I^{h}\rightarrow I^{h}$
such that
\begin{enumerate}
\item $f$ has a unique critical point $c$ such that $c\leq f(c)-\kappa$
and $f(c)\leq1-\kappa$,
\item $f$ has two fixed points $-1$ and $p$ such that $-1$ has an expanding
positive multiplier and $p$ has a negative multiplier,
\item $f$ has holomorphic extension to $I^{h}(\delta)$, 
\item $f$ can be factorized as $f=Q\circ\phi$ where $Q(x)=c^{(1)}-(c^{(1)}+1)x^{2}$,
$c^{(1)}$ is the critical value, and $\phi$ is an $\mathbb{R}$-symmetric
univalent map on $I^{h}(\delta)$, and 
\item $f$ has negative Schwarzian derivative.
\end{enumerate}
In the remaining article, we fix a small $\kappa>0$ such that the
class contains the renormalization fixed point $g$, and we suppress
the subscript from the notation $\mathcal{U}_{\delta}(I^{h})=\mathcal{U}_{\delta,\kappa}(I^{h})$.
\end{definition}
\begin{remark}
From the conditions $f(-1)=-1$ and $f(1)=-1$, this forces $\phi(-1)=-1$
and $\phi(1)=1$. Thus, $\mathcal{U}_{\delta}$ forms a normal family
by \cite[Theorem 3.2]{milnor2011dynamics}.
\end{remark}

\subsection{The class of Hénon-like maps}
\begin{definition}[Hénon-like map]
\nomenclature[F]{$F$}{Hénon-like map}\nomenclature[f]{$f$}{Unimodal component for Hénon-like map}\nomenclature[epsilon]{$\epsilon$}{Perturbation component for Hénon-like map}\nomenclature[h]{$h$}{$x$-component for Hénon-like map}Assume
that $I^{v}\supset I^{h}\Supset I$ are closed intervals. A Hénon-like
map\index{Hénon-like map|textbf} is a smooth map $F:I^{h}\times I^{v}\rightarrow\mathbb{R}^{2}$
of the form 
\[
F(x,y)=(f(x)-\epsilon(x,y),x)
\]
where $f$ is a unimodal map and $\epsilon$ is a small perturbation.
The function $h$ will also be used to express the $x$-component,
$h_{y}(x)=h(x,y)=\pi_{x}F(x,y)$. A representation of $F$ will be
expressed in the form $F=(f-\epsilon,x)$.
\end{definition}
The function spaces of the Hénon-like map is defined as follows.
\begin{definition}[Class of Hénon-like maps]
Assume that $I^{v}\supset I^{h}\Supset I$ and $\delta>0$. \nomenclature[I^v]{$I^{v}$}{Vertical domain for a Hénon-like map}\nomenclature[I^h]{$I^{h}$}{Horizontal domain for a Hénon-like map}\nomenclature[H]{$\mathcal{H}$}{Class of Hénon-like maps}
\begin{enumerate}
\item Denote $\mathcal{H}_{\delta}(I^{h}\times I^{v})$ to be the class
of real analytic Hénon-like maps $F:I^{h}\times I^{v}\rightarrow\mathbb{R}^{2}$
that have the following properties:
\begin{enumerate}
\item It has a representation $F=(f-\epsilon,x)$ such that $f\in\mathcal{U}_{\delta}(I^{h})$.
\item It has a saddle fixed point $p(-1)$ near the point $(-1,-1)$. The
fixed point has an expanding positive multiplier.
\item The $x$-component $h(x,y)$ has a holomorphic extension to $I^{h}(\delta)\times I^{v}(\delta)\rightarrow\mathbb{C}$.
\end{enumerate}
\item Given $\overline{\epsilon}>0$ and $f\in\mathcal{U}_{\delta}(I^{h})$.
Denote $\mathcal{H}_{\delta}(I^{h}\times I^{v},f,\overline{\epsilon})$
to be the class of Hénon-like maps $F\in\mathcal{H}_{\delta}(I^{h}\times I^{v})$
with the form $F=(f-\epsilon,x)$ such that $\left\Vert \epsilon\right\Vert <\overline{\epsilon}$. 
\item Denote $\mathcal{H}_{\delta}(I^{h}\times I^{v},\overline{\epsilon})=\cup\mathcal{H}_{\delta}(I^{h}\times I^{v},f,\overline{\epsilon})$
where the union is taken over all $f\in\mathcal{U}_{\delta}(I^{h})$.
\end{enumerate}
\end{definition}
\begin{remark}
\label{rem:Comparison between the domain}The domain $I^{h}\times I^{v}$
used in this article is slightly larger than the domain studied in
the two original papers \cite{de2005renormalization,lyubich2011renormalization}.
Their domain is equivalent to the dynamical interval $[f^{2}(c),f(c)]$
for unimodal maps which does not include the fixed point with positive
multiplier. The larger domain is necessary in this article to study
the rescaled orbit of a point. See Proposition \ref{prop:Dynamics on D},
Proposition \ref{prop:Dynamics of the partition on D}, and Proposition
\ref{prop:Dynamics on the partition}.

Their work also holds on the larger domain $I^{h}\times I^{v}$. See
for examples \cite[Footnote 7, Section 3.4]{de2005renormalization}
and \cite[Lemma 3.3, Proposition 3.5, Theorem 4.1]{lyubich2011renormalization}.
However, reproving their theorem on the larger domain is not the aim
here. This article will assume the results from \cite{de2005renormalization,lyubich2011renormalization}
also hold in the larger domain and rephrase them in the notations
used in this article without reproving. See also Remarks \ref{rem:Comparsion between renormalizable},
\ref{rem:Comparsion between infinite renormalizable}, and \ref{rem:Absence of wandering domain for CLM-renormalizable}.
\end{remark}
From the definition, it follows immediately that 
\begin{lemma}
\label{lem:H subset}Given $I^{v}\supset I^{h}\Supset I$, $\delta>0$,
$\overline{\epsilon}>0$, and $f\in\mathcal{U}_{\delta}(I^{h})$. 
\begin{enumerate}
\item If $\overline{\epsilon_{1}}<\overline{\epsilon_{2}}$ then $\mathcal{H}_{\delta}(f,\overline{\epsilon_{1}})\subset\mathcal{H}_{\delta}(f,\overline{\epsilon_{2}})$.
\item If $I\subset I_{1}^{h}\subset I_{2}^{h}\subset I^{v}$ and $f\in\mathcal{U}_{\delta}(I_{2}^{h})$,
then $\mathcal{U}_{\delta}(I_{1}^{h})\supset\mathcal{U}_{\delta}(I_{2}^{h})$
and $\mathcal{H}_{\delta}(I_{1}^{h}\times I^{v},f,\overline{\epsilon})\supset\mathcal{H}_{\delta}(I_{2}^{h}\times I^{v},f,\overline{\epsilon})$.
\end{enumerate}
\end{lemma}
An important property of a Hénon-like map is that it maps vertical
lines to horizontal lines; it maps horizontal lines to parabola-like
arcs.
\begin{example}[Degenerate case]
Assume that $F(x,y)=(f(x)-\epsilon(x,y),x)$ is a Hénon-like map.
The map is called a degenerate\index{Hénon-like map!degenerate|textbf}
Hénon-like map if $\frac{\partial\pi_{x}F}{\partial y}=\frac{\partial\epsilon}{\partial y}=0$;
a non-degenerate Hénon-like map if $\frac{\partial\pi_{x}F}{\partial y}=\frac{\partial\epsilon}{\partial y}\neq0$. 

If $F$ is degenerate, then $\epsilon$ only depends on $x$. In this
case, without lose of generality, we will assume the Hénon-like map
has the representation $F(x,y)=(f(x),x)$ where $f=\pi_{x}F$ and
$\epsilon=0$.

For the degenerate case, the dynamics of the Hénon-like map is completely
determined by its unimodal component. So it will also be called as
the unimodal case\index{unimodal case|see{degenerate Hénon-like map}}
in this article.

The degenerate case is an important example in this article. A proof
for the nonexistence of wandering intervals for unimodal maps will
be presented in Chapter \ref{sec:degenerate case} by identifying
a unimodal map as a degenerate Hénon-like map. The expansion argument
in the proof motivates the proof for the non-degenerate case. The
difference between the degenerate case and the non-degenerate case
produces the main difficulty (explained in Chapter \ref{sec:Good and Bad region}
and Chapter \ref{sec:Bad region}) of extending the proof to the non-degenerate
case.
\end{example}
\begin{example}[Classical Hénon maps]
The classical Hénon family is a two-parameter family of the form
$F_{a,b}(x,y)=(-1+a(1-x^{2})-by,x)$ where $a,b>0$. These are Hénon-like
maps $F_{a,b}\in\mathcal{H}_{\delta}(I^{h}\times I^{v},-1+a(1-x^{2}),b[\left|I^{v}\right|+2\delta])$
for all $\delta>0$ and $I^{v}\supset I^{h}$.
\end{example}

\subsection{\label{subsec:local stable manifolds=000026partition}Local stable
manifolds\index{local stable manifold|(} and partition\index{partition!Hénon-like map|seealso{local stable manifold}}
of a Hénon-like map}

To study the dynamics of a Hénon-like map, we need to find a domain
$D\subset I^{h}\times I^{v}$ that turns the Hénon-like map into a
self-map\index{self-map}. Also, to renormalize a Hénon-like map,
we need to find a subdomain $C\subset D$ that defines a first return
map. Motivated from unimodal maps, one can construct a partition of
the domain $I^{h}\times I^{v}$ to find the domains. In the unimodal
case, an orbit that maps to the fixed point $p(0)$ with an expanding
multiplier splits the domain $D$ into a partition $\left\{ A,B,C\right\} $
(Definition \ref{def:Partition of I}). For a strongly dissipative
Hénon-like map, the orbit becomes components of the stable manifold
of the saddle fixed point $p(0)$. These components are vertical graphs
that split the domain into multiple vertical strips. 
\begin{definition}
A set $\Gamma$ is a vertical graph\index{vertical graph|textbf}
if there exists a continuous function $\gamma:I^{v}\rightarrow I^{h}$
such that $\Gamma=\left\{ (\gamma(t),t);t\in I^{v}\right\} $. The
vertical graph $\Gamma$ is said to have Lipschitz constant $L$ if
the function $\gamma$ is Lipschitz with constant $L$.
\end{definition}
In this paper, a local stable manifold\index{local stable manifold}
is a connected component of a stable manifold. Inspired by \cite{de2005renormalization},
the partition will be the vertical strips separated by the associated
local stable manifolds.

First, we study the local stable manifolds of the saddle fixed point
$p(-1)$ which contains an expanding positive multiplier.

\begin{figure}
\begin{centering}
\includegraphics[scale=0.5]{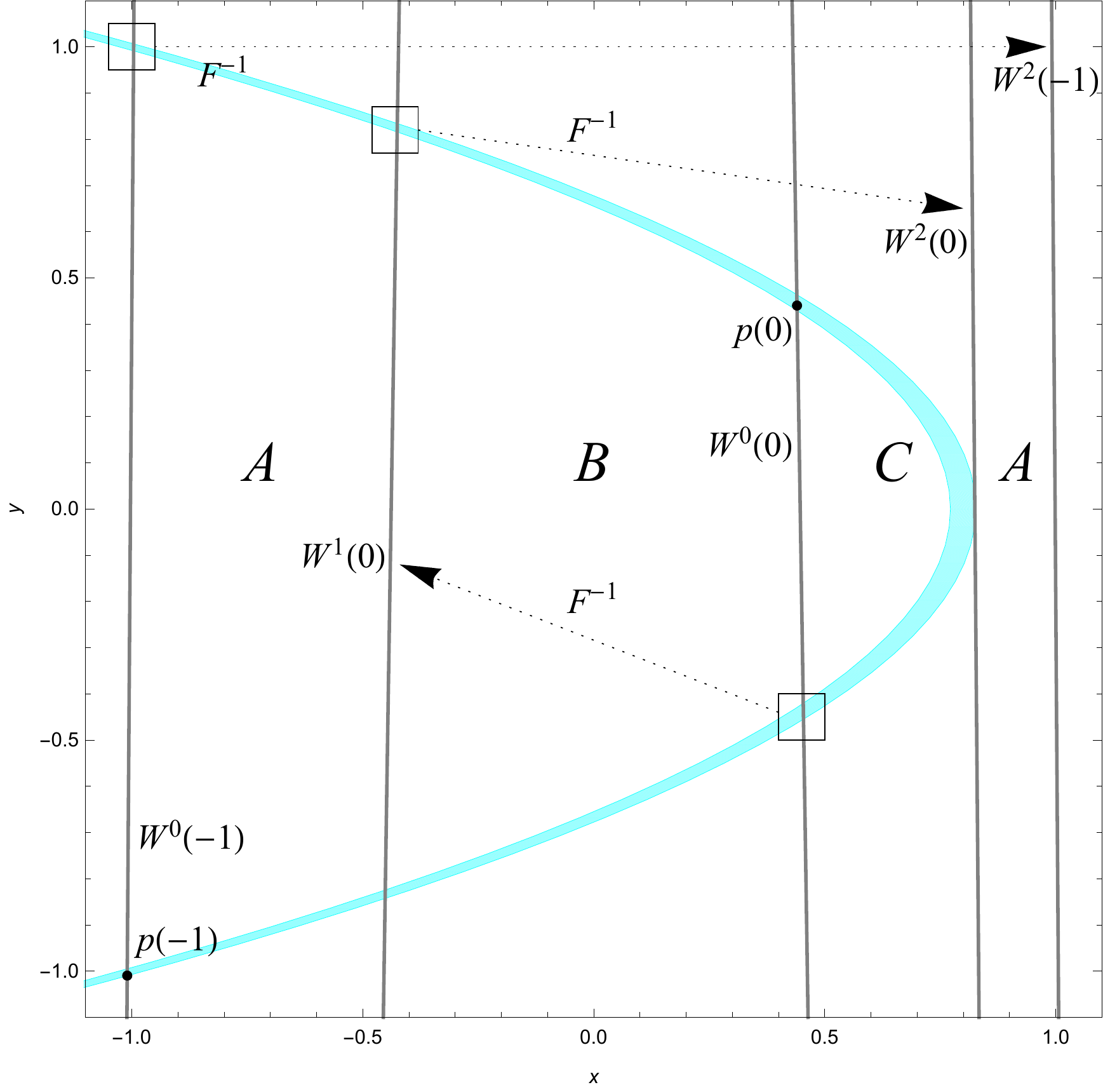}
\par\end{centering}
\caption{\label{fig:Partition for Henon-like map}Local stable manifolds and
partition $A$, $B$, $C$ for a Hénon-like map $F$. The shaded area
is the image of the Hénon-like map. The vertical graphs are the local
stable manifolds $W^{0}(-1)$, $W^{1}(0)$, $W^{0}(0)$, $W^{2}(0)$,
and $W^{2}(-1)$ from left to right. The arrows illustrates the construction
of each local stable manifold.}
\end{figure}
\begin{definition}[The local stable manifolds of $p(-1)$ and the iteration domain $D$]
\label{def:Local stable manifold W(-1)}\nomenclature[W(0)]{$W^{t}(-1)$}{Local stable manifolds of $p(-1)$}\index{local stable manifold|textbf}Given
$I^{v}\supset I^{h}\Supset I$, $\delta>0$, and $F\in\mathcal{H}_{\delta}(I^{h}\times I^{v})$.
Consider the stable manifold of the saddle fixed point $p(-1)$.
\begin{enumerate}
\item If the connected component that contains the fixed point $p(-1)$
is a vertical graph, let $W^{0}(-1)$ be the component.
\item Assume that $W^{0}(-1)$ exists. If $F^{-1}(W^{0}(-1))$ has two components,
one is $W^{0}(-1)$ and the other is a vertical graph. Let $W^{2}(-1)$
be the one that is disjoint from $W^{0}(-1)$.
\end{enumerate}
If the the local stable manifolds $W^{0}(-1)$ and $W^{2}(-1)$ exists,
define $D=D(F)\subset I^{h}\times I^{v}$ to be the open set bounded
between the two local stable manifolds. \nomenclature[D]{$D$}{The Hénon-like map is defined to be a self-map on $D\subset I^{h}\times I^{v}$}\index{partition!Hénon-like map|textbf}
See Figure \ref{fig:Partition for Henon-like map} for an illustration.
\end{definition}
The domain $D$ turns the Hénon-like map into a self-map.
\begin{proposition}
\label{prop:Dynamics on D}Given $\delta>0$ and intervals $I^{h}$
and $I^{v}$ with $I^{v}\supset I^{h}\Supset I$. There exists $\overline{\epsilon}>0$
and $c>0$ such that for all $F\in\mathcal{H}_{\delta}(I^{h}\times I^{v},\overline{\epsilon})$
the following properties hold: 
\begin{enumerate}
\item The sets $W^{0}(-1)$, $W^{2}(-1)$, and $D$ exist. The two local
stable manifolds are vertical graphs with Lipschitz constant $c\left\Vert \epsilon\right\Vert $.
\item $F(D)\subset D$.\index{self-map}
\end{enumerate}
\end{proposition}
\begin{proof}
The first property follows from the graph transformation. The techniques
were developed in \cite[Chapter 3]{lyubich2011renormalization}. See
\cite[Lemma 3.1, 3.2]{lyubich2011renormalization}.

The second property follows from the definition of the local stable
manifolds and $\overline{\epsilon}>0$ is sufficiently small.
\end{proof}

Next, we study the local stable manifolds of the other saddle fixed
point $p(0)$ with an expanding negative multiplier to define a partition
of $D$.
\begin{definition}[The local stable manifolds of $p(0)$]
\label{def:Local stable manifold W(0)}\nomenclature[W(0)]{$W^{t}(0)$}{Local stable manifolds of $p(0)$}\index{local stable manifold|textbf}Given
$I^{v}\supset I^{h}\Supset I$, $\delta>0$, and $F\in\mathcal{H}_{\delta}(I^{h}\times I^{v})$.
Assume that $F$ has a saddle fixed point $p(0)$ with an expanding
negative multiplier. Consider the stable manifold of $p(0)$.
\begin{enumerate}
\item If the connected component that contains $p(0)$ is a vertical graph,
let $W^{0}(0)$ be the component.
\item Assume that $W^{0}(0)$ exists. If $F^{-1}(W^{0}(0))$ has two components,
one is $W^{0}(0)$ and the other is a vertical graph. Let $W^{1}(0)$
be the one that is disjoint from $W^{0}(0)$.
\item Assume that $W^{0}(0)$ and $W^{1}(0)$ exist. If $F^{-1}(W^{1}(0))$
has two components and one component is a vertical graph located to
the right of $W^{0}(0)$. Let $W^{2}(0)$ be the component.
\end{enumerate}
\end{definition}
See Figure \ref{fig:Partition for Henon-like map} for an illustration.
\begin{remark}
At this moment, the numbers $0$ and $-1$ in the notation of the
fixed points $p(0)$ and $p(-1)$ (and also the local stable manifolds)
do not have a special meaning. After introducing infinitely renormalizable
Hénon-like maps, the notation $p(k)$ will be used to define a periodic
point with period $2^{k}$. See Definition \ref{def:rescaling levels}.
The numbers are introduced here for consistency.
\end{remark}
The local stable manifolds split the domain $D$ into vertical strips.
These strips define a partition of the domain.
\begin{definition}[$A$, $B$, and $C$]
\label{def:Domain A,B,C}\nomenclature[A, B, C]{$A$, $B$, $C$}{Partition of the domain $D$ for Hénon-like map}\index{partition!Hénon-like map|textbf}Given
$I^{v}\supset I^{h}\Supset I$, $\delta>0$, and $F\in\mathcal{H}_{\delta}(I^{h}\times I^{v})$.
Assume that $F$ has a saddle fixed point $p(0)$ with an expanding
negative multiplier, the local stable manifolds in Definition \ref{def:Local stable manifold W(0)}
exist, and $D$ exists. 
\begin{enumerate}
\item Define $A=A(F)\subset I^{h}\times I^{v}$ to be the union of two sets.
One is the open set bounded between $W^{0}(-1)$ and $W^{1}(0)$;
the other is the open set bounded between $W^{2}(0)$ and $W^{2}(-1)$.
\item Define $B=B(F)\subset I^{h}\times I^{v}$ to be the open set bounded
between $W^{0}(0)$ and $W^{1}(0)$. 
\item Define $C=C(F)\subset I^{h}\times I^{v}$ to be the open set bounded
between $W^{0}(0)$ and $W^{2}(0)$.
\end{enumerate}
\end{definition}
\begin{remark}
The local stable manifolds $W^{0}(-1)$, $W^{1}(0)$, $W^{0}(0)$
, $W^{2}(0)$, and $W^{2}(-1)$ are associated to the points $p(-1)=-1$,
$p^{(1)}$, $p(0)$, $p^{(2)}$, and $1$ respectively (Definition
\ref{def:Partition of I}).
\end{remark}
For a strongly dissipative Hénon-like map, the local stable manifolds
are vertical graphs and the dynamics on the partition is similar to
the unimodal case.
\begin{proposition}
\label{prop:Dynamics of the partition on D}\index{partition!Hénon-like map|textit}Given
$\delta>0$ and intervals $I^{h}$ and $I^{v}$ with $I^{v}\supset I^{h}\Supset I$.
There exists $\overline{\epsilon}>0$ and $c>0$ such that for all
$F\in\mathcal{H}_{\delta}(I^{h}\times I^{v},\overline{\epsilon})$
the following properties hold: 
\begin{enumerate}
\item The sets $W^{0}(0)$, $W^{1}(0)$, $W^{2}(0)$, $A$, $B$, and $C$
exist. The local stable manifolds are vertical graphs with Lipschitz
constant $c\left\Vert \epsilon\right\Vert $.
\item $F(A)\subset A\cup W^{1}(0)\cup B$.
\item $F(C)\subset B$.
\item If $z\in A$ then its orbit eventually escapes $A$, i.e. there exists
$n>0$ such that $F^{n}(z)\notin A$.
\end{enumerate}
\end{proposition}
\begin{proof}
The first property is proved by graph transformation. See \cite[Chapter 3]{lyubich2011renormalization}. 

The second and third properties follows from the definition of the
local stable manifolds. See also \cite[Lemma 4.2]{lyubich2011renormalization}.

The last property holds because the only fixed points are $p(-1)$
and $p(0)$ so the local unstable manifold of $p(-1)$ must extends
across the whole set $A$. See also \cite[Lemma 4.2]{lyubich2011renormalization}.
\end{proof}

By the definition of $B$, its iterate $F(B)$ is contained in the
right component of $D\backslash W^{0}(0)$. With the third property
of Proposition \ref{prop:Dynamics of the partition on D}, we can
define the condition ``renormalizable'' as follows.
\begin{definition}[Renormalizable]
Assume that $\overline{\epsilon}>0$ is sufficiently small. A Hénon-like
map $F\in\mathcal{H}_{\delta}(I^{h}\times I^{v},\overline{\epsilon})$
is (period-doubling) renormalizable\index{renormalizable!Hénon-like|textbf}
if it has a saddle fixed point $p(0)$ with an expanding negative
multiplier and $F(B)\subset C$. The class of renormalizable Hénon-like
maps is denoted by $\mathcal{H}_{\delta}^{r}(I^{h}\times I^{v},\overline{\epsilon})\subset\mathcal{H}_{\delta}(I^{h}\times I^{v},\overline{\epsilon})$.
\end{definition}
\begin{remark}
\label{rem:Comparsion between renormalizable}The notion of ``renormalizable''
here is similar to \cite[Section 3.4]{de2005renormalization} (which
they called pre-renormalization) but not exactly the same. The ``renormalizable''
in their paper is called CLM-renormalizable here to compare the difference.
In their article, the set ``$C$'' (they named the set $D$) where
they define the first return map is a region bounded between $W^{0}(0)$
and a section of the unstable manifold of $p(-1)$. In this article,
the set $C$ is defined to be the largest candidate (around the critical
value) that is invariant under $F^{2}$ which only uses the local
stable manifolds of $p(0)$. Thus, the sets $B$ and $C$ in this
article is slightly larger than theirs.

As a result, the property ``renormalizable'' in this article is
stronger than theirs. If a Hénon-like map is renormalizable then it
is also CLM-renormalizable. Although the converse is not true in general,
the hyperbolicity of the renormalization operator \cite[Theorem 4.1]{de2005renormalization}
allows us to apply the notion of renormalizable to an infinitely CLM-renormalizable
map. This makes the final result, Theorem \ref{thm:nonexistence of wandering domain},
also works for CLM-renormalizable maps. See Remarks \ref{rem:Comparsion between infinite renormalizable}
and \ref{rem:Absence of wandering domain for CLM-renormalizable}
for more details. 

Their definition has some advantages and disadvantages. Their notion
of renormalizable does not depend on the size of the vertical domain
$I^{v}$. However, their sets $B$ and $C$ are too small. It may
requires more iterations for an orbit to enter their $B$ and $C$.
See the proof of \cite[Lemma 4.2]{lyubich2011renormalization}. This
is the reason for adjusting their definition.
\end{remark}

For a renormalizable Hénon-like map, an orbit that is disjoint from
the stable manifold of the fixed points follows the paths in the following
diagram. 
\[
\xymatrix{A\ar[r]\ar@(ul,dl)_{\text{finite iterations}} & B\ar@<.5ex>[r] & C\ar@<.5ex>[l]}
\]
Therefore, a renormalizable map has a first return map on $C$.

\index{local stable manifold|)}

\subsection{Renormalization operator}

When a Hénon-like map is renormalizable, the map has a first return
map on $C$. However, the first return map is no longer a Hénon-like
map by a direct computation 
\[
F^{2}(x,y)=(h_{x}(h_{y}(x)),h_{y}(x)).
\]
The paper \cite{de2005renormalization} introduced a nonlinear coordinate
change $H(x,y)\equiv(h_{y}(x),y)$ \nomenclature[H]{$H$}{Nonlinear part of the Hénon rescaling}
that turns the first return map into a Hénon-like map. The next proposition
defines the renormalization operator.
\begin{proposition}[Renormalization operator]
\label{prop:Renormalization Operator}\index{renormalization operator!Hénon-like map|textit}Given
$\delta>0$ and intervals $I^{h},I^{v}$ with $I^{v}\supset I^{h}\Supset I$.
There exists $\overline{\epsilon}>0$ and $c>0$ so that for all $F\in\mathcal{H}_{\delta}^{r}(I^{h}\times I^{v},\overline{\epsilon})$
there exists an $\mathbb{R}$-symmetric orientation reversing affine
map $s=s(F)$ \nomenclature[s]{$s$}{Affine part of the Hénon rescaling}
which depends continuously on $F$ such that the following properties
hold:

Let $\Lambda(x,y)=(s(x),s(y))$ \nomenclature[Lambda]{$\Lambda$}{Affine part of the Hénon rescaling}
and $\phi=\Lambda\circ H$. \nomenclature[phi]{$\phi$}{Hénon rescaling}
\begin{enumerate}
\item The map $x\rightarrow h_{y}(x)$ is injective on a neighborhood of
$C(F)$ and hence $\phi$ is a diffeomorphism from a neighborhood
of $C(F)$ to its image.
\item The renormalization $RF\equiv\phi\circ F^{2}\circ\phi^{-1}$ is an
Hénon-like map defined on $I_{R}^{h}(\delta_{R})\times I_{R}^{v}(\delta_{R})$
for some $\delta_{R}>0$ and intervals $I_{R}^{h}$ and $I_{R}^{v}$.
The intervals satisfy $I_{R}^{h}\Supset[-1,1]$ and $I_{R}^{v}=s(I^{v})$.
\item The domain $I_{R}^{h}\times I_{R}^{v}$ contains $D(RF)$, and the
rescaling $\phi$ maps $\phi(C(F))=D(RF)$.
\item The fixed points satisfy the relation $\phi(p(0))=p_{RF}(-1)$ where
$p_{RF}(-1)$ is the saddle fixed point of $RF$ with an expanding
positive multiplier.
\item The renormalization has a representation $RF=(f_{R}-\epsilon_{R},x)$
where $f_{R}\in\mathcal{U}$. The representation satisfies the relations
\[
\left\Vert f_{R}-R_{c}f\right\Vert _{I_{R}^{h}(\delta_{R})}<c\left\Vert \epsilon\right\Vert 
\]
and
\[
\left\Vert \epsilon_{R}\right\Vert _{I_{R}^{h}(\delta_{R})\times I_{R}^{v}(\delta_{R})}<c\left\Vert \epsilon\right\Vert ^{2}.
\]
\end{enumerate}
\end{proposition}
\begin{proof}
See \cite[Section 3.5]{de2005renormalization}.
\end{proof}

\begin{remark}
The rescaling $\phi$ preserves the orientation along the $x$-coordinate
and reverses the orientation along the $y$-coordinate.
\end{remark}

A map is called infinitely renormalizable if the procedure of renormalization
can be done infinitely many times. The class of infinitely renormalizable\index{infinitely renormalizable}
Hénon-like map is denoted as $\mathcal{I}_{\delta}(I^{h}\times I^{v},\overline{\epsilon})\subset\mathcal{H}_{\delta}(I^{h}\times I^{v},\overline{\epsilon})$.
\nomenclature[I]{$\mathcal{I}_{\delta}$}{Class of infinite renormalizable Hénon-like maps.}

Assume that $F\in\mathcal{I}_{\delta}(I^{h}\times I^{v},\overline{\epsilon})$,
we define $F_{n}=R^{n}F$. The subscript $n$ is called the renormalization
level\index{renormalization level}. The subscript is also used to
indicate the associated renormalization level of an object. For example,
$H_{n}$, $s_{n}$, and $\Lambda_{n}$ are the functions in Proposition
\ref{prop:Renormalization Operator} that corresponds to $F_{n}$.
The vertical domain $I_{n}^{v}$ satisfies $I_{0}^{v}=I^{v}$ and
$I_{n+1}^{v}=s_{n}(I_{n}^{v})$ for all $n\geq0$. The vertical graphs
$W_{n}^{t}(j)$ are the local stable manifolds of $F_{n}$. The sets
$A_{n}$, $B_{n}$, and $C_{n}$ form a partition of the dynamical
domain $D_{n}$ that associates to $F_{n}$. The points $p_{n}(-1)$
and $p_{n}(0)$ are the two saddle fixed points of $F_{n}$.

\nomenclature[Phi_{n}^{j}]{$\Phi_{n}^{j}$}{Nonlinear rescaling from renormalization level $n$ to $n+j$}
Also, define $\Phi_{n}^{j}=\phi_{n+j-1}\circ\cdots\circ\phi_{n}$
and \nomenclature[lambda_{n}]{$\lambda_{n}$}{$s_n'$}$\lambda_{n}=s_{n}'(x)$. 

Recall $g\in\mathcal{U}$ is the fixed point of the renormalization
operator $\uRenormalizeC$, and $\lambda$ is the rescaling constant
defined in \ref{prop:Functional Equation}. \nomenclature[G]{$G$}{Fixed point for $R$}Let
$G(x,y)=(g(x),x)$ be the induced degenerate Hénon-like map.

The renormalization operator is hyperbolic. The next proposition lists
the properties of infinitely renormalizable Hénon-like maps.
\begin{proposition}[Hyperbolicity of the Renormalization operator]
\label{prop:Hyperbolicity of the Renormalization Operator}\index{renormalization operator!Hénon-like map!hyperbolicity}Given
$\delta>0$ and intervals $I^{h},I^{v}$ with $I^{v}\supset I^{h}\Supset I$.
There exists $\rho<1$ (universal), $\overline{\epsilon}>0$, $c>0$
such that for all $F\in\mathcal{I}_{\delta}(I^{h}\times I^{v},\overline{\epsilon})$
there exists $0<\delta_{R}<\delta$, an interval $I_{R}^{h}$ with
$I^{h}\supset I_{R}^{h}\Supset I$, and $b\in\mathbb{R}$ such that
the following properties hold: 

Let $F_{n}=R^{n}F$ be the sequence of renormalizations of $F$. Then
$F_{n}\in\mathcal{H}_{\delta_{R}}(I_{R}^{h}\times I_{n}^{v})$ for
all $n\geq0$. Also, the sequence has a representation $F_{n}=(f_{n}-\epsilon_{n},x)$
with $f_{n}\in\mathcal{U}_{\delta_{R}}(I_{R}^{h})$ that satisfies
\begin{enumerate}
\item $\left\Vert f_{n}-g\right\Vert _{I_{R}^{h}(\delta_{R})}<c\rho^{n}\left\Vert F-G\right\Vert _{I_{R}^{h}(\delta_{R})\times I^{v}(\delta_{R})}$
\item $\left\Vert \epsilon_{n+1}\right\Vert _{I_{R}^{h}(\delta_{R})\times I_{n+1}^{v}(\delta_{R})}<c\left\Vert \epsilon_{n}\right\Vert _{I_{R}^{h}(\delta_{R})\times I_{n}^{v}(\delta_{R})}^{2}$,
\item $\left\Vert f_{n+1}-s_{n}\circ f_{n}^{2}\circ s_{n}^{-1}\right\Vert _{I_{R}^{h}(\delta_{R})}<c\left\Vert \epsilon_{n}\right\Vert _{I_{R}^{h}(\delta_{R})\times I_{n}^{v}(\delta_{R})}$, 
\item $\left|\lambda_{n}-\lambda\right|<c\rho^{n}\left\Vert F-G\right\Vert _{I_{R}^{h}(\delta_{R})\times I^{v}(\delta_{R})}$,
and
\item $\epsilon_{n}(x,y)=b^{2^{n}}a(x)y(1+O(\rho^{n}))$ (universality)
\end{enumerate}
for all $n\geq0$ where $a(x)$ is a universal analytic positive function.
The value $\delta_{R}$ in the estimates can be replaced by any positive
number that is smaller than $\delta_{R}$.
\end{proposition}
\begin{proof}
See \cite[Theorem 3.5, 4.1, 7.9, and Lemma 7.4]{de2005renormalization}.
\end{proof}

\begin{remark}
The constant $b$ is called the average Jacobian of $F$. See \cite[Section 6]{de2005renormalization}.
\end{remark}
\begin{remark}
The Hénon-renormailzation is an operation that renormalizes around
the critical value. However, the renormalization $F_{n}$ converges
to the fixed point $G$ of the unimodal-renormalization that renormalizes
around the critical point. This is because of the nonlinear rescaling
$H$ maps the domain from $C$ to $B$ in the degenerate case. See
Chapter \ref{sec:degenerate case} for a more detail explanation.
\end{remark}
\begin{remark}
\label{rem:Comparsion between infinite renormalizable}Although infinitely
CLM-renormalizable in general does not imply infinitely renormalizable,
the hyperbolicity provides a connection between the two notions of
infinitely renormalizable. Assume that $F$ is infinitely CLM-renormalizable.
The hyperbolicity of the renormalizable operator \cite[Theorem 4.1]{de2005renormalization}
says that $R^{n}F$ converges to the fixed point $G$. This means
that $R^{n}F$ is also infinitely renormalizable for all $n$ sufficiently
large. This makes Theorem \ref{thm:nonexistence of wandering domain}
also applies to infinitely CLM-renormalizable Hénon-like maps. See
Remark \ref{rem:Absence of wandering domain for CLM-renormalizable}
for more details.
\end{remark}
From now on, for any infinitely renormalizable map $F$, we fix a
representation $F_{n}=(f_{n}-\epsilon_{n},x)$ such that the maps
$f_{n}$ and $\epsilon_{n}$ satisfy the properties given in Proposition
\ref{prop:Hyperbolicity of the Renormalization Operator}. Also, we
neglect the subscript of the supnorms $\left\Vert f_{n}-g\right\Vert =\left\Vert f_{n}-g\right\Vert _{I_{R}^{h}(\delta_{R})}$
and $\left\Vert \epsilon_{n}\right\Vert =\left\Vert \epsilon_{n}\right\Vert _{I_{R}^{h}(\delta_{R})\times I_{n}^{v}(\delta_{R})}$
whenever the context is clear. 
\begin{corollary}
There exists a constant $c>1$ such that 
\[
\left\Vert F_{n}-G\right\Vert <c\rho^{n}\left\Vert F-G\right\Vert 
\]
and
\[
\left\Vert \epsilon_{n+t}\right\Vert <\left(c\left\Vert \epsilon_{n}\right\Vert \right)^{2^{t}}
\]
for all $t\geq1$.
\end{corollary}

\begin{lemma}
\label{lem:Bounds for e_n}Assume that $\overline{\epsilon}>0$ small
enough such that Proposition \ref{prop:Hyperbolicity of the Renormalization Operator}
holds. There exists a constant $c_{1}>0$ such that the inequalities
hold
\begin{equation}
\left|\frac{\partial\epsilon_{n}}{\partial x}(x,y)\right|,\left|\frac{\partial\epsilon_{n}}{\partial y}(x,y)\right|\leq c_{1}\left\Vert \epsilon_{n}\right\Vert \label{eq:Upper bound for de}
\end{equation}
for all $F\in\mathcal{I}_{\delta}(I^{h}\times I^{v},\overline{\epsilon})$
and $(x,y)\in I^{h}\times I_{n}^{v}$. In addition, if $F$ is non-degenerate,
there exists $N=N(F)\geq0$, $\delta_{R}>0$, and $c_{2}>0$ such
that 
\begin{equation}
\frac{\partial\epsilon_{n}}{\partial y}(x,y)\geq\frac{c_{1}}{\left|I_{n}^{v}\right|}\left\Vert \epsilon_{n}\right\Vert \label{eq:Lower bound for de/dy}
\end{equation}
for all $(x,y)\in I^{h}\times I_{n}^{v}$ and $n\geq N$.
\end{lemma}
\begin{proof}
The first inequality (\ref{eq:Upper bound for de}) follows from Lemma
\ref{lem:Derivative Bound}.

By the universality (and the proof of \cite[Theorem 7.9]{de2005renormalization})
of the infinitely renormalizable Hénon-like maps, the perturbation
$\epsilon$ and its derivative has the asymptotic form
\[
\epsilon_{n}(x,y)=b^{2^{n}}a(x)y(1+O(\rho^{n}))
\]
and
\[
\frac{\partial\epsilon_{n}}{\partial y}(x,y)=b^{2^{n}}a(x)(1+O(\rho^{n})).
\]
Since $a$ is a positive map on a compact set that covers the whole
domain, the second inequality follows.
\end{proof}

To study the wandering domains, it is enough to consider Hénon-like
maps that are close to the hyperbolic fixed point $G$. By Corollary
\ref{cor:Wandering Domain/Renormalization} later, for any integer
$n\geq0$, we show an infinitely renormalizable Hénon-like map $F$
has a wandering domain in $D(F)$ if and only if $F_{n}$ has a wandering
domain in $D(F_{n})$. Also, the maps $F_{n}$ converge to the hyperbolic
fixed point $G$ as $n$ approaches to infinity by Proposition \ref{prop:Hyperbolicity of the Renormalization Operator}.
Thus, we focus on a small neighborhood of the fixed point $G$. 
\begin{definition}
Given $\delta>0$ and $I\Subset I^{h}\subset I^{v}$. If $\overline{\epsilon}$
is small enough such that Proposition \ref{prop:Hyperbolicity of the Renormalization Operator}
holds, define $\hat{\mathcal{I}}_{\delta}(I^{h}\times I^{v},\overline{\epsilon})$
to be the class of non-degenerate Hénon-like maps $F\in\mathcal{I}_{\delta}(I^{h}\times I^{v},\overline{\epsilon})$
such that $F_{n}\in\mathcal{H}_{\delta}(I^{h}\times I_{n}^{v},\overline{\epsilon})$,
$\left\Vert F_{n}-G\right\Vert <\overline{\epsilon}$, $\left|\lambda_{n}-\lambda\right|<\overline{\epsilon}$,
$\left\Vert s_{n}(x)-(-\lambda)x\right\Vert _{I^{h}}<\overline{\epsilon}$,
and (\ref{eq:Lower bound for de/dy}) holds for all $n\geq0$. 
\end{definition}
In the remaining part of the article, we will study the dynamics and
the topology of Hénon-like maps in this smaller class of maps.

\section{\label{sec:Infinite Renormalizable H=0000E9non-Like Maps}Structure
and Dynamics of Infinitely Renormalizable Hénon-Like Maps}

In this chapter, we study the topology of the local stable manifolds
and the dynamics on the partition for a infinitely renormalizable
Hénon-like map.

\subsection{\label{subsec:Rescaling-levels}Rescaling levels\index{rescaling level}}

This section introduces a finer partition of $C$, called the rescaling
levels\index{rescaling level}, based on the maximum possible rescalings
of a point in $C$.

For each two consecutive levels of renormalization $n$ and $n+1$,
the maps $F_{n}^{2}$ and $F_{n+1}$ are conjugated by the nonlinear
rescaling $\phi_{n}$. The rescaling $\phi_{n}$ relates the two renormalization
levels as follow.
\begin{lemma}
\label{lem:Two renormalization level-relation}Given $\delta>0$ and
$I^{v}\supset I^{h}\Supset I$. There exists $\overline{\epsilon}>0$
such that for all $F\in\hat{\mathcal{I}}_{\delta}(I^{h}\times I^{v},\overline{\epsilon})$
the following properties hold for all $n\geq0$:
\begin{enumerate}
\item $\phi_{n}(p_{n}(0))=p_{n+1}(-1)$, 
\item $\phi_{n}(W_{n}^{k}(0))=W_{n+1}^{k}(-1)$ for $k=0,2$, and 
\item $\phi_{n}:C_{n}\rightarrow D_{n}$ is a diffeomorphism.
\end{enumerate}
\end{lemma}
The itinerary of a point follows the arrows in the diagram.
\[
\xymatrix{ & A_{n+1}\ar[d]^{F_{n+1}}\ar@(ur,dr)^{{F_{n+1}}} &  & A_{n+2}\ar[d]^{F_{n+2}}\ar@(ur,dr)^{{F_{n+2}}}\\
 & B_{n+1}\ar[d]^{F_{n+1}} &  & B_{n+2}\ar[d]^{F_{n+2}}\\
C_{n}\ar[r]^{\phi_{n}}\ar[ur]|{\phi_{n}}\ar[uur]^{{\phi_{n}}} & C_{n+1}\ar[rr]\sp(0.5){\phi_{n+1}}\ar[urr]\sp(0.6){\phi_{n+1}}\ar[uurr]\sp(0.6){\phi_{n+1}} &  & C_{n+2}\ar[r]\sp(0.5){\phi_{n+2}} & \cdots
}
\]
The diagram says, if $z_{0}\in C_{n}$, then we can rescale the point.
The rescaled point $z_{1}=\phi_{n}(z_{0})$ enters the domain $D_{n+1}$
of the next renormalization level $n+1$ by Lemma \ref{lem:Two renormalization level-relation}.
On the renormalization level $n+1$, the rescaled point $z_{1}$ belongs
to one of the sets $A_{n+1}$, $B_{n+1}$, or $C_{n+1}$ if it is
disjoint from the stable manifolds. The process of rescaling stops
if $z_{1}$ belongs to $A_{n+1}$ or $B_{n+1}$ and $z_{0}$ can be
rescaled at most one time. If $z_{1}$ belongs to $C_{n+1}$, we can
continue to rescale the point. The rescaled point $z_{2}=\phi_{n+1}(z_{1})$
enters the domain $D_{n+2}$ of the next renormalization level $n+2$.
Similarly, the process of rescaling stops if $z_{2}$ belongs to $A_{n+2}$
or $B_{n+2}$ and $z_{0}$ can be rescaled at most two times. If $z_{2}$
belongs to $C_{n+2}$, we can rescale again and repeat the procedure
until the rescaled point enters the sets $A$ or $B$ of some deeper
renormalization level.

Motivated from the diagram, we define the finer partition $C_{n}(j)$
on $C_{n}$ by the maximal possible rescalings as follows.
\begin{definition}
\label{def:rescaling levels}\nomenclature[W_{n}^{2}(j)]{$W_{n}^{t}(j)$}{Local stable manifold of $p_{n}(j)$}\nomenclature[p_n(j)]{$p_n(j)$}{Periodic point with period $2^j$ for the Hénon-Like map $F_n$}\nomenclature[C_{n}(j)]{$C_{n}(j)$}{Subpartition for $C_{n}$ with rescaling level $j$}For
consistency, set $C_{n}(0)=A_{n}\cup W_{n}^{1}(0)\cup B_{n}$. Given
a positive integer $j$. The $j$-th rescaling level\index{rescaling level|textbf}
in $C$ is defined as $C_{n}(j)=\left(\Phi_{n}^{j}\right)^{-1}(C_{n+j}(0))$
and the $j$-th rescaling level in $B$ is defined as $B_{n}(j)=F_{n}^{-1}(C_{n}(j))$.
Also, set $p_{n}(j)=\left(\Phi_{n}^{j}\right)^{-1}(p_{n+j}(0))$ and
$W_{n}^{t}(j)=\left(\Phi_{n}^{j}\right)^{-1}(W_{n+j}^{t}(0))$\index{local stable manifold|textbf}
for $t=0,2$. 
\end{definition}

The diagram explains the definition of a rescaling level.
\[
\xymatrix{C_{n}(j)\ar[r]^{F_{n}^{2^{j}}}\ar[d]_{\Phi_{n}^{j}} & C_{n}(j)\ar[d]^{\Phi_{n}^{j}}\\
D_{n+j}\ar[r]_{F_{n+j}} & D_{n+j}
}
\]

From the definition, the relations of the rescaling levels between
two different renormalization level are listed as follow. 
\begin{proposition}
\label{prop:Dynamics on the partition}\index{partition!Hénon-like map|textit}\index{local stable manifold|textit}Given
$\delta>0$ and $I^{v}\supset I^{h}\Supset I$. There exists $\overline{\epsilon}>0$
such that for all $F\in\hat{\mathcal{I}}_{\delta}(I^{h}\times I^{v},\overline{\epsilon})$
the following properties hold for all $n\geq0$:
\begin{enumerate}
\item $p_{n}(j)$ is a periodic point of $F_{n}$ with period $2^{j}$ for
$j\geq0$.
\item $W_{n}^{t}(j)$ is a local stable manifold of $p_{n}(j)$ for $j\geq0$
and $t=0,2$.
\item $\Phi_{n}^{k}(W_{n}^{t}(j))=W_{n+k}^{t}(j-k)$ and $\Phi_{n}^{k}(p_{n}(j))=p_{n+k}(j-k)$
for $j\geq k-1$ and $t=0,2$.
\item The map $\Phi_{n}^{k}:C_{n}(j)\rightarrow C_{n+k}(j-k)$ is a diffeomorphism
for $j\geq k$, and 
\item For each $j\geq0$, the set $C_{n}(j)$ contains two components. The
left component \nomenclature[C_{n}^{l}(j)]{$C_{n}^{l}(j)$}{Left component of $C_{n}(j)$}$C_{n}^{l}(j)$
is the set bounded between $W_{n}^{0}(j-1)$ and $W_{n}^{0}(j)$ and
the right component \nomenclature[C_{n}^{r}(j)]{$C_{n}^{r}(j)$}{Right component of $C_{n}(j)$}$C_{n}^{r}(j)$
is the set bounded between $W_{n}^{2}(j)$ and $W_{n}^{2}(j-1)$.
\end{enumerate}
\end{proposition}
\begin{figure}
\resizebox{\columnwidth}{!}{%
\begin{centering}
\includegraphics[scale=0.8]{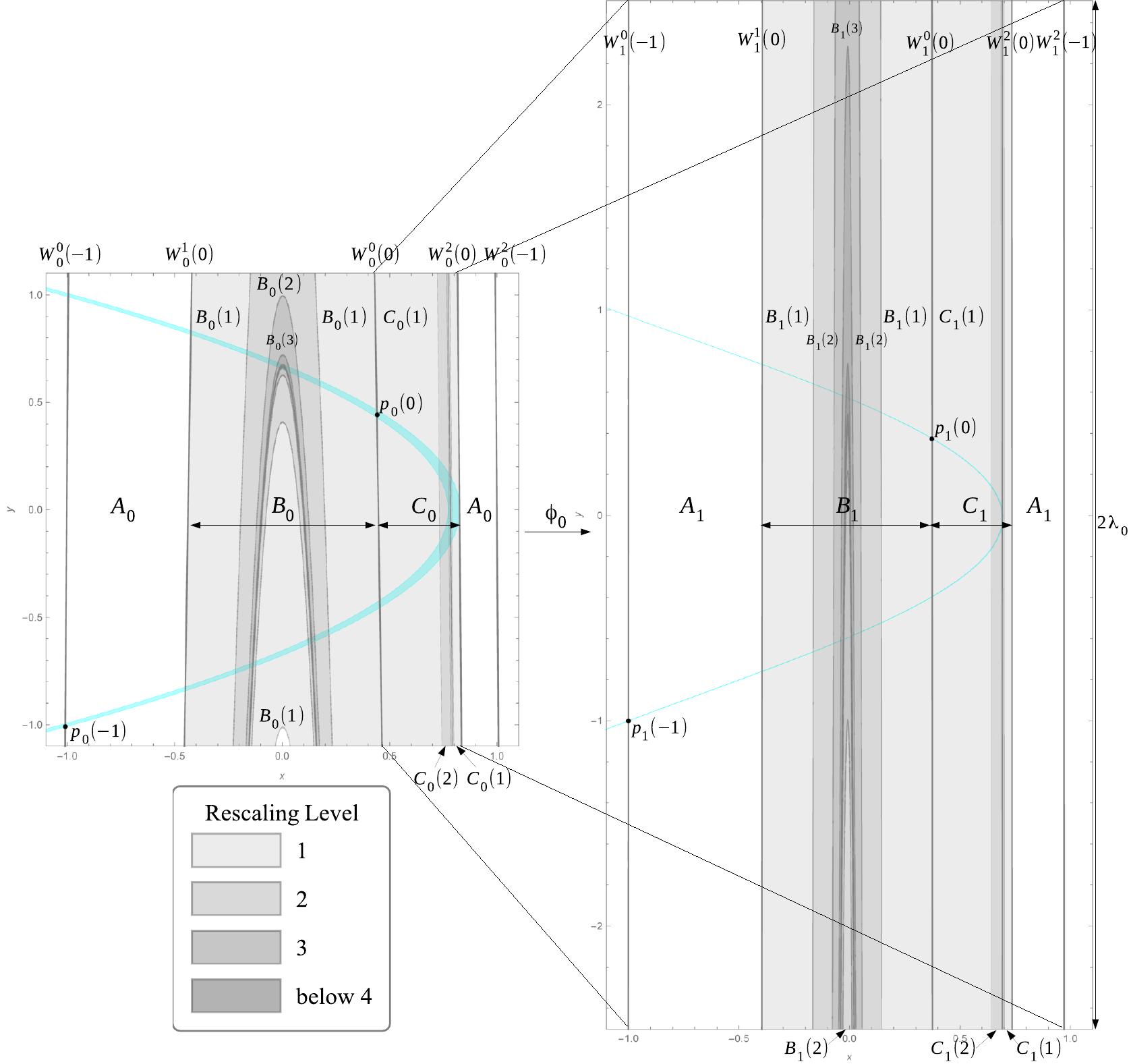}
\par\end{centering}
}

\caption{\label{fig:Partition ABC =000026 rescaling}The partition and the
local stable manifolds of two renormalization levels $F_{0}$ and
$F_{1}$ from the left to the right. The rescaling levels 1, 2, 3,
and below 4 are shaded from light to dark as shown in the legend. }
\end{figure}
 The partition and the local stable manifolds $W_{n}^{t}(j)$ are
illustrated in Figure \ref{fig:Partition ABC =000026 rescaling}.
The sets $\left\{ C_{n}(j)\right\} _{j\geq1}$ form a partition of
$C_{n}$ and the sets $\left\{ B_{n}(j)\right\} _{j\geq1}$ form a
partition of $B_{n}$. 

Next, we introduce the tip to study the geometric structure of the
rescaling levels in $C$. Recall from \cite[Section 7.2]{de2005renormalization}
that 
\begin{definition}[Tip]
\nomenclature[tau]{$\tau$}{The tip of an infinite renormalizable Hénon-like map}Assume
that $\overline{\epsilon}>0$ is sufficiently small. The tip\index{tip|textbf}
$\tau$ of an infinitely renormalizable Hénon-like map $F\in\hat{\mathcal{I}}_{\delta}(I^{h}\times I^{v},\overline{\epsilon})$
is the unique point such that 
\[
\left\{ \tau\right\} =\cap_{j=N}^{\infty}\left(\Phi_{0}^{j}\right)^{-1}(D_{j}\cap I^{h}\times I^{h})
\]
for all $N\geq0$.
\end{definition}
The tip is an analog of the critical value in the non-degenerate case.
Roughly speaking, the tip generates the attracting \index{Cantor set}Cantor
set of a Hénon-like map. See \cite[Chapter 5]{de2005renormalization}
for more information.

From Proposition \ref{prop:Dynamics on the partition}, a rescaling
level $C_{n}(j)$ contains two components which are both bounded by
two local stable manifolds. The following proposition lists the geometric
properties of the local stable manifolds. 

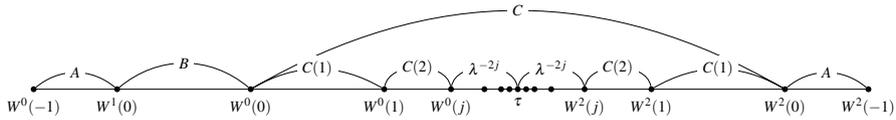
\begin{figure}
\begin{centering}
\resizebox{\columnwidth}{!}{%

\begin{tikzpicture}[
	letter/.style={circle, minimum size=3pt, inner sep=0, outer sep=0, fill=black, label=below:#1},
	number/.style={fill=white, pos=.5}   ]
	\draw (0,0) --
		node(wa0)[letter=$W^0(-1)$,pos=0]{}
		node(w01)[letter=$W^1(0)$,pos=.1]{}
		node(w00)[letter=$W^0(0)$,pos=.26]{}
		node(w10)[letter=$W^0(1)$,pos=.42]{}
		node(w20)[letter=$W^0(j)$,pos=.5]{}
		node(w30)[letter=$$,pos=.54]{}
		node(w40)[letter=$$,pos=.56]{}
		node(w50)[letter=$$,pos=.57]{}
		node(t)[letter=$\tau$,pos=.58]{}
		node(w52)[letter=$$,pos=.59]{}
		node(w42)[letter=$$,pos=.60]{}
		node(w32)[letter=$$,pos=.62]{}
		node(w22)[letter=$W^2(j)$,pos=.66]{}
		node(w12)[letter=$W^2(1)$,pos=.74]{}
		node(w02)[letter=$W^2(0)$,pos=.9]{}
		node(wa2)[letter=$W^2(-1)$,pos=1]{}
		(15,0);
	\draw[-]
		(wa0) to[bend left=40]	node[number]{$A$} (w01)
		(w01) to[bend left=40]	node[number]{$B$} (w00)
		(w00) to[bend left=30]	node[number]{$C$} (w02)
		(w00) to[bend left=30]	node[number]{$C(1)$} (w10)
		(w10) to[bend left=80]	node[number]{$C(2)$} (w20)
		(w20) to[bend left=80]	node[number]{$\lambda^{-2j}$} (t)
		(t) to[bend left=80]	node[number]{$\lambda^{-2j}$} (w22)
		(w22) to[bend left=80]	node[number]{$C(2)$} (w12)
		(w12) to[bend left=30]	node[number]{$C(1)$} (w02)
		(w02) to[bend left=40]	node[number]{$A$} (wa2);
\end{tikzpicture}

}
\par\end{centering}
\caption{\label{fig:Horizontal cross section that intersect the tip}The structure
of the partition of the domain. The figure shows the partition and
the local stable manifolds on the horizontal cross section that intersects
the tip.}
\end{figure}
\begin{proposition}
\label{prop:Geometric properties of W_n(j)}\index{local stable manifold|textit}\index{tip|textit}Given
$\delta>0$ and $I^{v}\supset I^{h}\Supset I$. There exists $\overline{\epsilon}>0$,
$c>0$ and $c'>1$ such that for all $F\in\hat{\mathcal{I}}_{\delta}(I^{h}\times I^{v},\overline{\epsilon})$
the following properties hold for all $n\geq0$:
\begin{enumerate}
\item $W_{n}^{t}(j)$ is a vertical graph with Lipschitz constant $c\left\Vert \epsilon_{n}\right\Vert $
for all $j\geq-1$ and $t=0,2$.
\item $\frac{1}{c'}\left(\frac{1}{\lambda}\right)^{2j}<\left|z_{n}^{(t)}(j)-\tau_{n}\right|<c'\left(\frac{1}{\lambda}\right)^{2j}$
for all $j\geq-1$ and $t=0,2$ where $z_{n}^{(t)}(j)$ is the intersection
point of $W_{n}^{t}(j)$ with the horizontal line through $\tau_{n}$.
See Figure \ref{fig:Horizontal cross section that intersect the tip}.
\end{enumerate}
\end{proposition}
\begin{proof}
See \cite[Lemma 3.4 and Proposition 3.5]{lyubich2011renormalization}.
\end{proof}

Finally, we study the geometric structure of the rescaling levels
in $B$. In the degenerate case, each rescaling level $B_{n}(j)$
contains two components which are bounded by local stable manifolds.
In the non-degenerate case, fix an integer $j$, the rescaling level
$B_{n}(j)$ also contains two components when $\overline{\epsilon}$
is small enough. The geometric properties of the boundary local stable
manifolds are listed as follow. 

\begin{proposition}
\label{prop:W(j) in B for n large}\index{local stable manifold|textit}Given
$\delta>0$ and $I^{v}\supset I^{h}\Supset I$. For all $j\geq0$
and $d>0$, there exists $\overline{\epsilon}=\overline{\epsilon}(j,d)>0$
and $c=c(j)>0$ such that for all $F\in\hat{\mathcal{I}}_{\delta}(I^{h}\times I^{v},\overline{\epsilon})$
the following properties hold for all $n\geq0$:
\begin{enumerate}
\item $F_{n}^{-1}(W_{n}^{0}(j))$ has exactly two components $W_{n}^{l}(j)\subset[q^{l}(j)-d,q^{l}(j)+d]\times I_{n}^{v}$
and $W_{n}^{r}(j)\subset[q^{r}(j)-d,q^{r}(j)+d]\times I_{n}^{v}$. 
\item Both components $W_{n}^{l}(j)$ and $W_{n}^{r}(j)$ are vertical graphs
with Lipschitz constant $c\left\Vert \epsilon_{n}\right\Vert $.
\end{enumerate}
\end{proposition}
\begin{proof}
The proof is similar to Proposition \ref{prop:Geometric properties of W_n(j)}.
\end{proof}

\begin{remark}
Unlike Proposition \ref{prop:Geometric properties of W_n(j)}, here
the constant $\overline{\epsilon}$ is not uniform on $j\geq0$. For
a non-degenerate Hénon-like map, the structure of the local stable
manifolds is similar to degenerate case when $j$ is large. The local
stable manifold $W_{n}^{0}(j)$ is far away from the tip and hence
the pullback $F_{n}^{-1}(W_{n}^{0}(j))$ is the union of two vertical
graphs in $B_{n}$. However, the structure turns to be different when
$j$ is large. The local stable manifold is close to the tip and the
vertical line argument in Chapter \ref{sec:Good and Bad region} shows
that the pullback $F_{n}^{-1}(W_{n}^{0}(j))$ is a concave curve in
$B_{n}$.
\end{remark}

\subsection{Asymptotic behavior near $G$}

In this section, we estimate the derivatives of a Hénon-like map that
is close to the hyperbolic fixed point $G$. Define $v_{n}\in I^{h}$
to be the critical point of $f_{n}$ and $w_{n}=f_{n}(v_{n})$ be
the critical value.

The first lemma proves that a Hénon-like map acts like a quadratic
map on $B$.
\begin{lemma}
\label{lem:Quadratic center}Given $\delta>0$ and $I^{v}\supset I^{h}\Supset I$.
There exists $a>0$ (universal), $\overline{\epsilon}>0$, and an
interval $I^{B}\subset I^{h}$ (universal) such that for all $F\in\hat{\mathcal{I}}_{\delta}(I^{h}\times I^{v},\overline{\epsilon})$
the following properties hold for all $n\geq0$:

The interior of $I^{B}$ contains $\hat{q}(0)$ and $q(0)$, $I^{B}\times I_{n}^{v}\supset B_{n}$,
\[
\frac{1}{a}\left|x-v_{n}\right|\leq\left|f_{n}'(x)\right|\leq a\left|x-v_{n}\right|,
\]
and
\[
\frac{1}{2a}\left(x-v_{n}\right)^{2}\leq\left|f_{n}(x)-f_{n}(v_{n})\right|\leq\frac{a}{2}\left(x-v_{n}\right)^{2}
\]
for all $x\in I^{B}$.
\end{lemma}
\begin{proof}
The lemma is true because the map $F$ is close to the hyperbolic
fixed point $G$ and the map $g$ is concave on the compact set $[-c^{(1)},c^{(1)}]$
by Proposition \ref{prop:Functional Equation}.
\end{proof}

The next lemma shows that a Hénon-like map is expanding on $A$ and
$C$ in the $x$-coordinate when it is close enough to the fixed point
$G$.
\begin{lemma}
\label{lem:Expanding rate on A and C}Given $\delta>0$ and $I^{v}\supset I^{h}\Supset I$.
There exists $E>1$ (universal), $\overline{\epsilon}>0$, and a union
of two intervals $I^{AC}\subset I^{h}$ such that for all $F\in\hat{\mathcal{I}}_{\delta}(I^{h}\times I^{v},\overline{\epsilon})$
the following properties hold for all $n\geq0$:

The interior of $I^{AC}$ contains $q(-1)$, $\hat{q}(0)$ , $q(0)$,
and $\hat{q}(-1)$, $I^{AC}\times I_{n}^{v}\supset A_{n}\cup W_{n}^{2}(0)\cup C_{n}$,
and 
\[
\left|\frac{\partial h_{n}}{\partial x}(x,y)\right|\geq E
\]
 for all $(x,y)\in I^{AC}\times I_{n}^{v}$ .
\end{lemma}
\begin{proof}
The lemma is true because the map $g$ is expanding on $A$ and $C$
by Proposition \ref{prop:Derivative of g on A} and the Hénon-like
map $F$ is close to the hyperbolic fixed point $G$ of the renormalization
operator.
\end{proof}

\subsection{Relation between the tip\index{tip} and the critical value\index{critical value}}

In Lemma \ref{lem:Quadratic center}, we proved that a Hénon-like
map behaves like a quadratic map when a point is close to the critical
point $v_{n}$ of $f_{n}$ for the representation $F_{n}=(f_{n}-\epsilon_{n},x)$.
However, the critical point $v_{n}$ and the critical value $w_{n}$
in the estimates depend on the representation.

In this section, we show that the critical value $w_{n}$ (for any
representation) is $\left\Vert \epsilon_{n}\right\Vert $-close to
the tip $\tau_{n}$ in Proposition \ref{prop:Bound for f(v)-tau}.
This allows us to replace $v_{n}$ and $w_{n}$ by the representation
independent quantity $\tau_{n}$. This makes the quadratic estimations
in Lemma \ref{lem:Quadratic center} useful when a point is $\left\Vert \epsilon_{n}\right\Vert $-away
from the tip.

To estimate the distance from the tip to the critical value, we write
$\tau_{n}=(a_{n},b_{n})$. Since the rescaling $\phi_{n}$ maps a
horizontal line to a horizontal line, we focus on the horizontal slice
that intersects the tip in each renormalization level. Define the
restriction of the rescaling map $\phi$ to the slice as 
\[
\eta_{n}(x)=\pi_{x}\circ\phi_{n}(x,b_{n})=s_{n}\circ h_{n}(x,b_{n}).
\]
By the definition of the tip, the quantities satisfy the recurrence
relations $\phi_{n}(\tau_{n})=\tau_{n+1}$, $\eta_{n}(a_{n})=a_{n+1}$,
and $s_{n}(b_{n})=b_{n+1}$.

First, we prove a lemma that allows us to compare the critical value
between two renormaliztion levels.
\begin{lemma}
\label{lem:Critical value is e-orbit}Given $\delta>0$ and $I^{v}\supset I^{h}\Supset I$.
There exists $\overline{\epsilon}>0$ and $c>0$ such that for all
$F\in\hat{\mathcal{I}}_{\delta}(I^{h}\times I^{v},\overline{\epsilon})$
we have 
\[
\left|w_{n+1}-\eta_{n}(w_{n})\right|<c\left\Vert \epsilon_{n}\right\Vert 
\]
for all $n\geq0$.
\end{lemma}
\begin{proof}
First, we compare the critical points $v_{n}$ and $v_{n+1}$. By
Proposition \ref{prop:Hyperbolicity of the Renormalization Operator},
we have
\[
\left\Vert f_{n+1}'-(s_{n}\circ f_{n}^{2}\circ s_{n}^{-1})'\right\Vert _{I^{h}}<c\left\Vert f_{n+1}-s_{n}\circ f_{n}^{2}\circ s_{n}^{-1}\right\Vert _{I^{h}(\delta)}<c\left\Vert \epsilon_{n}\right\Vert 
\]
for some constant $c>0$ when $\overline{\epsilon}>0$ is sufficiently
small. Since the critical point of the map $f_{n+1}$ is nondegenerate and
the map $(s_{n}\circ f_{n}^{2}\circ s_{n}^{-1})'$ is a small perturbation
of $f_{n+1}'$, the root $s_{n}(v_{n})$ of $(s_{n}\circ f_{n}^{2}\circ s_{n}^{-1})'$
is also a small perturbation of the root $v_{n+1}$ of $f_{n+1}'$.
That is, there exists $c'>0$ such that 
\[
\left|v_{n+1}-s_{n}(v_{n})\right|\leq c'\left\Vert \epsilon_{n}\right\Vert .
\]
The constant $c'$ can be chosen to be independent of $F$ because
$F$ is close to $G$.

Moreover, by the quadratic estimates in Lemma \ref{lem:Quadratic center},
we get
\begin{eqnarray*}
\left|f_{n+1}(v_{n+1})-s_{n}\circ f_{n}^{2}(v_{n})\right| & \leq & \left|f_{n+1}(v_{n+1})-f_{n+1}(s_{n}(v_{n}))\right|+\left|f_{n+1}(s_{n}(v_{n}))-s_{n}\circ f_{n}^{2}(v_{n})\right|\\
 & \leq & \frac{a}{2}\left|v_{n+1}-s_{n}(v_{n})\right|^{2}+\left|f_{n+1}(s_{n}(v_{n}))-s_{n}\circ f_{n}^{2}\circ s_{n}^{-1}(s_{n}(v_{n}))\right|\\
 & \leq & \frac{ac^{\prime2}}{2}\left\Vert \epsilon_{n}\right\Vert ^{2}+c\left\Vert \epsilon_{n}\right\Vert \\
 & \leq & c''\left\Vert \epsilon_{n}\right\Vert 
\end{eqnarray*}
for some constant $c''>0$.

Finally, we compare the critical values $w_{n}$ and $w_{n+1}$. Compute
\begin{eqnarray*}
\left|w_{n+1}-\eta_{n}(w_{n})\right| & = & \left|f_{n+1}(v_{n+1})-s_{n}(f_{n}^{2}(v_{n})-\epsilon_{n}(f_{n}(v_{n}),b_{n}))\right|\\
 & \leq & \left|f_{n+1}(v_{n+1})-s_{n}\circ f_{n}^{2}(v_{n})\right|+\lambda_{n}\left|\epsilon_{n}(f_{n}(v_{n}),b_{n})\right|\\
 & \leq & c''\left\Vert \epsilon_{n}\right\Vert +2\lambda\left\Vert \epsilon_{n}\right\Vert \\
 & = & (c''+2\lambda)\left\Vert \epsilon_{n}\right\Vert 
\end{eqnarray*}
for all $n\geq0$ whenever $\overline{\epsilon}$ is small enough
such that $\lambda_{n}\leq2\lambda$.
\end{proof}

The rescaling maps $\{\eta_{n}\}_{n\ge0}$ can be viewed as a non-autonomous
dynamical system (system that depends on time). An orbit is defined
as follows.
\begin{definition}[Orbit of Non-Autonomous Systems]
Let $Y_{n}$ be a complete metric space, $X_{n}\subset Y_{n}$ be
a closed subset, and $f_{n}:X_{n}\rightarrow Y_{n+1}$ be a continuous
map for all $n\geq1$. A sequence $\left\{ x_{n}\right\} _{n=1}^{\infty}$
is an orbit of the non-autonomous system $\left\{ f_{n}\right\} _{n=1}^{\infty}$
if $x_{n}\in X_{n}$ and $x_{n+1}=f_{n}(x_{n})$ for all $n\geq1$.
A sequence $\left\{ x_{n}\right\} _{n=1}^{\infty}$ is an $\epsilon$-orbit
of the non-autonomous system $\left\{ f_{n}\right\} _{n=1}^{\infty}$
if $x_{n}\in X_{n}$ and $\left|x_{n+1}-f_{n}(x_{n})\right|<\epsilon$
for all $n\geq1$.
\end{definition}
Next, we  state an analog of the shadowing theorem for non-autonomous
systems.
\begin{lemma}[Shadowing Theorem for Non-Autonomous Systems]
\label{lem:Shadowing theorem for random maps}\index{shadowing theorem|textit}For
each $n\geq1$, let $Y_{n}$ be a complete metric space equipped with
a metric $d$ (the metric depends on $n$), $X_{n}\subset Y_{n}$
be a closed subset, and $f_{n}:X_{n}\rightarrow Y_{n+1}$ be a homeomorphism.
Also assume that the non-autonomous system $\left\{ f_{n}\right\} _{n=1}^{\infty}$
has a uniform expansion. That is, there exists a constant $L>1$ such
that $\left|f_{n}(a)-f_{n}(b)\right|\geq L\left|a-b\right|$ for all
$a,b\in X_{n}$ and $n\geq1$. 

If $\left\{ x_{n}\right\} _{n=1}^{\infty}$ is an $\epsilon$-orbit
of $\left\{ f_{n}\right\} _{n=1}^{\infty}$, there exists a unique
orbit $\left\{ u_{n}\right\} _{n=1}^{\infty}$ of $\left\{ f_{n}\right\} _{n=1}^{\infty}$
such that 
\[
d(x_{n},u_{n})\leq\frac{\epsilon}{L-1}
\]
for all $n\geq1$. In addition, if $\left\{ X_{n}\right\} _{n=1}^{\infty}$
is uniformly bounded, then the non-autonomous system $\left\{ f_{n}\right\} _{n=1}^{\infty}$
has exactly one orbit $\left\{ u_{n}\right\} _{n=1}^{\infty}$. 
\end{lemma}
This Lemma is an analog of the Anosov's Shadowing Theorem. See \cite[Exercise 5.1.3, Corollary 5.3.2]{brin2002introduction}
for the version of autonomous systems. The proof is left to the reader.

The result from Lemma \ref{lem:Critical value is e-orbit} shows that
the sequence of critical values $w_{n}$ is an $\epsilon$-orbit of
the expanding non-autonomous system $\eta_{n}$. With the help from
the Shadowing Theorem, we are able to obtain the goal for this section.
\begin{proposition}
\label{prop:Bound for f(v)-tau}\index{tip|textit}\index{critical value|textit}Given
$\delta>0$ and $I^{v}\supset I^{h}\Supset I$. There exists $\overline{\epsilon}>0$
and $c>0$ such that for all $F\in\hat{\mathcal{I}}_{\delta}(I^{h}\times I^{v},\overline{\epsilon})$
we have 
\[
\left|f_{n}(v_{n})-\pi_{x}(\tau_{n})\right|<c\left\Vert \epsilon_{n}\right\Vert 
\]
for all $n\geq0$.
\end{proposition}
\begin{proof}
Fix $n\geq0$. The critical values $\{w_{j}\}_{j\geq n}$ form an
$\left\Vert \epsilon_{n}\right\Vert $-orbit and the tips $\{\tau_{j}\}_{j\geq n}$
form an orbit of the perturbed maps $\{\eta_{j}\}_{j\geq n}$. Also,
the perturbed maps are uniform expanding by Lemma \ref{lem:Expanding rate on A and C}
and $\lambda_{n}>1$. Therefore, the proposition follows by Lemma
\ref{lem:Shadowing theorem for random maps}.
\end{proof}

In addition, we can also estimate the distance from the critical point
to the preimage of the tip.
\begin{corollary}
\label{cor:Bound for v-tau}\index{tip|textit}\index{critical point|textit}Given
$\delta>0$ and $I^{v}\supset I^{h}\Supset I$. There exists $\overline{\epsilon}>0$
and $c>0$ such that for all $F\in\hat{\mathcal{I}}_{\hat{\delta}}(\hat{I}^{h}\times I^{v},\overline{\epsilon})$
we have 
\[
\left|v_{n}-\pi_{y}(\tau_{n})\right|<c\sqrt{\left\Vert \epsilon_{n}\right\Vert }
\]
for all $n\geq0$.
\end{corollary}
\begin{proof}
The corollary is true because the map $f_{n}$ behaves like a quadratic
map near the critical point by Lemma \ref{lem:Quadratic center}.
\end{proof}

\section{\label{sec:Framework}Closest Approach}

The proof for the nonexistence of wandering domains begins from this
chapter. We assume the contrapositive: there exists a wandering domain
$J$.

In this chapter, we construct a rescaled orbit $\left\{ J_{n}\right\} _{n=0}^{\infty}$
of an wandering domain $J$ which is called the $J$-closest approach.
Then we define the horizontal size $l_{n}$, the vertical size $h_{n}$,
and the rescaling level $k_{n}$ of an element $J_{n}$. 

Recall the definition of wandering domains.
\begin{definition}[Wandering Domain]
\label{def:Wandering domain}Assume that $F\in\mathcal{H}_{\delta}(I^{h}\times I^{v})$,
$D(F)$ exists, and $F$ is an open map (diffeomorphism from $D(F)$
to the image). A nonempty connected open set $J\subset D(F)$ is a
\index{wandering domain|textbf}wandering domain of $F$ if the orbit$\left\{ F^{n}(J)\right\} _{n\geq0}$
does not intersect the stable manifold of a periodic point.
\end{definition}
\begin{remark}
\label{rem:Wandering domain compare with the classical definition}The
classical definition of wandering intervals contains one additional
condition: the elements of the orbit do not intersect. This condition
is redundant for case of the unimodal maps. Assume that $J$ is an
nonempty open interval that does not contain points from the basin
of a periodic orbit. If the elements in the orbit of $J$ intersect,
then take a connected component $A$ of the union of the orbit that
contains at least two elements from the orbit. Then, there exists
a positive integer $n$ such that $f^{n}(U)\subset U$. It is easy
to show that $f^{n}$ has a fixed point in the interior of $U$ by
applying the Brouwer fixed-point theorem several times which leads
to a contradiction. Therefore, the orbit elements of $J$ are disjoint.
\end{remark}
The following proposition allow us to generate wandering domains by
iteration and rescaling.
\begin{proposition}
\label{prop:Wandering Domain/Renormalization}\index{wandering domain|textit}Given
$\delta>0$ and $I^{v}\supset I^{h}\Supset I$. There exists $\overline{\epsilon}>0$
such that for all open maps $F\in\mathcal{H}_{\delta}^{r}(I^{h}\times I^{v})$,
the following properties hold:
\begin{enumerate}
\item A set $J\subset D(F)$ is a wandering domain of $F$ if and only if
$F(J)$ is a wandering domain of $F$.
\item A set $J\subset C(F)$ is a wandering domain of $F$ if and only if
$\phi(J)\subset D(RF)$ is a wandering domain of $RF$.
\end{enumerate}
\end{proposition}
\begin{proof}
The proposition is true because the stable manifold of a periodic
orbit is invariant under iteration and the rescaling of a stable manifold
is also a stable manifold.
\end{proof}

\begin{corollary}
\label{cor:Wandering Domain/Renormalization}\index{wandering domain|textit}Given
$\delta>0$ and $I^{v}\supset I^{h}\Supset I$. There exists $\overline{\epsilon}>0$
such that for all open maps $F\in\mathcal{H}_{\delta}^{r}(I^{h}\times I^{v},\overline{\epsilon})$,
$F$ has a wandering domain in $D(F)$ if and only if $RF$ has a
wandering domain in $D(RF)$.
\end{corollary}
\begin{proof}
Assume that $J\subset D(F)$ is a wandering domain. If $J\subset C$,
then $RF$ has a wandering domain by Proposition \ref{prop:Wandering Domain/Renormalization}.
If $J\subset A$, there exists $n\geq1$ such that $F^{n}(J)\subset B$
by Proposition \ref{prop:Dynamics of the partition on D}. If $J\subset B$,
then $F(J)\subset C$ by Proposition \ref{prop:Dynamics of the partition on D}.
Thus, $RF$ has a wandering domain by Proposition \ref{prop:Wandering Domain/Renormalization}.

The converse follows from the second property of Proposition \ref{prop:Wandering Domain/Renormalization}.
\end{proof}

Also, we define the rescaling level of a wandering domain in $B$.
\begin{definition}[Rescaling level]
Assume that $U\subset A_{n}\cup B_{n}$ is a connected set that does
not intersect any of the stable manifolds. Define the rescaling level\index{rescaling level!Hénon-like|textbf}
\nomenclature[k]{$k$}{Level of rescaling}$k(U)$ as follows. If $U\subset B_{n}$,
set $k(U)$ to be the integer such that $U\subset B_{n}(k(U))$; otherwise
if $U\subset A_{n}$, set $k(U)=0$.
\end{definition}
To study the dynamics of a wandering domain, we apply the procedure
of renormalization. If a wandering domain is contained in $A_{0}$
or $B_{0}$, then its orbit will eventually leave $A_{0}$ and $B_{0}$
and enter $C_{0}$. If the orbit of the wandering domain enters $C_{0}$,
we rescale the orbit element by $\phi_{0}$, $\phi_{1}$, $\cdots$
as many times as possible until it lands on one of the sets $A_{n}$
or $B_{n}$ of some renormalization level $n$, then study the dynamics
of the rescaled orbit by the renormalized map $F_{n}$. If the rescaled
orbit enters $C_{n}$ again, then we rescale it and repeat the same
procedure again. By this process, we construct a rescaled orbit as
follows.
\begin{definition}[Closest approach]
\label{def: J_n Sequence of Wandering Domain}\index{wandering domain!sequence|see{closest approach}}Assume
that $\overline{\epsilon}>0$ is sufficiently small and $F\in\mathcal{I}_{\delta}(I^{h}\times I^{v},\overline{\epsilon})$. 

Given a set $J\subset A\cup B$ such that it does not intersect any
of the stable manifolds. Define a sequence of sets \nomenclature[J_n]{$J_{n}$}{$J$-closest approach}$\left\{ J_{n}\right\} _{n=0}^{\infty}$
and the associate renormalization levels \nomenclature[r(n)]{$r(n)$}{Level of renormalization of the sequence of wandering domain $J_n$}$\left\{ r(n)\right\} _{n=0}^{\infty}$
by induction such that  $J_{n}\subset A_{r(n)}\cup B_{r(n)}$ for
all $n\geq0$. 
\begin{enumerate}
\item Set $J_{0}=J$ and $r(0)=0$. 
\item Write the rescaling level of $J_{n}$ as $k_{n}=k(J_{n})$ whenever
$J_{n}$ is defined.
\item If $J_{n}\subset A_{r(n)}$, set $J_{n+1}=F_{r(n)}(J_{n})$ and $r(n+1)=r(n)$. 
\item If $J_{n}\subset B_{r(n)}$, set $J_{n+1}=\Phi_{r(n)}^{k_{n}}\circ F_{r(n)}(J_{n})$
and $r(n+1)=r(n)+k_{n}$. 
\end{enumerate}
The transition between two constitutive sequence elements, one iteration
together with rescaling (if possible), is called one step\index{closest approach!step}.
The sequence $\left\{ J_{n}\right\} _{n=0}^{\infty}$ is called the
rescaled iterations of $J$ that closest approaches the tip\index{tip},
or $J$-closest approach\index{closest approach|textbf} for short.
\end{definition}

The itinerary of a closest approach is summarized by the following
diagram. 

\[
\xymatrix@C=6pc{A_{r(n)}\ar[d]^{F_{r(n)}}\ar@(ur,dr)^{{F_{r(n)}}} &  & A_{r(n+1)}\\
B_{r(n)}\ar[r]\sp(0.4){\Phi_{r(n)}^{k_{n}}\circ F_{r(n)}} & A_{r(n+1)}\cup B_{r(n+1)}\ar[r]\ar[ur] & B_{r(n+1)}
}
\]
\begin{example}
\label{exa:Sequence of wandering domain}In this example, we explain
the construction of a closest approach and demonstrate the idea of
proving the nonexistence of wandering domains. Let $F=(f-\epsilon,x)$
be a Hénon-like map such that $f(x)=1.7996565(1+x)(1-x)-1$ and $\epsilon(x,y)=0.025y$.
The map $F$ is numerically checked to be seven times renormalizable.
Given a set $J=\left(-0.950,-0.947\right)\times\left(0.042,0.045\right)\subset A$.
We show that the set is not a wandering domain by contradiction. 

\begin{figure}
\resizebox{\columnwidth}{!}{%
\begin{centering}
\includegraphics[scale=0.8]{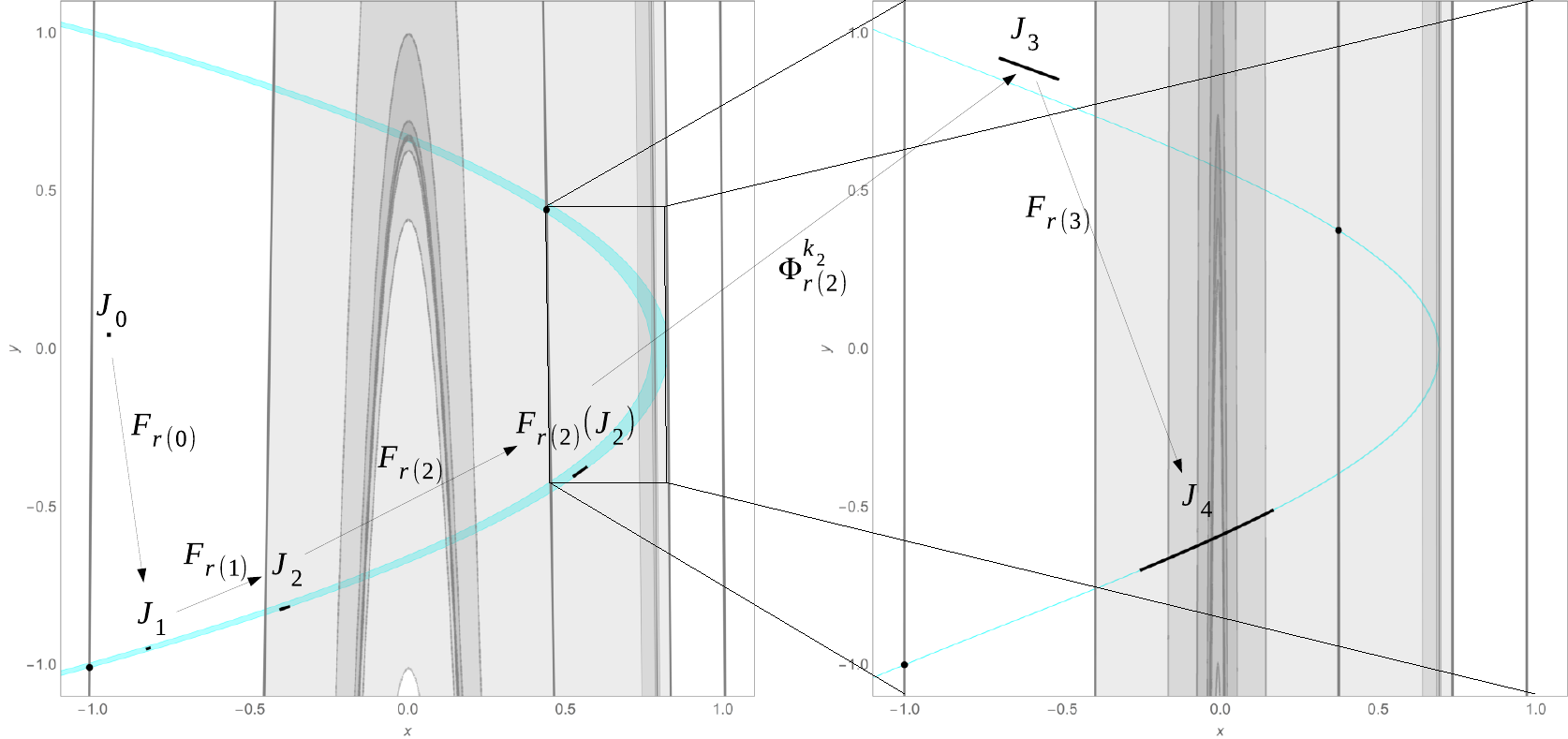}
\par\end{centering}
}

\caption{\label{fig:Sequence of wandering domain}The construction of a closest
approach $J_{n}$. The graphs are the domains and the partitions of
$F_{0}$ and $F_{1}$ from the left to the right.}
\end{figure}

If $J$ is a wandering domain, we construct a $J$-closest approach
as shown in Figure \ref{fig:Sequence of wandering domain}. Set $J_{0}=J$
and $r(0)=0$. The set $J_{0}$ is contained in $A_{r(0)}$. The next
element is defined to be $J_{1}=F_{r(0)}(J_{0})$ and $r(1)=r(0)=0$.
The set $J_{1}$ is also contained in $A_{r(1)}$. Set $J_{2}=F_{r(1)}(J_{1})$
and $r(2)=r(1)=0$. The set $J_{2}$ is contained in $B_{r(2)}(1)$.
Set $k_{2}=1$, $r(3)=r(2)+k_{2}=1$, and $J_{3}=\Phi_{r(2)}^{k_{2}}\circ F_{r(2)}(J_{2})=\phi_{0}\circ F_{0}(J_{2})$.
The set $J_{3}$ is contained in $A_{r(3)}$. Set $J_{4}=F_{r(3)}(J_{3})$
and $r(4)=r(3)=1$.

From the graph, we see that the sizes of the elements $\{J_{n}\}$
grow as the procedure continues and the element $J_{4}\subset B_{1}$
becomes so large that it intersects some local stable manifolds. This
leads to a contradiction. Therefore, $J$ is not a wandering domain.
\end{example}
Motivated from the example, we study the growth of the horizontal
size and prove the sizes of the elements approach to infinity to obtain
a contradiction. The size is defined as follow.
\begin{definition}[Horizontal and Vertical size]
\label{def:l,h}Assume that $J\subset\mathbb{R}^{2}$. Define the
horizontal size\index{horizontal size|textbf} as\nomenclature[l]{$l$}{Horizontal size}
\[
l(J)=\sup\left\{ \left|x_{1}-x_{2}\right|;(x_{1},y_{1}),(x_{2},y_{2})\in J\right\} =\left|\pi_{x}J\right|
\]
and the vertical size\index{vertical size|textbf} as\nomenclature[h]{$h$}{Vertical size}
\[
h(J)=\sup\left\{ \left|y_{1}-y_{2}\right|;(x_{1},y_{1}),(x_{2},y_{2})\in U\right\} =\left|\pi_{y}J\right|.
\]
If $J$ is compact, the horizontal endpoints\index{horizontal endpoints}
of $J$ are two points in the set that determines $l(J)$.
\end{definition}
Figure \ref{fig:Comparison of lengths} illustrates the horizontal
size and the vertical size of a set $J$. For a Hénon-like map $F\in\mathcal{H}_{\delta}(I^{h}\times I^{v})$,
it follows from the definition that 
\[
h(F(J))=l(J)
\]
for all $J\subset I^{h}\times I^{v}$.

For simplicity, we start from a closed subset $J$ of a wandering
domain such that $\overline{\text{int}(J)}=J$. Then consider the
$J$-closest approach $\left\{ J_{n}\right\} _{n\geq0}$ instead to
ensure the horizontal endpoints exist. Note that the sequence element
$J_{n}$ is also a subset of a wandering domain of $F_{r(n)}$. For
elements in a closest approach, set $l_{n}=l(J_{n})$ and $h_{n}=h(J_{n})$.
Our final goal is to show that the horizontal size $l_{n}$ approaches
to infinity and hence wandering domains cannot exist.

\section{\label{sec:degenerate case}{*}The Degenerate Case\index{Hénon-like map!degenerate}}

In this chapter, we study the relationship between the unimodal renormalization\index{unimodal maps}
and the Hénon renormalization by identifying a unimodal map as a degenerate
Hénon-like map. The main goal is to present a short proof for the
nonexistence of wandering intervals for an infinitely renormalizable
unimodal map at the end of this chapter. It is well known that a unimodal
map (under some regularity condition) does not have wandering interval
\cite{de1988one,de1989structure,lyubich1989non,blokh1989non,martens1992julia}.
 Here, we give a different proof by using the Hénon renormalization
instead of the unimodal renormalization. The expansion argument introduced
in the proof motivates the proof for the non-degenerate case.

\subsection{Local stable manifolds and partition}

First, we adopt the notations from unimodal maps and Hénon-like maps.
Let $F$ be a degenerate Hénon-like map
\[
F(x,y)=(f(x),x).
\]
We use the super-scripts ``$u$'' and ``$h$'' to distinguish
the difference between the notations for unimodal maps and Hénon-like
maps to avoid confusion. For example, $p^{u}(-1)=-1$ and $p^{u}(0)$
are the fixed points of $f$; $p^{h}(-1)=(-1,-1)$ and $p^{h}(0)$
are the saddle fixed points of $F$; $A^{u},B^{u},C^{u}\subset I$
is the partition defined for $f$; $A^{h},B^{h},C^{h}\subset I^{h}\times I^{v}$
is the partition defined for $F$. 

The next lemma gives the relations between the local stable manifolds
for the degenerate Hénon-like map with the fixed points and their
preimages for the unimodal maps. Recall that $p^{(1)}$ and $p^{(2)}$
are the points such that $f(p^{(2)})=p^{(1)}$, $f(p^{(1)})=p^{u}(0)$,
and $p^{(1)}<p^{u}(0)<p^{(2)}$ (Definition \ref{def:Partition of I});
$W^{0}(-1)$ and $W^{2}(-1)$ are the local stable manifolds of $p^{h}(-1)$
(Definition \ref{def:Local stable manifold W(-1)}); $W^{0}(0),W^{1}(0),W^{2}(0)$
are the local stable manifolds of $p^{h}(0)$ (Definition \ref{def:Local stable manifold W(0)}).
\begin{lemma}[Fixed points and local stable manifolds]
\index{local stable manifold!degenerate|textit}Assume that $F\in\mathcal{H}_{\delta}(I^{h}\times I^{v})$
is a degenerate Hénon-like map. Then
\begin{enumerate}
\item $p^{h}(j)=(p^{u}(j),p^{u}(j))$ for $j=-1,0$, 
\item the local stable manifold $W^{0}(j)$ is the vertical line $x=p^{u}(j)$
for $j=-1,0$, 
\item the local stable manifold $W^{2}(-1)$ is the vertical line $x=\hat{p}^{u}(-1)$, 
\item the local stable manifold $W^{1}(0)$ is the vertical line $x=p^{(1)}$,
and 
\item the local stable manifold $W^{2}(0)$ is the vertical line $x=p^{(2)}$.
\end{enumerate}
\end{lemma}
It follows from the definition that the partition for unimodal maps
and degenerate Hénon-like maps coincide.
\begin{corollary}[Partition]
\index{partition!degenerate|textit}Assume that $F\in\mathcal{H}_{\delta}(I^{h}\times I^{v})$
is a degenerate Hénon-like map. Then $A^{h}=A^{u}\times I^{v}$, $B^{h}=B^{u}\times I^{v}$,
$C^{h}=C^{u}\times I^{v}$, and $D^{h}=I\times I^{v}$.
\end{corollary}

\subsection{Renormalization operator}

Next we compare the renormaliztion operator for Hénon-like maps with
the renormalization operator  for unimodal maps. Recall the definitions
of the rescaling maps. For a degenerate renormalizable Hénon-like
map $F$, the rescaling map has the form $\phi=\Lambda\circ H$ where
$\Lambda(x,y)=(s^{h}(x),s^{h}(y))$, $s^{h}$ is the affine rescaling
map, and $H(x,y)=(f(x),y)$ is the nonlinear rescaling term. The renormalized
map is $RF=\phi\circ F^{2}\circ\phi^{-1}$. For a renormalizable unimodal
map $f$, $s^{u}$ is the affine rescaling and $\uRenormalizeC f=s^{u}\circ f^{2}\circ(s^{u})^{-1}$
is the renormalization about the critical point.

Although the Hénon renormaliation rescales the first return map around
the ``critical value'', the operation acts like the unimodal renormalization
which rescales the first return map around the ``critical point''.
This is because of the nonlinear rescaling term $H$ for the Hénon-renormalization.
Let $A_{0}^{h}$,$B_{0}^{h}$,$C_{0}^{h}$ be the partition for $F$
and $D_{1}^{h}$ be the domain for $RF$. The rescaling map $\phi(x,y)=(s^{h}\circ f(x),s^{h}(y))$
maps $C_{0}^{h}$ to $D_{1}^{h}$. This means the operation $f$ in
the $x$-component maps $C_{0}^{u}$ to $B_{0}^{u}$ and the affine
map $s^{h}$ maps $C_{0}^{u}$ back to the unit size $I$. Thus, the
two affine maps $s^{u}$ and $s^{h}$ are the same and 
\[
H\circ F^{2}\circ H^{-1}(x,y)=(f^{2}|_{B_{0}^{u}}(x),x)
\]
is the first return map on $B_{0}^{h}$. Therefore, the two renormalizations
coincide
\[
RF(x,y)=(s^{u}\circ f^{2}\circ(s^{u})^{-1}(x),x)=(\uRenormalizeC f(x),x).
\]
 This also explains why $R^{n}F$ converges to the fixed point $g$
of the unimodal renormalization operator.

The observation is summarized as follows.
\begin{lemma}[Renormalization operator]
\label{lem:degenerate-renormalization operator}\index{renormalization!degenerate|textit}Assume
that $F\in\mathcal{H}_{\delta}(I^{h}\times I^{v})$ is a degenerate
Hénon-like map. Then $F$ is Hénon renormalizable if and only if $f$
is unimodal renormalizable. When the map is renormalizable, we have
\begin{enumerate}
\item $s^{h}=s^{u}$ and
\item $RF(x,y)=(\uRenormalizeC f(x),x)$.
\end{enumerate}
In fact, if $F$ is infinitely renormalizable, then the affine term
$\Lambda_{n}:B_{n}(j)\rightarrow B_{n+1}(j-1)$ is a bijection for
all $n\geq0$ and $j\geq1$ where $B_{n}(0)\equiv A_{n}\cup W_{n}^{2}(0)\cup C_{n}$.
\end{lemma}
From now on, we remove the super-script from $s$ because the maps
are the same.

For an infinitely renormalizable Hénon-like map, we also adopt the
subscript used for the renormalization levels to the degenerate case.
Assume that a degenerate Hénon-like map $F(x,y)=(f(x),x)$ is infinitely
renormalizable. Let $F_{n}=R^{n}F$ and $f_{n}=\uRenormalizeC^{n}f$.
Then $F_{n}(x,y)=(f_{n}(x),x)$ by the second property of Lemma \ref{lem:degenerate-renormalization operator}.

Next proposition proves an important equality which will be used to
prove the nonexistence of wandering intervals for infinitely renormalizable
unimodal maps. The expansion argument comes from this proposition.
\begin{proposition}[Rescaling trick]
\index{rescaling trick}\label{prop:degenerate-rescaling trick}Assume
that $f\in\mathcal{I}$. Then 
\[
(s_{n+j-1}\circ f_{n+j-1})\circ\cdots\circ(s_{n}\circ f_{n})\circ f_{n}=f_{n+j}\circ s_{n+j-1}\circ\cdots\circ s_{n}
\]
for all integers $n\geq0$ and $j\geq0$.
\end{proposition}
\begin{proof}
Prove by induction on $j$. It is clear that the equality holds when
$j=0$.

Assume that the equality holds for some $j$. Then 
\begin{align*}
 & (s_{n+j}\circ f_{n+j})\circ(s_{n+j-1}\circ f_{n+j-1})\circ\cdots\circ(s_{n}\circ f_{n})\circ f_{n}\\
= & (s_{n+j}\circ f_{n+j})\circ f_{n+j}\circ s_{n+j-1}\circ\cdots\circ s_{n}\\
= & (s_{n+j}\circ f_{n+j}\circ f_{n+j}\circ s_{n+j}^{-1})\circ s_{n+j}\circ s_{n+j-1}\circ\cdots\circ s_{n}\\
= & f_{n+j+1}\circ s_{n+j}\circ s_{n+j-1}\circ\cdots\circ s_{n}.
\end{align*}
Therefore, the lemma is proved by induction.
\end{proof}

By Lemma \ref{lem:degenerate-renormalization operator} and Proposition
\ref{lem:degenerate-renormalization operator}, we get
\begin{corollary}
\label{cor:degenerate-Rescaling trick}Assume that $F\in\mathcal{I}_{\delta}(I^{h}\times I^{v})$
is a degenerate Hénon-like map. Then 
\[
\Phi_{n}^{j}\circ F_{n}=F_{n+j}\circ\Lambda_{n+j-1}\circ\cdots\circ\Lambda_{n}
\]
for all integers $n\geq0$ and $j\geq0$.
\end{corollary}

\subsection{Nonexistence of wandering intervals}

In this section, we present a proof for the nonexistence of wandering
intervals for infinitely renormalizable unimodal maps by identifying
a unimodal map as a degenerate Hénon-like map and using the Hénon
renormalization. A wandering interval\index{wandering interval|textbf}
is a nonempty interval such that its orbit does not intersect itself
and the omega limit set does not contain a periodic point.
\begin{proposition}
\label{No wandering interval (degenerate)}\index{expansion argument!degenerate|textit}A
infinitely renormalizable unimodal map does not have a wandering interval\index{wandering interval|textit}.
\end{proposition}
\begin{proof}
(Sketch of the proof) Prove by contradiction. Assume that $f$ is
an infinitely renormalizable unimodal map that has a wandering interval
$J^{u}$. Without lose of generality, we may assume that the map is
close to the fixed point $g$ of the renormalization operator because
the sequence of renormalizations $\uRenormalizeC^{n}f$ converges
to $g$ as $n$ approaches to infinity. Let $F=(f,x)$. Then $F$
is a degenerate infinitely renormalizable Hénon-like map. Assume that
$J^{u}\subset I$ is a wandering interval of $f_{0}$. Let $J^{h}=J^{u}\times\{0\}$
and $J_{n}\subset A_{r(n)}\cup B_{r(n)}$ be the $J^{h}$-closest
approach. The projection $\pi_{x}J_{n}$ is a wandering interval of
$f_{r(n)}$ and the horizontal size $l_{n}$ is the length of the
projection.

If $J_{n}\subset A_{r(n)}$, then
\[
l_{n+1}>El_{n}
\]
for some constant $E>1$ because $g$ is expanding on $A(g)$ by Proposition
\ref{prop:Derivative of g on A} and the map $f_{r(n)}$ is close
to $g$.

If $J_{n}\subset B_{r(n)}(k_{n})$, then $J_{n+1}=F_{r(n+1)}\circ\Lambda_{r(n)+k_{n}-1}\circ\cdots\circ\Lambda_{r(n)}(J_{n})$
by Corollary \ref{cor:degenerate-Rescaling trick}. The rescaling
maps $\Lambda_{r(n)},\cdots,\Lambda_{r(n)+k_{n}-1}$ expands the horizontal
size. The map $F_{r(n+1)}$ also expands the horizontal size because
$\Lambda_{r(n)+k_{n}-1}\circ\cdots\circ\Lambda_{r(n)}(J_{n})\subset A_{r(n+1)}\cup C_{r(n+1)}$,
$g$ is expanding on $A(g)\cup C(g)$ by Proposition \ref{prop:Derivative of g on A},
and the map $f_{r(n+1)}$ is close to $g$. Thus, 
\[
l_{n+1}>E'l_{n}
\]
for some constant $E'>1$.

This shows that the horizontal size $l_{n}$ approaches to infinity
which yields a contradiction. Therefore, wandering intervals cannot
exist.
\end{proof}

In the proof, we showed that the horizontal size expands at a definite
size in each size. This motivates the proof for the non-degenerate
case. In the remain part of the article, we will study the growth
rate or contraction rate of the horizontal size. In Chapter \ref{sec:Good region},
we will show this is also true for the non-degenerate case under some
conditions.

\section{\label{sec:Good and Bad region}The Good\index{good region} Region
and the Bad Region\index{bad region}}

In this chapter, we group the sub-partitions of $\left\{ B_{n}(j)\right\} _{j=1}^{\infty}$
and $\left\{ C_{n}(j)\right\} _{j=1}^{\infty}$ into two regions by
the following phenomena. Assume that $\left\{ J_{n}\right\} _{n=0}^{\infty}$
is the $J$-closest approach and $J_{n}\in B_{r(n)}(k_{n})$ for some
$n$.

When $J_{n}$ is far from the center, i.e. $k_{n}$ is small, the
topology of $B_{r(n)}(k_{n})$ and the dynamics of $F_{r(n)}$ behave
like the unimodal case. It can be proved that the boundaries of $B_{n}(k_{n})$ are vertical
graphs of small Lipschitz constant.  Also, studying the iteration of horizontal endpoints of a wandering
domain provides a good approximation to the expansion rate of the
horizontal size. Chapter \ref{sec:Good region} will show the expansion
argument works in this case. This group is called ``the good region''.

However, when $J_{n}$ is close to the center, i.e. $k_{n}$ is large,
the topology of $B_{r(n)}(k_{n})$ and dynamics of $F_{r(n)}$ is
different from the unimodal case. In this group, the two boundary
local stable manifolds of $C_{r(n)}(k_{n})$ are so close to the tip
$\tau_{r(n)}$ that they only intersect the image $F_{r(n)}(D_{r(n)})$
once. Thus, the preimage of a local stable manifold becomes concave
and hence $B_{r(n)}(j)$ becomes an arch-like domain. See the left
graph of Figure \ref{fig:Sequence of wandering domain} and the next
paragraph. Also, the expansion argument fails. The iteration of horizontal
endpoints of a wandering domain fails to provide an approximation
for the change rate of the horizontal size. In fact, we show that
the $x$-coordinate of the two iterated horizontal endpoints can be
as close as possible in the next paragraph. This group is called ``the
bad region''.

The vertical line argument\index{vertical line argument} in Figure
\ref{fig:vertical line argument} explains why the expansion argument
fails in the bad region. The construction is as follows. Draw a vertical
line (dashed vertical line in the figure) so close to the tip that
its intersection with the image of $F_{r(n)}$ only has one component.
Take the preimage of the intersection. Unlike the case in the good
region, the preimage is not a vertical graph. Instead, it is a concave
curve that has a $y$-extremal point close to the center of the domain.
When a sequence element $J_{n}$ is in the bad region, it is close
to the center. The size of $J_{n}$ is small because the size of bad
region is small and hence $F_{r(n)}$ acts like a linear map. If the
line $\overleftrightarrow{UV}$ connecting horizontal endpoints $U$
and $V$ of $J_{n}$ is also parallel to the concave curve as in Figure
\ref{fig:vertical argument-J}, the image of the horizontal endpoints
will also be parallel to the vertical line as Figure \ref{fig:vertical argument-F(J)}
shows. In this case, the iterated horizontal endpoints forms a vertical
line that has no $x$-displacement. Therefore, the horizontal size
shrinks and the horizontal endpoints fail to estimate the change of
horizontal size when a sequence element enters the bad region.

\begin{figure}
\begin{centering}
\includegraphics[scale=0.5]{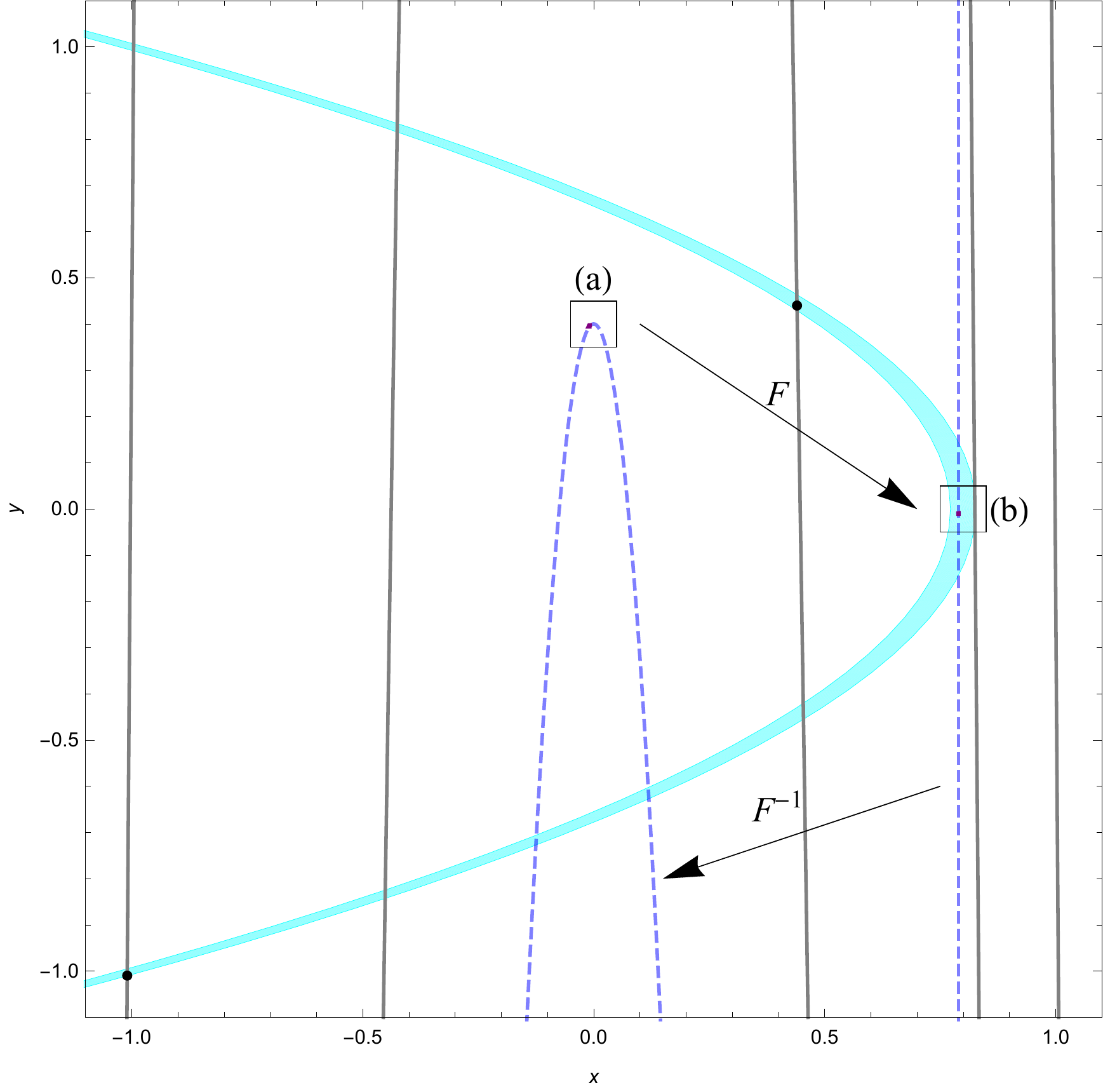}
\par\end{centering}
\resizebox{\columnwidth}{!}{%
\begin{centering}
\subfloat[\label{fig:vertical argument-J}]{\includegraphics[scale=0.5]{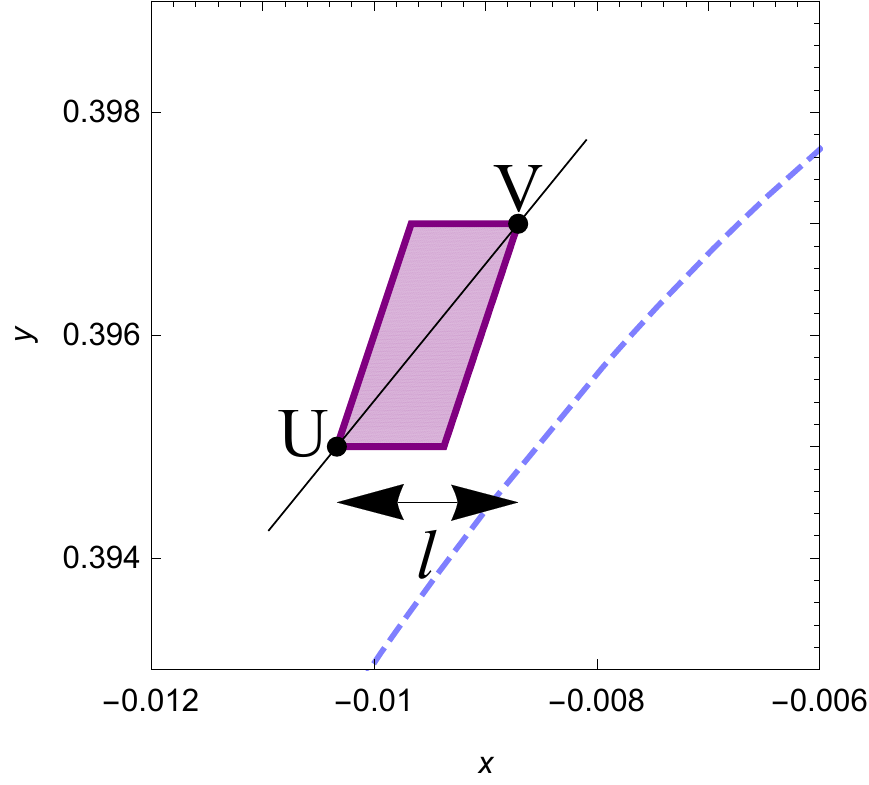}

}\subfloat[\label{fig:vertical argument-F(J)}]{\includegraphics[scale=0.5]{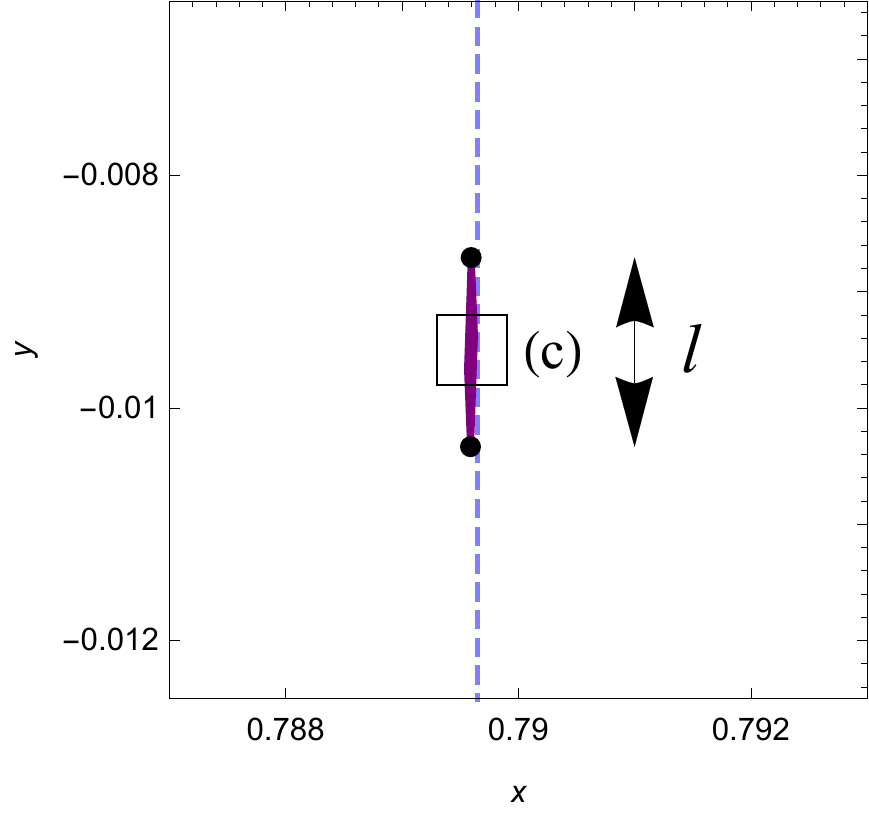}

}\subfloat[\label{fig:vertical argument-area}Zoom in (b).]{\includegraphics[scale=0.5]{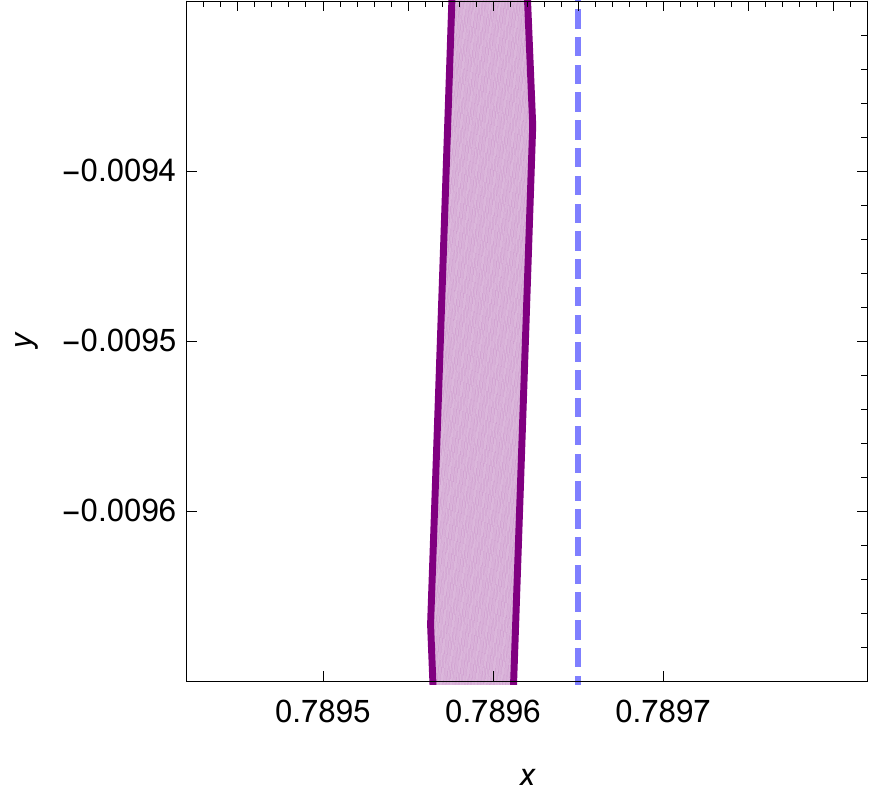}

}
\par\end{centering}
}

\caption{\label{fig:vertical line argument}Vertical line argument. The scales
in (a) and (b) are chosen to be the same for the reader to compare
the change of horizontal size.}
\end{figure}

From the vertical line argument, it becomes crucial to group the sets
$\left\{ C_{n}(j)\right\} _{j\geq1}$ by how close the set to the
tip is. The size of the image is $\left\Vert \epsilon_{n}\right\Vert $.
To avoid a vertical line intersecting the image only once, the line
has to be $\left\Vert \epsilon_{n}\right\Vert $ away from the tip.
This motivates the definition of the boundary sequence $\left\{ K_{n}\right\} _{n\geq0}$,
the good region, and the bad region.
\begin{definition}[Good and Bad Regions]
\label{def:Good and bad region}\nomenclature[K_n]{$K_{n}$}{Boundary for good and bad regions}\label{def:K_n}Fixed
$b>0$. Assume that $\overline{\epsilon}>0$ is sufficiently small
so that Proposition \ref{prop:Geometric properties of W_n(j)} holds
and $F\in\hat{\mathcal{I}}_{\delta}(I^{h}\times I^{v},\overline{\epsilon})$.
For each $n\geq0$, define $K_{n}=K_{n}(b)$ to be the largest positive
integer such that 
\[
\left|\pi_{x}z-\pi_{x}\tau_{n}\right|>b\left\Vert \epsilon_{n}\right\Vert 
\]
for all $z\in W_{n}^{0}(K_{n})\cap\left(I^{h}\times I^{h}\right)$.

The set $C_{n}(j)$ (resp. $B_{n}(j)$) is in the good region\index{good region|textbf}
if $j\leq K_{n}$; in the bad region\index{bad region|textbf} if
$j>K_{n}$. The sequence $K_{n}$ is called the boundary for the good
region and the bad region. See Figure \ref{fig:Good and Bad region}.
\end{definition}
\begin{remark}
It is enough to consider the subdomain $I^{h}\times I^{h}\subset I^{h}\times I_{n}^{v}$
in the definition because $F_{n}(D_{n})\subset I^{h}\times I^{h}$.
\end{remark}
\begin{remark}
Here the boundary sequence $\left\{ K_{n}\right\} _{n\geq0}$ depends
on the constant $b$ and we make $b$ flexible. In the theorems of
this chapter, we will prove that each property holds for all $b$
that satisfies certain constraints. At the end, we will fix a constant
$b$ sufficiently large that makes all theorems work. So the sequence
$\left\{ K_{n}\right\} _{n\geq0}$ will be fixed in the remaining
article.
\end{remark}
\begin{remark}
One can see that the bad region is a special feature for the Hénon
case. For the degenerate case, $\epsilon_{n}=0$ and hence $K_{n}=\infty$.
This means that there are no bad region for the degenerate case.

\begin{figure}
\begin{centering}
\includegraphics{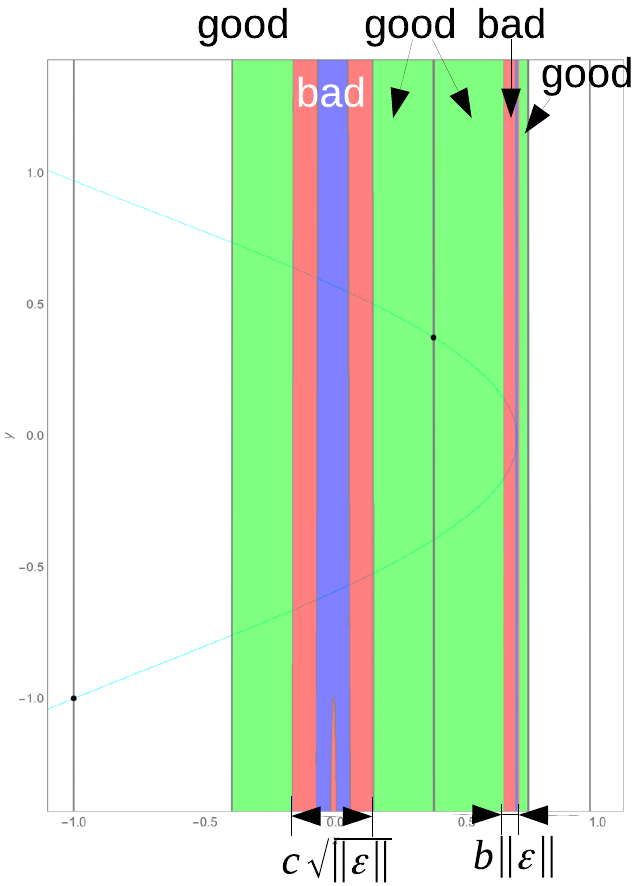}
\par\end{centering}
\caption{\label{fig:Good and Bad region}Good and bad region. The shaded areas
represent the rescaling levels 1, 2, and below from light to dark.
In this example, one can see a tiny light area on the center bottom
part of the graph. This is because $C^{r}(2)$ intersects the image
$F(D)$. Hence, the boundary is $K=2$ and the good region contains
the lightest part and the bad region contains the two darker parts.}
\end{figure}
\end{remark}
Our goal in this chapter is to study the geometric properties for
the good region and the bad region. The main theorem is stated as
follows.
\begin{proposition}[Geometric properties for the good region and the bad region]
\label{prop:Good and Bad Regions}\index{good region|textit}\index{bad region|textit}Given
$\delta>0$ and $I^{v}\supset I^{h}\Supset I$. There exists $\overline{\epsilon}>0$,
$\overline{b}>0$, and $c>1$ such that for all $F\in\hat{\mathcal{I}}_{\delta}(I^{h}\times I^{v},\overline{\epsilon})$
and $b>\overline{b}$ the following properties hold for all $n\geq0$:

The boundary $K_{n}$ is bounded by 
\begin{equation}
\frac{1}{c}\frac{1}{\sqrt{b\left\Vert \epsilon_{n}\right\Vert }}\leq\lambda^{K_{n}}\leq c\frac{1}{\sqrt{b\left\Vert \epsilon_{n}\right\Vert }}.\label{eq:K_n bounds}
\end{equation}
For the good region $1\leq j\leq K_{n}$, we have

\begin{enumerate}
\item $C_{n}^{r}(j)\cap F_{n}(D_{n})=\phi$,
\item $\left|\pi_{x}z-\pi_{x}\tau_{n}\right|>b\left\Vert \epsilon_{n}\right\Vert $
for all $z\in C_{n}(j)\cap F_{n}(D_{n})$,
\item $\left|\pi_{x}z-v_{n}\right|>\frac{1}{c}\sqrt{b\left\Vert \epsilon_{n}\right\Vert }$
for all $z\in B_{n}(j)$, and 
\item $\frac{1}{c}\left(\frac{1}{\lambda}\right)^{2j}<\left|\pi_{x}z-\pi_{x}\tau_{n}\right|<c\left(\frac{1}{\lambda}\right)^{2j}$
for all $z\in C_{n}(j)\cap F_{n}(D_{n})$.
\end{enumerate}
For the bad region $j>K_{n},$ we have

\begin{enumerate}
\item $\left|\pi_{x}z-\pi_{x}\tau_{n}\right|<cb\left\Vert \epsilon_{n}\right\Vert $
for all $z\in C_{n}(j)\cap F_{n}(D_{n})$ and
\item $\left|\pi_{x}z-v_{n}\right|<c\sqrt{b\left\Vert \epsilon_{n}\right\Vert }$
for all $z\in B_{n}(j)$.
\end{enumerate}
\end{proposition}
This proposition will be proved by the lemmas in this chapter.

First, we estimate the bounds for the boundary $K_{n}$.
\begin{lemma}
\label{lem:K_n bounds}Given $\delta>0$ and $I^{v}\supset I^{h}\Supset I$.
There exists $\overline{\epsilon}>0$, $\overline{b}>0$, and $c>1$
such that for all $F\in\hat{\mathcal{I}}_{\delta}(I^{h}\times I^{v},\overline{\epsilon})$
and $b>\overline{b}$ we have
\[
\frac{1}{c}\frac{1}{\sqrt{b\left\Vert \epsilon_{n}\right\Vert }}\leq\lambda^{K_{n}}\leq c\frac{1}{\sqrt{b\left\Vert \epsilon_{n}\right\Vert }}
\]
for all $n\geq0$.
\end{lemma}
\begin{proof}
In the proof, we apply Proposition \ref{prop:Geometric properties of W_n(j)}
to relate the rescaling level $K_{n}$ with the $x$-coordinate of
the local stable manifold. Assume that $\overline{\epsilon}>0$ is
small enough.

For the upper bound, by the definition of $K_{n}$ and Proposition
\ref{prop:Geometric properties of W_n(j)}, we have 
\[
c'\left(\frac{1}{\lambda}\right)^{2K_{n}}\geq\left|z_{n}^{(0)}(K_{n})-\tau_{n}\right|\geq b\left\Vert \epsilon_{n}\right\Vert .
\]
for some constant $c'>1$. Thus, 
\[
\lambda^{K_{n}}\leq\sqrt{\frac{c'}{b}}\frac{1}{\sqrt{\left\Vert \epsilon_{n}\right\Vert }}.
\]

For the lower bound, by the definition of $K_{n}$, there exists $z\in W_{n}^{0}(K_{n}+1)\cap I^{h}\times I^{h}$
such that $\left|\pi_{x}z-\pi_{x}\tau_{n}\right|\leq b\left\Vert \epsilon_{n}\right\Vert $.
Apply Proposition \ref{prop:Geometric properties of W_n(j)} , we
get 
\begin{eqnarray*}
\frac{1}{c'}\left(\frac{1}{\lambda}\right)^{2(K_{n}+1)} & \leq & \left|z_{n}^{(0)}(K_{n}+1)-\tau_{n}\right|\\
 & \leq & \left|\pi_{x}z-\pi_{x}\tau_{n}\right|+\left|\pi_{x}z-\pi_{x}z_{n}^{(0)}(K_{n}+1)\right|\\
 & \leq & b\left\Vert \epsilon_{n}\right\Vert +c\left|I^{h}\right|\left\Vert \epsilon_{n}\right\Vert .
\end{eqnarray*}
for some constant $c>0$. We solved 
\[
\lambda^{K_{n}}\geq\frac{1}{\lambda}\sqrt{\frac{1}{c'(b+c\left|I^{h}\right|)}}\frac{1}{\sqrt{\left\Vert \epsilon_{n}\right\Vert }}\geq\frac{1}{\lambda}\sqrt{\frac{1}{2c'b}}\frac{1}{\sqrt{\left\Vert \epsilon_{n}\right\Vert }}
\]
when $b\geq c\left|I^{h}\right|$.
\end{proof}

\subsection{Properties for the good region\index{good region|(}}

To prove the properties, the strategy is to first estimate the $x$-location
of the local stable manifolds $W_{n}^{t}(j)$. Since the local stable
manifolds bounds $C_{n}(j)$, the properties for $C_{n}(j)$ follows.
Properties for $B_{n}(j)$ follows directly by the quadratic estimations
from Lemma \ref{lem:Quadratic center} and Proposition \ref{prop:Bound for f(v)-tau}.
\begin{lemma}
\label{lem:W location}Given $\delta>0$ and $I^{v}\supset I^{h}\Supset I$.
There exists $\overline{\epsilon}>0$, $\overline{b}>0$, and $c>1$
such that for all $F\in\hat{\mathcal{I}}_{\delta}(I^{h}\times I^{v},\overline{\epsilon})$
and $b>\overline{b}$ we have 
\[
\frac{1}{c}\left(\frac{1}{\lambda}\right)^{2j}\leq\left|\pi_{x}z-\pi_{x}\tau_{n}\right|\leq c\left(\frac{1}{\lambda}\right)^{2j}
\]
 for all $z\in W_{n}^{t}(j)\cap\left(I^{h}\times I^{h}\right)$ with
$t\in\{0,2\}$, $0\leq j\leq K_{n}$, and $n\geq0$.
\end{lemma}
\begin{proof}
In the proof, we apply Proposition \ref{prop:Geometric properties of W_n(j)}
to relate the rescaling level $j$ with the $x$-coordinate of the
local stable manifold $W_{n}^{t}(j)$. Assume that $\overline{\epsilon}>0$
is sufficiently small. We prove the case for $t=0$ and the other
case is similar.

Let $z\in W_{n}^{0}(j)\cap\left(I^{h}\times I^{h}\right)$ and $b>\overline{b}$
where $\overline{b}$ is given by Lemma \ref{lem:K_n bounds}.

To prove the lower bound, apply Proposition \ref{prop:Geometric properties of W_n(j)},
we get 
\begin{eqnarray*}
\left|\pi_{x}z-\pi_{x}\tau_{n}\right| & \geq & \left|z_{n}^{(0)}(j)-\tau_{n}\right|-\left|\pi_{x}z-\pi_{x}z_{n}^{(0)}(j)\right|\\
 & \geq & \frac{1}{c}\left(\frac{1}{\lambda}\right)^{2j}-c\left\Vert \epsilon_{n}\right\Vert \left|I^{h}\right|\\
 & \geq & \left(\frac{1}{c}-c\left|I^{h}\right|\left\Vert \epsilon_{n}\right\Vert \lambda^{2K_{n}}\right)\left(\frac{1}{\lambda}\right)^{2j}
\end{eqnarray*}
for some constant $c>1$. By Lemma \ref{lem:K_n bounds}, there exists
$c'>1$ such that 
\begin{eqnarray*}
\left|\pi_{x}z-\pi_{x}\tau_{n}\right| & \geq & \left(\frac{1}{c}-\frac{cc^{\prime2}\left|I^{h}\right|}{b}\right)\left(\frac{1}{\lambda}\right)^{2j}\\
 & \geq & \frac{1}{2c}\left(\frac{1}{\lambda}\right)^{2j}.
\end{eqnarray*}
Here we assume that $b\geq2c^{2}c^{\prime2}\left|I^{h}\right|$.

Similarly, to prove the upper bound, apply Proposition \ref{prop:Geometric properties of W_n(j)},
we get 
\begin{eqnarray*}
\left|\pi_{x}z-\pi_{x}\tau_{n}\right| & \le & \left|z_{j}^{(0)}(j)-\tau_{n}\right|+\left|\pi_{x}z-\pi_{x}z_{n}^{(0)}(j)\right|\\
 & \leq & c\left(\frac{1}{\lambda}\right)^{2j}+c\left\Vert \epsilon_{n}\right\Vert \left|I^{h}\right|\\
 & \leq & \left(c+c\left|I^{h}\right|\left\Vert \epsilon_{n}\right\Vert \lambda^{2K_{n}}\right)\left(\frac{1}{\lambda}\right)^{2j}.
\end{eqnarray*}
By Lemma \ref{lem:K_n bounds}, we get
\begin{eqnarray*}
\left|\pi_{x}z-\pi_{x}\tau_{n}\right| & \leq & \left(\frac{1}{c}+\frac{cc^{\prime2}\left|I^{h}\right|}{b}\right)\left(\frac{1}{\lambda}\right)^{2j}\\
 & \leq & \frac{3}{2c}\left(\frac{1}{\lambda}\right)^{2j}.
\end{eqnarray*}
\end{proof}

We prove the first property for the good region.
\begin{lemma}
\label{lem:C^r is not in the image of F}Given $\delta>0$ and $I^{v}\supset I^{h}\Supset I$.
There exists $\overline{\epsilon}>0$ and $\overline{b}>0$ such that
for all $F\in\hat{\mathcal{I}}_{\delta}(I^{h}\times I^{v},\overline{\epsilon})$
and $b>\overline{b}$ we have 
\[
C_{n}^{r}(j)\cap F_{n}(D_{n})=\phi
\]
 for all $1\leq j\leq K_{n}$ and $n\geq0$.
\end{lemma}
\begin{proof}
Since $\overline{\cup_{j=1}^{K_{n}}C_{n}^{r}(j)}$ is bounded by the
local manifolds $W_{n}^{2}(K_{n})$ and $W_{n}^{2}(0)$, it suffices
to prove the local stable manifold $W_{n}^{2}(K_{n})$ is far away
form the image.

We have 
\[
f_{n}(v_{n})-\left\Vert \epsilon_{n}\right\Vert \leq\sup_{z'\in D_{n}}h_{n}(z')=\sup_{z'\in D_{n}}\left(f_{n}(\pi_{x}z')+\epsilon_{n}(z')\right)\leq f_{n}(v_{n})+\left\Vert \epsilon_{n}\right\Vert .
\]
Apply Proposition \ref{prop:Bound for f(v)-tau}, Lemma \ref{lem:K_n bounds},
and Lemma \ref{lem:W location}, there exists constants $c>0$ and
$a>1$ such that 
\begin{eqnarray*}
\pi_{x}z-\sup_{z'\in D_{n}}h_{n}(z') & \geq & (\pi_{x}z-\pi_{x}\tau_{n})-\left|\pi_{x}\tau_{n}-f_{n}(v_{n})\right|-\left|f_{n}(v_{n})-\sup_{z'\in D_{n}}h_{n}(z')\right|\\
 & \geq & \frac{1}{a}\left(\frac{1}{\lambda}\right)^{2K_{n}}-c\left\Vert \epsilon_{n}\right\Vert -\left\Vert \epsilon_{n}\right\Vert \\
 & \geq & \left(\frac{b}{a^{3}}-c-1\right)\left\Vert \epsilon_{n}\right\Vert 
\end{eqnarray*}
for all $z\in W_{n}^{2}(K_{n})\cap\left(I^{h}\times I^{h}\right)$.
The coefficient is positive when $b>0$ is large enough. Consequently,
$C_{n}^{r}(j)\cap F_{n}(D_{n})=\phi$ for all $1\leq j\leq K_{n}$.
\end{proof}

By the previous lemma, it is enough to prove the rest of the properties
for the left component $C_{n}^{l}(j)$. The second property follows
directly from the definition of $K_{n}$.
\begin{lemma}
\label{lem:Good region away from the tip}Given $\delta>0$ and $I^{v}\supset I^{h}\Supset I$.
There exists $\overline{\epsilon}>0$ and $\overline{b}>0$ such that
for all $F\in\hat{\mathcal{I}}_{\delta}(I^{h}\times I^{v},\overline{\epsilon})$
and $b>\overline{b}$ we have 
\[
\left|\pi_{x}z-\pi_{x}\tau_{n}\right|>b\left\Vert \epsilon_{n}\right\Vert 
\]
 for all $z\in C_{n}(j)\cap F_{n}(D_{n})$ with $1\leq j\leq K_{n}$
and $n\geq0$.
\end{lemma}
\begin{proof}
By Lemma \ref{lem:C^r is not in the image of F}, only the left component
$C_{n}^{l}(j)$ intersects the image. Also the set $C_{n}^{l}(j)$
is bounded by the local stable manifolds $W_{n}^{0}(j-1)$ and $W_{n}^{0}(j)$.
Thus, the lemma follows from the definition of $K_{n}$.
\end{proof}

The third property of the good regions follows by the previous proposition
and the quadratic estimates from Lemma \ref{lem:Quadratic center}.
\begin{corollary}
\label{cor:Good region away from the center}Given $\delta>0$ and
$I^{v}\supset I^{h}\Supset I$. There exists $\overline{\epsilon}>0$,
$\overline{b}>0$, and $c>0$ such that for all $F\in\hat{\mathcal{I}}_{\delta}(I^{h}\times I^{v},\overline{\epsilon})$
and $b>\overline{b}$ so that the following property hold for all
$n\geq0$:

If $z\in I^{B}\times I_{n}^{v}$ satisfies $\left|h_{n}(z)-\pi_{x}\tau_{n}\right|\geq b\left\Vert \epsilon_{n}\right\Vert $,
then 
\begin{equation}
\left|\pi_{x}z-v_{n}\right|\geq c\sqrt{b\left\Vert \epsilon_{n}\right\Vert }.\label{eq:lower bound for good region in B}
\end{equation}
In particular, (\ref{eq:lower bound for good region in B}) holds
for all $z\in B_{n}(j)$ with $1\leq j\leq K_{n}$.
\end{corollary}
\begin{proof}
Assume that $b>0$ is large enough such that Lemma \ref{lem:Good region away from the tip}
holds. Let $z\in I^{B}\times I_{n}^{v}$ be such that $\left|h_{n}(z)-\pi_{x}\tau_{n}\right|>b\left\Vert \epsilon_{n}\right\Vert $.
Apply Proposition \ref{prop:Bound for f(v)-tau}, we get 
\begin{eqnarray*}
\left|f_{n}(\pi_{x}z)-f_{n}(v_{n})\right| & \geq & \left|h_{n}(z)-\pi_{x}\tau_{n}\right|-\left|f_{n}(\pi_{x}z)-h_{n}(z)\right|-\left|\pi_{x}\tau_{n}-f_{n}(v_{n})\right|\\
 & \geq & (b-1-c)\left\Vert \epsilon_{n}\right\Vert \\
 & > & \frac{b}{2}\left\Vert \epsilon_{n}\right\Vert 
\end{eqnarray*}
for some $c>0$ when $b>2(1+c)$.

Moreover, by the quadratic estimates from Lemma \ref{lem:Quadratic center},
there exists $a>1$ such that 
\[
\left|f_{n}(\pi_{x}z)-f_{n}(v_{n})\right|\leq\frac{a}{2}(\pi_{x}z-v_{n})^{2}
\]
for all $n\geq0$ when $\overline{\epsilon}>0$ is small enough. Therefore,
\[
\left|\pi_{x}z-v_{n}\right|\geq\sqrt{\frac{1}{a}}\sqrt{b\left\Vert \epsilon_{n}\right\Vert }.
\]
\end{proof}

The fourth property for the good region gives an estimate for the
$x$-location of $C_{n}^{l}(j)$ in terms of the rescaling level $j$.
To prove the property, we use the boundary stable manifolds $W_{n}^{0}(j-1)$
and $W_{n}^{0}(j)$ to estimate the $x$-location of $C_{n}^{l}(j)$.
\begin{corollary}
Given $\delta>0$ and $I^{v}\supset I^{h}\Supset I$. There exists
$\overline{\epsilon}>0$, $\overline{b}>0$, and $c>1$ such that
for all $F\in\hat{\mathcal{I}}_{\delta}(I^{h}\times I^{v},\overline{\epsilon})$
and $b>\overline{b}$ we have 
\[
\frac{1}{c}\left(\frac{1}{\lambda}\right)^{2j}<\left|\pi_{x}z-\pi_{x}\tau_{n}\right|<c\left(\frac{1}{\lambda}\right)^{2j}
\]
 for all $z\in C_{n}(j)\cap F_{n}(D_{n})$ with $1\leq j\leq K_{n}$
and $n\geq0$.
\end{corollary}
\begin{proof}
We use the property that $C_{n}^{l}(j)$ is bounded by $W_{n}^{0}(j-1)$
and $W_{n}^{0}(j)$ then apply the estimations from the local stable
manifolds Lemma \ref{lem:W location}. Assume that $b>0$ is large
enough such that Lemma \ref{lem:W location} and Lemma \ref{lem:C^r is not in the image of F}
hold.

For all $z\in C_{n}^{l}(j)\cap F_{n}(D_{n})$ with $1\leq j\leq K_{n}$,
there exists $z_{1}\in W_{n}^{0}(j-1)\cap\left(I^{h}\times I^{h}\right)$
and $z_{2}\in W_{n}^{0}(j)\cap\left(I^{h}\times I^{h}\right)$ such
that $\pi_{y}z=\pi_{y}z_{1}=\pi_{y}z_{2}$ since the local stable
manifolds are vertical graphs. From Lemma \ref{lem:W location}, we
obtain 
\[
\frac{1}{c}\left(\frac{1}{\lambda}\right)^{2j}\leq\left|\pi_{x}z_{2}-\pi_{x}\tau_{n}\right|\leq\left|\pi_{x}z-\pi_{x}\tau_{n}\right|\leq\left|\pi_{x}z_{1}-\pi_{x}\tau_{n}\right|\leq c\lambda^{2}\left(\frac{1}{\lambda}\right)^{2j}.
\]
\end{proof}

\index{good region|)}

\subsection{Properties for the bad region\index{bad region|(}}

We prove the first property for the bad region by applying Lemma \ref{lem:W location}
to the boundary local stable manifolds $W_{n}^{0}(K_{n})$ and $W_{n}^{2}(K_{n})$.
\begin{lemma}
\label{lem:Bad region close to the tip}Given $\delta>0$ and $I^{v}\supset I^{h}\Supset I$.
There exists $\overline{\epsilon}>0$, $\overline{b}>0$, and $c>0$
such that for all $F\in\hat{\mathcal{I}}_{\delta}(I^{h}\times I^{v},\overline{\epsilon})$
and $b>\overline{b}$ we have
\[
\left|\pi_{x}z-\pi_{x}\tau_{n}\right|<cb\left\Vert \epsilon_{n}\right\Vert 
\]
for all $z\in C_{n}(j)\cap F_{n}(D_{n})$ with $j>K_{n}$ and $n\geq0$.
\end{lemma}
\begin{proof}
Since the bad region $\overline{\cup_{j>K_{n}}C_{n}(j)}$ is bounded
by the local stable manifolds $W_{n}^{0}(K_{n})$ and $W_{n}^{2}(K_{n})$,
it is sufficient to estimate the location of $W_{n}^{0}(K_{n})$ and
$W_{n}^{2}(K_{n})$. 

Assume that $z\in W_{n}^{t}(K_{n})\cap\left(I^{h}\times I^{h}\right)$
with $t\in\{0,2\}$. By Lemma \ref{lem:K_n bounds} and Lemma \ref{lem:W location},
there exists $c>1$ such that 
\[
\left|\pi_{x}z-\pi_{x}\tau_{n}\right|\leq c\left(\frac{1}{\lambda}\right)^{2K_{n}}\leq c^{3}b\left\Vert \epsilon_{n}\right\Vert 
\]
for all $b>0$ sufficiently large.
\end{proof}

The second property for the bad region follows from the quadratic
estimates Lemma \ref{lem:Quadratic center}.
\begin{corollary}
Given $\delta>0$ and $I^{v}\supset I^{h}\Supset I$. There exists
$\overline{\epsilon}>0$, $\overline{b}>0$, and $c>0$ such that
for all $F\in\hat{\mathcal{I}}_{\delta}(I^{h}\times I^{v},\overline{\epsilon})$
and $b>\overline{b}$ we have
\[
\left|\pi_{x}z-v_{n}\right|<c\sqrt{b\left\Vert \epsilon_{n}\right\Vert }
\]
 for all $z\in B_{n}(j)$ with $j>K_{n}$ and $n\geq0$.
\end{corollary}
\begin{proof}
Assume that $z\in B_{n}(j)$. Then $F_{n}(z)\in C_{n}(j)\cap F_{n}(D_{n})$.
By Proposition \ref{prop:Bound for f(v)-tau} and Lemma \ref{lem:Bad region close to the tip},
there exists $c>0$ such that 
\begin{eqnarray*}
\left|f_{n}(\pi_{x}z)-f_{n}(v_{n})\right| & \leq & \left|h_{n}(z)-\pi_{x}\tau_{n}\right|+\left|f_{n}(\pi_{x}z)-h_{n}(z)\right|+\left|\pi_{x}\tau_{n}-f_{n}(v_{n})\right|\\
 & \leq & (cb+1+c)\left\Vert \epsilon_{n}\right\Vert \\
 & < & 2cb\left\Vert \epsilon_{n}\right\Vert 
\end{eqnarray*}
for all $b>0$ sufficiently large. Also, by Lemma \ref{lem:Quadratic center},
we have 
\[
\left|f_{n}(\pi_{x}z)-f_{n}(v_{n})\right|\geq\frac{1}{2a}(\pi_{x}z-v_{n})^{2}
\]
for some constant $a>0$. Combine the two inequalities, we obtain
\[
\left|\pi_{x}z-v_{n}\right|\leq\sqrt{4ac}\sqrt{b\left\Vert \epsilon_{n}\right\Vert }.
\]
\end{proof}

\index{bad region|)}

\section{\label{sec:Good region}The Good Region\index{good region} and the
Expansion Argument\index{expansion argument}}

Our goal in this chapter is to prove the expansion argument in the
good region, Proposition \ref{prop:Good region estimates}: the horizontal
size of the closest approach expands when the wandering domains stay
in the good region. This shows that Hénon-like maps behaves like unimodal
maps in the good region. 

From now on, fix $b>0$ sufficiently large so that Proposition \ref{prop:Good and Bad Regions}
holds and the sequence $\left\{ K_{n}\right\} _{n\geq0}$ depends
only on $F$.

To prove the horizontal size expands, we iterate the horizontal endpoints
to estimate the expansion of the horizontal size. The vertical line
argument in Chapter \ref{sec:Good and Bad region} showed that the
iteration of the horizontal endpoints fails to approximate the expansion
rate when the line connecting the two horizontal endpoints is parallel
to the preimage of a vertical line. The following condition, $R$-regular,
provides a criteria to ensure ``parallel'' does not happen in the
good region.
\begin{definition}[Regular]
Let $R>0$. A set $U\subset D(F)$ is $R$-regular\index{regular|textbf}
if 
\begin{equation}
\frac{h(U)}{l(U)}\leq R\frac{1}{\left\Vert \epsilon\right\Vert ^{1/4}}.\label{eq:h/l bound}
\end{equation}
\end{definition}
To see $R$-regular implies not parallel, we estimate the slope of
the preimage of a vertical line. Assume that $\gamma:I^{v}\rightarrow I^{h}$
is the vertical graph of the preimage of some vertical line $x=x_{0}$
by the Hénon-like map $F_{n}$ and the vertical graph is in the good
region. Then
\[
h_{n}(\gamma(y),y)=x_{0}.
\]
Apply the derivative in terms of $y$ to the both sides, we solved
\[
\gamma'(y)=\frac{\frac{\partial\epsilon_{n}}{\partial y}(\gamma(y),y)}{f_{n}'(\gamma(y))-\frac{\partial\epsilon_{n}}{\partial x}(\gamma(y),y)}.
\]
By Lemma \ref{lem:Quadratic center} and Proposition \ref{prop:Good and Bad Regions},
we get
\begin{eqnarray*}
\left|f_{n}'(y)-\frac{\partial\epsilon_{n}}{\partial x}(\gamma(y),y)\right| & \geq & \frac{1}{a}\left|\gamma(y)-v_{n}\right|-\frac{1}{\delta}\left\Vert \epsilon_{n}\right\Vert \\
 & \geq & \frac{c}{a}\sqrt{\left\Vert \epsilon_{n}\right\Vert }-\frac{1}{\delta}\left\Vert \epsilon_{n}\right\Vert \\
 & \geq & \frac{c}{2a}\sqrt{\left\Vert \epsilon_{n}\right\Vert }
\end{eqnarray*}
when $\overline{\epsilon}$ is small enough. This yields
\begin{equation}
\left|\gamma'(y)\right|\leq c'\sqrt{\left\Vert \epsilon_{n}\right\Vert }\label{eq:slope of preimage of VL}
\end{equation}
for some constant $c'>0$. 

The condition $R$-regular says that the vertical slope of the line
determined by the horizontal endpoints $(x_{1},y_{1})$ and $(x_{2},y_{2})$
of $J$ is bounded by
\begin{equation}
\frac{\left|x_{2}-x_{1}\right|}{\left|y_{2}-y_{1}\right|}\geq\frac{l(J)}{h(J)}\geq\frac{1}{R}\left\Vert \epsilon_{n}\right\Vert ^{1/4}.\label{eq:slope of horizontal endpoints}
\end{equation}
From (\ref{eq:slope of preimage of VL}) and (\ref{eq:slope of horizontal endpoints}),
we get
\[
\frac{\left|x_{2}-x_{1}\right|}{\left|y_{2}-y_{1}\right|}\gg\left|\gamma'(y)\right|.
\]
This concludes that the line connecting the horizontal endpoints is
not parallel to the preimage of a vertical line if the wandering domain
is $R$-regular.

Now we state the main proposition of this chapter.
\begin{proposition}[Expansion argument]
\label{prop:Good region estimates}\index{good region|textit}\index{horizontal size|textit}\index{expansion argument|textit}Given
$\delta>0$ and $I^{v}\supset I^{h}\Supset I$. There exists $\overline{\epsilon}>0$,
$E>1$, and $R>0$ such that for all $F\in\hat{\mathcal{I}}_{\delta}(I^{h}\times I^{v},\overline{\epsilon})$
the following property hold:

Assume that $J\subset A\cup B$ is a $R$-regular closed subset of
a wandering domain for $F$ and $\left\{ J_{n}\right\} _{n=0}^{\infty}$
is the $J$-closest approach. If $k_{n}\leq K_{r(n)}$ for all $n\leq m$,
then $J_{n}$ is $R$-regular for all $n\leq m+1$ and 
\begin{equation}
l_{n+1}\geq El_{n}\label{eq:Good region expansion}
\end{equation}
for all $n\leq m$.
\end{proposition}
\begin{proof}
The proof is based on the estimations for the expansion rate developed
later in this chapter. The expansion rate will be computed in three
different cases: 

\begin{enumerate}
\item $J_{n}\subset A_{r(n)}$, proved in Lemma \ref{lem:A->A bounds}.
\item $J_{n}\subset B_{r(n)}(k_{n})$ for the intermediate region $1\leq k_{n}<\overline{K}$
where $\overline{K}$ is some constant, proved in Lemma \ref{lem:B->C1 bounds}.
\item $J_{n}\subset B_{r(n)}(k_{n})$ for the region close to the center
$\overline{K}\leq k_{n}\leq K_{r(n)}$, proved in Lemma \ref{lem:B->C2 bounds}.
\end{enumerate}
Here we assume the three lemmas to prove this proposition.

Fixed $R>0$ to be the constant given by Lemma \ref{lem:B->C2 bounds}.
Also, fixed $\overline{K}\geq1$ to be an integer large enough such
that
\[
cE_{2}^{\overline{K}}>1
\]
from (\ref{eq:B->C2 expansion rate}) where $c>0$ and $E_{2}>1$
are the constants in Lemma \ref{lem:B->C2 bounds}. Let $E_{1}>1$
be the expansion constant in Lemma \ref{lem:A->A bounds} and $E_{3}>1$
be the expansion constant in Lemma \ref{lem:B->C1 bounds}. Set $E=\min(E_{1},cE_{2}^{\overline{K}},E_{3})>1$.
Let $\overline{\epsilon}>0$ be small enough such that Proposition
\ref{prop:Good and Bad Regions}, Lemma \ref{lem:A->A bounds}, Lemma
\ref{lem:B->C2 bounds}, and Lemma \ref{lem:B->C1 bounds} hold.

Assume that $F\in\hat{\mathcal{I}}_{\delta}(I^{h}\times I^{v},\overline{\epsilon})$
and $J\subset A\cup B$ is a closed $R$-regular subset that is a
wandering domain of $F$. We prove that $J_{n}$ is $R$-regular by
induction then (\ref{eq:Good region expansion}) follows by the three
lemmas. 

For the base case, $J_{0}$ is $R$-regular by assumption.

If $J_{n}$ is $R$-regular and $k_{n}\leq K_{r(n)}$ for some $n\geq0$.
If $J_{n}\subset A_{r(n)}$, then $J_{n+1}$ is $R$-regular and 
\[
l_{n+1}\geq E_{1}l_{n}\geq El_{n}
\]
by Lemma \ref{lem:A->A bounds}. If $J_{n}\subset B_{r(n)}(k_{n})$
with $1\leq k_{n}\leq\overline{K}$, then $J_{n+1}$ is $R$-regular
and
\[
l_{n+1}\geq E_{3}\lambda^{k_{n}}l_{n}\geq El_{n}
\]
by Lemma \ref{lem:B->C1 bounds}. If $J_{n}\subset B_{r(n)}(k_{n})$
with $\overline{K}\leq k_{n}\leq K_{n}$, then $J_{n+1}$ is $R$-regular
and
\[
l_{n+1}\geq cE_{2}^{k_{n}}l_{n}\geq cE_{2}^{\overline{K}}l_{n}\geq El_{n}
\]
by Lemma \ref{lem:B->C2 bounds}.

Therefore, the theorem is proved by induction.
\end{proof}

\subsection[Case A]{\label{subsec:Case A->A or B}Case $J_{n}\subset A_{r(n)}$}

In this section, we compute the expansion rate when a wandering domain
$J$ lies in $A$. We use the property that $F_{n}$ is close to the
fixed point $G$ then apply the properties for $g$ in Section \ref{subsec:Feigenbaum map}
to estimate the expansion.
\begin{lemma}
\label{lem:A->A bounds}Given $\delta>0$ and $I^{v}\supset I^{h}\Supset I$.
For all $R>0$, there exists $\overline{\epsilon}=\overline{\epsilon}(R)>0$
and $E>1$ such that for all $F\in\hat{\mathcal{I}}_{\delta}(I^{h}\times I^{v},\overline{\epsilon})$
the following property hold for all $n\geq0$:

Assume that $J\subset A_{n}$ is an $R$-regular closed set. Then
$J'\subset C_{n}(0)=A_{n}\cup W_{n}^{1}(0)\cup B_{n}$ is $R$-regular
and 
\[
l(J')\geq El(J)
\]
where $J'=F_{n}(J)$.
\end{lemma}
\begin{proof}
Let $E>1$ be the constant defined in Lemma \ref{lem:Expanding rate on A and C}
and $\Delta E>0$ be small enough such that $E'\equiv E-\Delta E>1$.
Assume that $\overline{\epsilon}>0$ is small enough such that Lemma
\ref{lem:Expanding rate on A and C} holds and 
\begin{equation}
\frac{R}{\delta}\left\Vert \epsilon_{n}\right\Vert ^{3/4}<\Delta E\label{eq:delta E}
\end{equation}
for all $n\geq0$.

To prove the inequality, let $(x_{1},y_{1}),(x_{2},y_{2})\in J$ such
that $x_{2}-x_{1}=l(J)$. Then $h(J)\geq\left|y_{2}-y_{1}\right|$. 

Compute
\begin{eqnarray*}
l(J') & \geq & \left|\pi_{x}\left[F_{n}(x_{2},y_{2})-F_{n}(x_{1},y_{1})\right]\right|\\
 & \geq & \left|\pi_{x}\left[F_{n}(x_{2},y_{2})-F_{n}(x_{1},y_{2})\right]\right|-\left|\pi_{x}\left[F_{n}(x_{1},y_{2})-F_{n}(x_{1},y_{1})\right]\right|.
\end{eqnarray*}
By the mean value theorem, there exists $\xi\in(x_{1},x_{2})$ and
$\eta\in(y_{1},y_{2})$ such that 
\[
\pi_{x}\left[F_{n}(x_{2},y_{2})-F_{n}(x_{1},y_{2})\right]=\frac{\partial h_{n}}{\partial x}(\xi,y_{2})(x_{2}-x_{1})
\]
and 
\[
\pi_{x}\left[F_{n}(x_{1},y_{2})-F_{n}(x_{1},y_{1})\right]=\frac{\partial\epsilon_{n}}{\partial y}(x_{1},\eta)(y_{2}-y_{1}).
\]
Since $(x_{1},y_{1}),(x_{2},y_{2})\in A_{n}\subset I^{AC}\times I_{n}^{v}$,
we have $(\xi,y_{2})\in I^{AC}\times I_{n}^{v}$. By Lemma \ref{lem:Bounds for e_n}
and Lemma \ref{lem:Expanding rate on A and C}, we get

\begin{eqnarray*}
l(J') & \geq & El(J)-\frac{1}{\delta}\left\Vert \epsilon_{n}\right\Vert h(J)\\
 & = & \left(E-\frac{1}{\delta}\left\Vert \epsilon_{n}\right\Vert \frac{h(J)}{l(J)}\right)l(J).
\end{eqnarray*}
Also, by $J$ is $R$-regular and (\ref{eq:delta E}), this yields
\begin{eqnarray}
l(J') & \geq & \left(E-\frac{R}{\delta}\left\Vert \epsilon_{n}\right\Vert ^{3/4}\right)l(J)\nonumber \\
 & \geq & E'l(J).\label{eq:A->B E'}
\end{eqnarray}

To prove that $J'$ is $R$-regular, we apply (\ref{eq:A->B E'})
and $h(J')=l(J)$. We get
\[
\frac{h(J')}{l(J')}\leq\frac{1}{E'}.
\]
Also assume that $\overline{\epsilon}$ is small enough such that
$\frac{1}{E'}\leq R\left\Vert \epsilon_{n}\right\Vert ^{-1/4}$ for
all $n\geq0$. This proves that $J'$ is $R$-regular.
\end{proof}

\subsection[Case B center]{\label{subsec:Case B->C2}Case $J_{n}\subset B_{r(n)}(k_{n})$,
$\overline{K}\leq k_{n}\leq K_{r(n)}$}

In this section, we prove the horizontal size expands when a wandering
domains is in the center of good region. 

Unfortunately, the rescaling trick, Proposition \ref{prop:degenerate-rescaling trick},
does not work in the non-degenerate case. In the degenerate case,
the affine rescaling coincide with the rescaling for the unimodal
renormalization about the critical point. Thus, the affine rescaling
$\Lambda_{n}$ maps rescaling levels in $B_{n}$ to renormalization
levels in $B_{n+1}$ (Proposition \ref{lem:degenerate-renormalization operator}).
However, in the non-degenerate case, the affine rescaling has no geometrical
and dynamical meaning. So the proofs for Proposition \ref{prop:degenerate-rescaling trick}
and Corollary \ref{cor:degenerate-Rescaling trick} do not apply to
the non-degenerate case. We have to find another strategy to estimate
the expansion rate for the horizontal size.

The idea of the proof is as follows. When the sequence element $J_{n}$
enters $B_{r(n)}$, the step from $J_{n}$ to $J_{n+1}$ contains
an iteration $F_{r(n)}$ and a composition of rescalings $\Phi_{r(n)}^{k_{n}}$.
On the one hand, the horizontal size contracts by the iteration $F_{r(n)}$
because the Hénon-like map acts like a quadratic map. Lemma \ref{lem:Quadratic center}
says that the contraction becomes strong as $J_{n}$ approach to the
center. On the other hand, the horizontal size expands by the rescaling
$\Phi_{r(n)}^{k_{n}}$ (Lemma \ref{lem:Expanding rate on A and C}).
Proposition \ref{prop:Good and Bad Regions} says the number of rescaling
$k_{n}$ becomes large as $J_{n}$ approach to the center. We will
show the expansion compensates with the contraction. This yields the
following lemma.
\begin{lemma}
\label{lem:B->C2 bounds}Given $\delta>0$ and $I^{v}\supset I^{h}\Supset I$.
There exists $\overline{\epsilon}>0$, $E>1$, $R>0$, and $c>0$
such that for all $F\in\hat{\mathcal{I}}_{\delta}(I^{h}\times I^{v},\overline{\epsilon})$
the following property holds for all $n\geq0$:

Assume that $J\subset B_{n}(k)$ is an $R$-regular closed set and
$1\leq k\leq K_{n}$, then $J'\subset C_{n+k}(0)=A_{n+k}\cup W_{n+k}^{1}(0)\cup B_{n+k}$
is $R$-regular and
\begin{equation}
l(J')\geq cE^{k}l(J)\label{eq:B->C2 expansion rate}
\end{equation}
where $J'=\Phi_{n}^{k}\circ F_{n}(J)$.
\end{lemma}
In the remaining part of this chapter, we will fixed $\overline{K}>0$
to be sufficiently large so that (\ref{eq:B->C2 expansion rate})
provides a strict expansion to the horizontal size for all $k\geq\overline{K}$.

To prove this lemma, we set up the notations. Given a closed set $J\subset B_{n}(k)$.
Let $(x_{1},y_{1}),(x_{2},y_{2})\in J$ be such that $l(J)=\left|x_{2}-x_{1}\right|$.
Then $h(J)\geq\left|y_{2}-y_{1}\right|$. We define $(x_{1}^{(j)},y_{1}^{(j)})=\Phi_{n}^{j}\circ F_{n}(x_{1},y_{1})$
and $(x_{2}^{(j)},y_{2}^{(j)})=\Phi_{n}^{j}\circ F_{n}(x_{2},y_{2})$
for $j=0,\cdots,k$. Also, let $x\in\{x_{1},x_{2}\}$ be such that
$\left|x-v_{n}\right|=\min_{i=1,2}\left|x_{i}-v_{n}\right|$. 

We first prove the following estimate.
\begin{lemma}
\label{lem:B->C2 estimates}Given $\delta>0$ and $I^{v}\supset I^{h}\Supset I$.
For all $R>0$, there exists $\overline{\epsilon}=\overline{\epsilon}(R)>0$,
$E>1$, $a>1$, and $R'>0$ such that for all $F\in\hat{\mathcal{I}}_{\delta}(I^{h}\times I^{v},\overline{\epsilon})$
the following property holds for all $n\geq0$:

Assume that $J\subset B_{n}(k)$ is an $R$-regular closed set and
$k\leq K_{n}$ then 
\begin{eqnarray*}
\left|x_{2}^{(j)}-x_{1}^{(j)}\right| & \geq & \frac{1}{2a}\left|x-v_{n}\right|\left(\lambda E\right)^{j}l(J)
\end{eqnarray*}
and
\[
\left|\frac{y_{2}^{(j)}-y_{1}^{(j)}}{x_{2}^{(j)}-x_{1}^{(j)}}\right|\leq R'\frac{1}{\sqrt{\left\Vert \epsilon_{n}\right\Vert }}
\]
 for $j=0,\cdots,k$.

The constants $E$, $a$, and $R'$ does not depend on $R$.
\end{lemma}
\begin{proof}
Let $\overline{\epsilon}>0$ be small enough so that Lemma \ref{lem:Bounds for e_n},
Lemma \ref{lem:Quadratic center}, and Proposition \ref{prop:Good and Bad Regions}
hold. Let $E'>1$ be the constant defined in Lemma \ref{lem:Expanding rate on A and C}
and $E$ be a constant such that $E'>E>1$. We prove the lemma by
induction on $j$. 

For the case $j=0$, we have 
\begin{eqnarray}
\left|x_{2}^{(0)}-x_{1}^{(0)}\right| & = & \left|\pi_{x}\left(F_{n}(x_{2},y_{2})-F_{n}(x_{1},y_{1})\right)\right|\nonumber \\
 & \ge & \left|\pi_{x}\left(F_{n}(x_{2},y_{2})-F_{n}(x_{1},y_{2})\right)\right|-\left|\pi_{x}\left(F_{n}(x_{1},y_{2})-F_{n}(x_{1},y_{1})\right)\right|\label{eq:B->C2 estimates base triangular}
\end{eqnarray}
Apply the mean value theorem, there exists $\xi\in(x_{1},x_{2})$
and $\eta\in(y_{1},y_{2})$ such that
\begin{equation}
\pi_{x}\left(F_{n}(x_{2},y_{2})-F_{n}(x_{1},y_{2})\right)=\left[f_{n}'(\xi)-\frac{\partial\epsilon_{n}}{\partial x}(\xi,y_{2})\right](x_{2}-x_{1})\label{eq:B->C2 estimates base triangular 1}
\end{equation}
and
\begin{equation}
\pi_{x}\left(F_{n}(x_{1},y_{2})-F_{n}(x_{1},y_{1})\right)=-\frac{\partial\epsilon_{n}}{\partial y}(x_{1},\eta)(y_{2}-y_{1}).\label{eq:B->C2 estimates base triangular 2}
\end{equation}
Then $\xi\in I^{B}$ since $(x_{1},y_{1}),(x_{2},y_{2})\in B_{n}\subset I^{B}\times I_{n}^{v}$.
By Lemma \ref{lem:Quadratic center}, (\ref{eq:B->C2 estimates base triangular 1})
yields 
\begin{eqnarray}
\left|\pi_{x}\left(F_{n}(x_{2},y_{2})-F_{n}(x_{1},y_{2})\right)\right| & \geq & \left(\left|f_{n}'(\xi)\right|-\left|\frac{\partial\epsilon_{n}}{\partial x}(\xi,y_{2})\right|\right)l(J)\nonumber \\
 & \geq & \left(\frac{1}{a}\left|x-v_{n}\right|-\frac{1}{\delta}\left\Vert \epsilon_{n}\right\Vert \right)l(J).\label{eq:B->C2 estimates base triangular 1'}
\end{eqnarray}
Also, since $J$ is $R$-regular, (\ref{eq:B->C2 estimates base triangular 2})
yields
\begin{equation}
\left|\pi_{x}\left(F_{n}(x_{1},y_{2})-F_{n}(x_{1},y_{1})\right)\right|\leq\frac{1}{\delta}\left\Vert \epsilon_{n}\right\Vert h(J)\leq\frac{R}{\delta}\left\Vert \epsilon_{n}\right\Vert ^{3/4}l(J).\label{eq:B->C2 estimates base triangular 2'}
\end{equation}
Combine (\ref{eq:B->C2 estimates base triangular}), (\ref{eq:B->C2 estimates base triangular 1'}),
and (\ref{eq:B->C2 estimates base triangular 2'}), we get 
\[
\left|x_{2}^{(0)}-x_{1}^{(0)}\right|\geq\left[\frac{1}{a}\left|x-v_{n}\right|-\frac{1}{\delta}\left(\left\Vert \epsilon_{n}\right\Vert ^{1/2}+R\left\Vert \epsilon_{n}\right\Vert ^{1/4}\right)\sqrt{\left\Vert \epsilon_{n}\right\Vert }\right]l(J)
\]
By Proposition \ref{prop:Good and Bad Regions}, $c\left|x-v_{n}\right|>\sqrt{\left\Vert \epsilon_{n}\right\Vert }$
for some constant $c>1$. Also, assume that $\overline{\epsilon}=\overline{\epsilon}(R)$
is small enough such that $\frac{c}{\delta}\left(\left\Vert \epsilon_{n}\right\Vert ^{1/2}+R\left\Vert \epsilon_{n}\right\Vert ^{1/4}\right)<\frac{1}{2a}$
for all $n\geq0$. We obtain 
\[
\left|x_{2}^{(0)}-x_{1}^{(0)}\right|\geq\frac{1}{2a}\left|x-v_{n}\right|l(J).
\]

Moreover, by applying $y_{2}^{(0)}-y_{1}^{(0)}=l(J)$ and the previous
inequality, we have 
\begin{eqnarray*}
\left|\frac{y_{2}^{(0)}-y_{1}^{(0)}}{x_{2}^{(0)}-x_{1}^{(0)}}\right| & \leq & \frac{2a}{\left|x-v_{n}\right|}.
\end{eqnarray*}
By Proposition \ref{prop:Good and Bad Regions} again, we get 
\[
\left|\frac{y_{2}^{(0)}-y_{1}^{(0)}}{x_{2}^{(0)}-x_{1}^{(0)}}\right|\leq R'\frac{1}{\sqrt{\left\Vert \epsilon_{n}\right\Vert }}
\]
where $R'=2ac$.

Assume that the two inequalities are true for $j\leq k$. We prove
the inequalities for $j+1\leq k$. We have $(x_{1}^{(j)},y_{1}^{(j)}),(x_{2}^{(j)},y_{2}^{(j)})\in C_{n+j}(k-j)$.
By the mean value theorem, there exists $\xi_{j}\in(x_{1}^{(j)},x_{2}^{(j)})\subset I^{AC}$
and $\eta_{j}\in(y_{1}^{(j)},y_{2}^{(j)})$ such that
\begin{equation}
\pi_{x}\left(\phi_{n+j}(x_{2}^{(j)},y_{2}^{(j)})-\phi_{n+j}(x_{1}^{(j)},y_{2}^{(j)})\right)=-\lambda_{n+j}\frac{\partial h_{n}}{\partial x}(\xi_{j},y_{2}^{(j)})\left(x_{2}^{(j)}-x_{1}^{(j)}\right)\label{eq:B->C2 estimates j triangular 1}
\end{equation}
and
\begin{equation}
\pi_{x}\left(\phi_{n+j}(x_{1}^{(j)},y_{2}^{(j)})-\phi_{n+j}(x_{1}^{(j)},y_{1}^{(j)})\right)=\lambda_{n+j}\frac{\partial\epsilon_{n+j}}{\partial y}(x_{1}^{(j)},\eta_{j})\left(y_{2}^{(j)}-y_{1}^{(j)}\right).\label{eq:B->C2 estimates j triangular 2}
\end{equation}
Apply Lemma \ref{lem:Expanding rate on A and C} to (\ref{eq:B->C2 estimates j triangular 1}),
we have 
\begin{eqnarray}
\left|\pi_{x}\left(\phi_{n+j}(x_{2}^{(j)},y_{2}^{(j)})-\phi_{n+j}(x_{1}^{(j)},y_{1}^{(j)})\right)\right| & \geq & \lambda_{n+j}E'\left|x_{2}^{(j)}-x_{1}^{(j)}\right|\label{eq:B->C2 estimates j triangular 1'}
\end{eqnarray}
for some constant $E'>1$. Also apply the induction hypothesis to
(\ref{eq:B->C2 estimates j triangular 2}), we have
\begin{eqnarray}
\left|\pi_{x}\left(\phi_{n+j}(x_{1}^{(j)},y_{2}^{(j)})-\phi_{n+j}(x_{1}^{(j)},y_{1}^{(j)})\right)\right| & \leq & \frac{\lambda_{n+j}}{\delta}\left\Vert \epsilon_{n+j}\right\Vert \left|y_{2}^{(j)}-y_{1}^{(j)}\right|\nonumber \\
 & \leq & \frac{\lambda_{n+j}R'}{\delta}\sqrt{\left\Vert \epsilon_{n}\right\Vert }\left|x_{2}^{(j)}-x_{1}^{(j)}\right|.\label{eq:B->C2 estimates j triangular 2'}
\end{eqnarray}
By the triangular inequality, (\ref{eq:B->C2 estimates j triangular 1'}),
and (\ref{eq:B->C2 estimates j triangular 2'}),we get 
\begin{align}
 & \left|x_{2}^{(j+1)}-x_{1}^{(j+1)}\right|\nonumber \\
\leq & \left|\pi_{x}\left(\phi_{n+j}(x_{2}^{(j)},y_{2}^{(j)})-\phi_{n+j}(x_{1}^{(j)},y_{2}^{(j)})\right)\right|-\left|\pi_{x}\left(\phi_{n+j}(x_{1}^{(j)},y_{2}^{(j)})-\phi_{n+j}(x_{1}^{(j)},y_{1}^{(j)})\right)\right|\nonumber \\
\leq & \lambda_{n+j}\left(E'-\frac{R'}{\delta}\sqrt{\left\Vert \epsilon_{n}\right\Vert }\right)\left|x_{2}^{(j)}-x_{1}^{(j)}\right|\label{eq:B->C2 induction step}
\end{align}
Assume that $\overline{\epsilon}$ is sufficiently small such that
$\lambda_{n+j}\left(E'-\frac{R'}{\delta}\sqrt{\left\Vert \epsilon_{n}\right\Vert }\right)>\lambda E$
for all $n\geq0$ and $j\geq0$ since $E'>E>1$ and $\left|\lambda_{n}-\lambda\right|<\overline{\epsilon}$.
Apply the induction hypothesis to (\ref{eq:B->C2 induction step}),
we get 
\[
\left|x_{2}^{(j+1)}-x_{1}^{(j+1)}\right|\geq\lambda E\left|x_{2}^{(j)}-x_{1}^{(j)}\right|\geq\frac{1}{2a}\left|x-v_{n}\right|\left(\lambda E\right)^{j+1}l(J).
\]

Moreover, by applying $y_{2}^{(j+1)}-y_{1}^{(j+1)}=\lambda_{n+j}\left(y_{2}^{(j)}-y_{1}^{(j)}\right)$,
(\ref{eq:B->C2 induction step}), and the induction hypothesis, we
get 
\[
\frac{\left|y_{2}^{(j+1)}-y_{1}^{(j+1)}\right|}{\left|x_{2}^{(j+1)}-x_{1}^{(j+1)}\right|}\leq\frac{1}{E'-\frac{1}{\delta}R'\sqrt{\left\Vert \epsilon_{n}\right\Vert }}\frac{\left|y_{2}^{(j)}-y_{1}^{(j)}\right|}{\left|x_{2}^{(j)}-x_{1}^{(j)}\right|}<R'\frac{1}{\sqrt{\left\Vert \epsilon_{n}\right\Vert }}
\]
since $E'-\frac{1}{\delta}R'\sqrt{\left\Vert \epsilon_{n}\right\Vert }>1$
when $\overline{\epsilon}$ is small enough. 

Therefore, the two inequalities are proved by induction.
\end{proof}

We also need a lemma to relate the rescaling level $k_{n}$ with the
distance from the wandering domain $J_{n}$ to the center.
\begin{lemma}
\label{lem:C_good location =000026 renormalization level}Given $\delta>0$
and $I^{v}\supset I^{h}\Supset I$. There exists $\overline{\epsilon}>0$
and $c>0$ such that for all $F\in\hat{\mathcal{I}}_{\delta}(I^{h}\times I^{v},\overline{\epsilon})$
the following property holds for all $n\geq0$:

If $(x,y)\in B_{n}(k)$ with $k\leq K_{n}$ then $k$ is bounded below
by
\[
\lambda^{k}>\frac{c}{\left|x-v_{n}\right|}.
\]
\end{lemma}
\begin{proof}
Let $\overline{\epsilon}$ be sufficiently small. Apply Proposition
\ref{prop:Good and Bad Regions} and triangular inequality, we have
\begin{eqnarray*}
\frac{1}{c}\left(\frac{1}{\lambda}\right)^{2k} & < & \left|h_{n}(x,y)-\pi_{x}\tau_{n}\right|\\
 & \leq & \left|f_{n}(x)-f_{n}(v_{n})\right|+\left|f_{n}(v_{n})-\pi_{x}\tau_{n}\right|+\left|h_{n}(x,y)-f_{n}(x)\right|
\end{eqnarray*}
for some constant $c>0$ since $F_{n}(x,y)\in C_{n}(k)$. By Lemma
\ref{lem:Quadratic center}, Lemma \ref{prop:Bound for f(v)-tau},
and Proposition \ref{prop:Good and Bad Regions}, we get
\begin{eqnarray*}
\frac{1}{c}\left(\frac{1}{\lambda}\right)^{2k} & \leq & \frac{a}{2}\left(x-v_{n}\right)^{2}+(c'+1)\left\Vert \epsilon_{n}\right\Vert \\
 & \leq & \left[\frac{a}{2}+(c'+1)c'\right]\left(x-v_{n}\right)^{2}
\end{eqnarray*}
for some constants $a>0$ and $c'>0$. Therefore, the lemma is proved.
\end{proof}

Now we are ready to prove the main lemma of this section.
\begin{proof}
(Proof of Lemma \ref{lem:B->C2 bounds}) By Lemma \ref{lem:B->C2 estimates},
we have 
\[
l(J')\geq\left|x_{2}^{(k)}-x_{1}^{(k)}\right|\geq\frac{1}{2a}\left|x-v_{n}\right|\left(\lambda E\right)^{k}l(J).
\]
for some constant $a>0$. Apply Lemma \ref{lem:C_good location =000026 renormalization level},
we get 
\[
l(J')\geq\frac{c}{2a}E^{k}l(J)
\]
for some constant $c>0$. This proves (\ref{eq:B->C2 expansion rate}).

It remains to show that $J'$ is $R$-regular. By Proposition \ref{prop:Hyperbolicity of the Renormalization Operator},
we have 
\[
\left\Vert \epsilon_{n+k}\right\Vert \leq\left\Vert \epsilon_{n+1}\right\Vert \leq c\left\Vert \epsilon_{n}\right\Vert ^{2}
\]
for some constant $c>0$ since $k\geq1$. Also, the Hénon-like map
$F_{n}$ maps the $x$-th coordinate to $y$-th coordinate and $\Phi_{n}^{k}$
rescales the $y$-th coordinate affinely, we have $h(J')=\left|y_{2}^{(k)}-y_{1}^{(k)}\right|$.
By Lemma \ref{lem:B->C2 estimates}, we get 
\[
\frac{h(J')}{l(J')}\leq R'\frac{1}{\sqrt{\left\Vert \epsilon_{n}\right\Vert }}\leq R'c\frac{1}{\left\Vert \epsilon_{n+k}\right\Vert ^{1/4}}.
\]
Set $R=R'c$. Then $J'$ is $R$-regular. Here, we fixed $R$ so $\overline{\epsilon}$
is a constant.
\end{proof}

\subsection[Case B intermediate]{\label{subsec:Case B->C1}Case $J_{n}\subset B_{r(n)}(k_{n})$,
$1\leq k_{n}<\overline{K}$}

In Lemma \ref{lem:B->C2 bounds}, the expansion of horizontal size
only works for levels $k\geq\overline{K}$ which are close to the
center. In this section, we prove the expansion argument also holds
in the intermediate region $1\leq k<\overline{K}$.

Although the rescaling trick does not apply to the non-degenerate
case, we still can apply it to the limiting degenerate Hénon-like
map $G$. In the limiting case, the horizontal size expands by the
rescaling trick. Because $\overline{K}$ is a fixed number, we will
show the horizontal size also expands in the intermediate region of
a non-degenerate Hénon-like map when it is close enough to the limiting
function $G$.

Observe in the limiting case, we have 
\[
\lim_{n\rightarrow\infty}F_{n}(x,y)=(g(x),x),
\]
and
\[
\lim_{n\rightarrow\infty}\phi_{n}(x,y)=(-\lambda)(g(x),y).
\]
Then 
\begin{equation}
\lim_{n\rightarrow\infty}\Phi_{n}^{j}(x,y)=(\left[(-\lambda)g\right]^{j}(x),(-\lambda)^{j}y)\label{eq:limiting case Phi}
\end{equation}
and 
\[
\lim_{n\rightarrow\infty}\Phi_{n}^{j}\circ F_{n}(x,y)=(\left[(-\lambda)g\right]^{j}\circ g(x),(-\lambda)^{j}x)
\]
where $\left[(-\lambda)g\right]^{j}$ means the function $x\rightarrow(-\lambda)g(x)$
is composed $j$ times.

The following is the rescaling trick, Lemma \ref{prop:degenerate-rescaling trick},
for the limiting case.
\begin{lemma}[Rescaling trick]
\index{rescaling trick}Assume that $j\geq0$ is an integer. Then
\begin{equation}
\left[(-\lambda)g\right]^{j}\circ g(x)=g((-\lambda)^{j}x)\label{eq:Rescaling Trick}
\end{equation}
for all $-\left(\frac{1}{\lambda}\right)^{j}\leq x\leq\left(\frac{1}{\lambda}\right)^{j}$.
\end{lemma}
\begin{proof}
The lemma follows either from the functional equation \ref{eq:Functional Equation}
or Proposition \ref{prop:degenerate-rescaling trick}.
\end{proof}

By the rescaling trick, we are able to estimate the derivative for
the limiting case as follows.
\begin{lemma}
\label{lem:B->C1 derivative estimation}There exists universal constants
$E,E'>1$ such that for all integer $j\geq0$ we have 
\begin{equation}
E\lambda^{j}\leq\left|\frac{\mathrm{d}\left[(-\lambda)g\right]^{j}\circ g}{\mathrm{d}x}(x)\right|\leq E'\lambda^{j}\label{eq:derivative estimation}
\end{equation}
for all and $\left(\frac{1}{\lambda}\right)^{j+1}\leq\left|x\right|\leq\left(\frac{1}{\lambda}\right)^{j}$.
\end{lemma}
\begin{proof}
By the rescaling trick and chain rule, we get
\[
\frac{\mathrm{d}\left[(-\lambda)g\right]^{j}\circ g}{\mathrm{d}x}(x)=(-\lambda)^{j}g'((-\lambda)^{j}x)
\]
for all $\left|x\right|\leq\left(\frac{1}{\lambda}\right)^{j}$. By
Proposition \ref{prop:Derivative of g on A}, there exists $E>1$
such that 
\[
\left|g'(x)\right|\geq E
\]
for all $\frac{1}{\lambda}\leq\left|x\right|\leq1$. Also, by compactness,
there exists $E'>0$ such that 
\[
\left|g'(x)\right|\leq E'
\]
for all $x\in I$. This yields (\ref{eq:derivative estimation}) since
$\frac{1}{\lambda}\leq\left|(-\lambda)^{j}x\right|\leq1$ for all
$\left(\frac{1}{\lambda}\right)^{j+1}\leq\left|x\right|\leq\left(\frac{1}{\lambda}\right)^{j}$.
\end{proof}

Next, we need to do some hard work to make these expansion estimates
also work on Hénon-like maps that are close enough to the fixed point
$G$.

One of the difficulty is the subpartitions on $B$ and $C$ are not
rectangular in the non-degenerate case. The following lemmas, Lemma
\ref{lem:unif cts on interval}, Lemma \ref{lem:rect extension on C},
and Corollary \ref{cor:rect ext on B}, allow us to apply the function
$\Phi_{n}^{j}\circ F_{n}$ on a rectangular neighborhood of $B_{n}(j)$.
\begin{lemma}
\label{lem:unif cts on interval}Given $\delta>0$ and $I^{v}\supset I^{h}\Supset I$.
For all $d>0$, there exists $\overline{\epsilon}=\overline{\epsilon}(d)>0$
and $d'=d'(d)$ such that for all $F\in\hat{\mathcal{I}}_{\delta}(I^{h}\times I^{v},\overline{\epsilon})$
we have
\[
s_{n}\circ h_{n}([a-d',b+d'],y)\subset[(-\lambda g)(a)-d,(-\lambda g)(b)+d]
\]
for all $[a-d',b+d']\subset I^{AC,r}$, $y\in I_{n}^{v}$, and $n\geq0$
where $I^{AC,r}$ is the right component of $I^{AC}$ defined in Lemma
\ref{lem:Expanding rate on A and C}.
\end{lemma}
\begin{proof}
By the compactness of $\overline{I^{h}}$, there exists $E\geq1$
such that $\left|g'(x)\right|\leq E$ for all $x\in I^{h}$. Then
\begin{align*}
 & s_{n}\circ h_{n}(a-d')\\
\geq & (-\lambda)g(a-d')-\left|(-\lambda)g(a-d')-(-\lambda)h_{n}(a-d')\right|-\left|(-\lambda)h_{n}(a-d')-s_{n}\circ h_{n}(a-d')\right|\\
\geq & (-\lambda)g(a)-\lambda Ed'-\lambda\left\Vert F_{n}-G\right\Vert -\left\Vert s_{n}(x)-(-\lambda)x\right\Vert _{I^{h}}\\
\geq & (-\lambda)g(a)-d
\end{align*}
when $d'$ is small enough such that $\lambda Ed'<\frac{d}{2}$ and
$\overline{\epsilon}$ is small enough such that $\lambda\left\Vert F_{n}-G\right\Vert +\left\Vert s_{n}(x)-(-\lambda)x\right\Vert _{I^{h}}<\frac{d}{2}$
for all $n\geq0$.

Similarly, 
\begin{align*}
 & s_{n}\circ h_{n}(b+d')\\
\leq & (-\lambda)g(b+d')+\left|(-\lambda)g(b+d')-(-\lambda)h_{n}(b+d')\right|+\left|(-\lambda)h_{n}(b+d')-s_{n}\circ h_{n}(b+d')\right|\\
\leq & (-\lambda)g(b)+\lambda Ed'+\lambda\left\Vert F_{n}-G\right\Vert +\left\Vert s_{n}(x)-(-\lambda)x\right\Vert _{I^{h}}\\
\leq & (-\lambda)g(b)+d.
\end{align*}
Therefore, the lemma is proved.
\end{proof}

Recall that $q(j)=g(q^{c}(j))$ from Definition \ref{def:q}.
\begin{lemma}
\label{lem:rect extension on C}Given $\delta>0$ and $I^{v}\supset I^{h}\Supset I$.
For all integer $j\geq1$, there exists $\overline{\epsilon}(j)>0$
and $d^{C}(j)>0$ such that for all $F\in\hat{\mathcal{I}}_{\delta}(I^{h}\times I^{v},\overline{\epsilon})$
the rescaling $\Phi_{n}^{j}$ is defined on $[q(j-1)-d^{C}(j),q(j)+d^{C}(j)]\times I_{n}^{v}$
for all $n\geq0$ and $[q(0)-d^{C}(1),q(1)+d^{C}(1)]\subset I^{AC}$.

In addition, 
\begin{equation}
\phi_{n}([q(k-1)-d^{C}(k),q(k)+d^{C}(k)]\times I_{n}^{v})\subset[q(k-2)-d^{C}(k-1),q(k-1)+d^{C}(k-1)]\times I_{n+1}^{v}\label{eq:extension for C(j)}
\end{equation}
 for all $F\in\hat{\mathcal{I}}_{\delta}(I^{h}\times I^{v},\overline{\epsilon}(j))$,
$n\geq0$, and $2\leq k\leq j$.
\end{lemma}
\begin{proof}
We prove by induction on $j\geq1$. 

For the base case $j=1$, $\Phi_{n}^{1}=\phi_{n}$ is defined on $I^{AC,r}\times I_{n}^{v}$
by Lemma \ref{lem:Expanding rate on A and C} where $I^{AC,r}$ is
the right component of $I^{AC}$.

Assume that there exists $d>0$ such that $\Phi_{n}^{j}$ is defined
on $[q(j-1)-d,q(j)+d]\times I_{n}^{v}$ for all $n\geq0$.

For the case $j+1$, we know that $(-\lambda g)(q(j))=q(j-1)$ and
$(-\lambda g)(q(j+1))=q(j)$. By Lemma \ref{lem:unif cts on interval},
there exists $d'>0$ and $\overline{\epsilon}(j+1)\geq\overline{\epsilon}(j)$
such that 
\[
s_{n}\circ h_{n}([q(j)-d',q(j+1)+d'],y)\subset[q(j-1)-d,q(j)+d]
\]
for all $F\in\hat{\mathcal{I}}_{\delta}(I^{h}\times I^{v},\overline{\epsilon}(j+1))$,
$y\in I_{n}^{v}$, and $n\geq0$. Then $\pi_{x}\circ\phi_{n}([q(j)-d',q(j+1)+d']\times I_{n}^{v})\subset[q(j-1)-d,q(j)+d]\times I_{n+1}^{v}$.
Therefore, $\Phi_{n}^{j+1}=\Phi_{n+1}^{j}\circ\phi_{n}$ is defined
on $[q(j)-d',q(j+1)+d']\times I_{n}^{v}$ by the induction hypothesis.

The relation (\ref{eq:extension for C(j)}) follows from the definition
of $d(j)$.
\end{proof}

Recall $q^{l}(j)=-\left|q^{c}(j)\right|$ and $q^{r}(j)=\left|q^{c}(j)\right|$
from Definition \ref{def:ql,qr}.
\begin{corollary}
\label{cor:rect ext on B}Given $\delta>0$ and $I^{v}\supset I^{h}\Supset I$.
For all integer $j\geq1$, there exists $\overline{\epsilon}(j)>0$,
$d^{B}(j)>0$, $E>1$, and $E'>1$ such that for all $F\in\hat{\mathcal{I}}_{\delta}(I^{h}\times I^{v},\overline{\epsilon})$
the following properties hold for all $n\geq0$:
\begin{enumerate}
\item $F_{n}([q^{l}(j-1)-d^{B}(j),q^{l}(j)+d^{B}(j)]\times I_{n}^{v})\subset[q(j-1)-d^{C}(j),q(j)+d^{C}(j)]\times I_{n}^{v}$
and $F_{n}([q^{r}(j-1)-d^{B}(j),q^{r}(j)+d^{B}(j)]\times I_{n}^{v})\subset[q(j-1)-d^{C}(j),q(j)+d^{C}(j)]\times I_{n}^{v}$.
\\
That is, $\Phi_{n}^{j}\circ F_{n}$ is defined on $([q^{l}(j-1)-d^{B}(j),q^{l}(j)+d^{B}(j)]\cup[q^{r}(j)-d^{B}(j),q^{r}(j-1)+d^{B}(j)])\times I_{n}^{v}$.
\item $B_{n}^{l}(j)\subset[q^{l}(j-1)-d^{B}(j),q^{l}(j)+d^{B}(j)]\times I_{n}^{v}$
and $B_{n}^{r}(j)\subset[q^{r}(j-1)-d^{B}(j),q^{r}(j)+d^{B}(j)]\times I_{n}^{v}$.
\end{enumerate}
Here $B_{n}^{l}(j)$ and $B_{n}^{r}(j)$ are the left and right components
of $B_{n}(j)$ respectively.
\end{corollary}
\begin{proof}
Fixed $j\geq1$. 

There exists $d'>0$ and $\overline{\epsilon}>0$ small enough such
that $\Phi_{n}^{j}$ is defined on $[q(j-1)-d',q(j)+d']\times I_{n}^{v}$
for all $n\geq0$. By the continuity of $g$, there exists $d>0$
such that $g([q^{l}(j-1)-d,q^{l}(j)+d])\subset[q(j-1)-\frac{d'}{2},q(j)+\frac{d'}{2}]$.
Then 
\begin{eqnarray*}
h_{n}(q^{l}(j-1)-d,y) & \geq & g(q^{l}(j-1)-d)-\left\Vert h_{n}-g\right\Vert _{I^{h}(\delta)\times I_{n}^{v}(\delta)}\\
 & \geq & q(j-1)-d'
\end{eqnarray*}
for all $n\geq0$. Here, we assume that $\overline{\epsilon}$ is
small enough such that $\left\Vert h_{n}-g\right\Vert _{I^{h}(\delta)\times I_{n}^{v}(\delta)}<\frac{d'}{2}$.
Similarly, 
\begin{eqnarray*}
h_{n}(q^{l}(j)+d,y) & \geq & g(q^{l}(j)+d)-\left\Vert h_{n}-g\right\Vert _{I^{h}(\delta)\times I_{n}^{v}(\delta)}\\
 & \geq & q(j)-d'.
\end{eqnarray*}
Thus, $F_{n}([q^{l}(j-1)-d,q^{l}(j)+d]\times I_{n}^{v})\subset[q(j-1)-d',q(j)+d']\times I_{n}^{v}$.
This proves that $\Phi_{n}^{j}\circ F_{n}$ is defined on $[q^{l}(j-1)-d,q^{l}(j)+d]\times I_{n}^{v}$.

Similarly, we can choose $d>0$ to be small enough such that $F_{n}([q^{r}(j-1)-d,q^{r}(j)+d]\times I_{n}^{v})\subset[q(j-1)-d',q(j)+d']\times I_{n}^{v}$.
Consequently, the first property is proved.

By Proposition \ref{prop:W(j) in B for n large}, we may also assume
that $\overline{\epsilon}=\overline{\epsilon}(j)$ is small enough
such that $W_{n}^{l}(j-1)\subset[q^{l}(j-1)-d,q^{l}(j-1)+d]\times I_{n}^{v}$,
$W_{n}^{l}(j)\subset[q^{l}(j)-d,q^{l}(j)+d]\times I_{n}^{v}$, $W_{n}^{r}(j)\subset[q^{r}(j)-d,q^{r}(j)+d]\times I_{n}^{v}$,
and $W_{n}^{r}(j-1)\subset[q^{r}(j-1)-d,q^{r}(j-1)+d]\times I_{n}^{v}$.
This proves the second property.
\end{proof}

The next lemma will show that the expansion of a Hénon-like map is
close to the limiting case when $\overline{\epsilon}$ is small.
\begin{lemma}
\label{lem:limit of the jth iteration+rescaling dx}Given $\delta>0$
and $I^{v}\supset I^{h}\Supset I$. For all $\hat{\epsilon}>0$ and
integer $j\geq1$, there exists $\overline{\epsilon}=\overline{\epsilon}(\hat{\epsilon},j)>0$
such that 
\begin{equation}
\left|\frac{\partial\pi_{x}\circ\Phi_{n}^{j}\circ F_{n}}{\partial x}(x,y)-\frac{\mathrm{d}\left[(-\lambda)g\right]^{j}\circ g}{\mathrm{d}x}(x)\right|<\hat{\epsilon}\label{eq:limit of the jth iteration+rescaling dx}
\end{equation}
for all $F\in\hat{\mathcal{I}}_{\delta}(I^{h}\times I^{v},\overline{\epsilon})$,
$n\geq0$, and $(x,y)\in([q^{l}(j-1)-d^{B}(j),q^{l}(j)+d^{B}(j)]\cup[q^{r}(j)-d^{B}(j),q^{r}(j-1)+d^{B}(j)])\times I_{n}^{v}$.
\end{lemma}
\begin{proof}
The Lemma is true because $j$ is fixed and the Hénon-like maps $F_{n}$
are close to the fixed point $G$ when $\overline{\epsilon}$ is small.
The proof is left as an exercise to the reader.
\end{proof}

\begin{corollary}
\label{cor:Rescaling trick expansion}Given $\delta>0$ and $I^{v}\supset I^{h}\Supset I$.
For all integer $j\geq1$, there exists $\overline{\epsilon}(j)>0$,
$\hat{d}^{B}(j)$, $E>1$, and $E'>1$ such that for all $F\in\hat{\mathcal{I}}_{\delta}(I^{h}\times I^{v},\overline{\epsilon})$
the following properties hold for all $n\geq0$:

\[
E\lambda^{j}\leq\left|\frac{\partial\pi_{x}\circ\Phi_{n}^{j}\circ F_{n}}{\partial x}(x,y)\right|\leq E'\lambda^{j}
\]
 for all $x\in[q^{l}(j-1)-\hat{d}^{B}(j),q^{l}(j)+\hat{d}^{B}(j)]\cup[q^{r}(j)-\hat{d}^{B}(j),q^{r}(j-1)+\hat{d}^{B}(j)]$
and $y\in I_{n}^{v}$.
\end{corollary}
\begin{proof}
By the continuity of $g$ and (\ref{eq:derivative estimation}), we
may assume that $\hat{d}^{B}(j)<d^{B}(j)$ is small enough and $E'>E>1$
such that 
\[
E\lambda^{j}\leq\left|\frac{\mathrm{d}\left[(-\lambda)g\right]^{j}\circ g}{\mathrm{d}x}(x)\right|\leq E'\lambda^{j}
\]
for all $\left(\frac{1}{\lambda}\right)^{j+1}-\hat{d}^{B}(j)\leq\left|x\right|\leq\left(\frac{1}{\lambda}\right)^{j}+\hat{d}^{B}(j)$. 

By Lemma \ref{lem:limit of the jth iteration+rescaling dx}, we get
\[
\sqrt{E}\lambda^{j}\leq\left|\frac{\partial\pi_{x}\circ\Phi_{n}^{j}\circ F_{n}}{\partial x}(x,y)\right|\leq\sqrt{E'}\lambda^{j}
\]
for all $\left(\frac{1}{\lambda}\right)^{j+1}-\hat{d}^{B}(j)\leq\left|x\right|\leq\left(\frac{1}{\lambda}\right)^{j}+\hat{d}^{B}(j)$,
$y\in I_{n}^{v}$, and $n\geq0$ when $\overline{\epsilon}$ is small
enough.
\end{proof}

By applying the estimates from the limiting case, we are able to estimate
the expanding rate for the intermediate case as follows.
\begin{lemma}
\label{lem:B->C1 bounds}Given $\delta>0$ and $I^{v}\supset I^{h}\Supset I$.
For all $\overline{K}>0$ and $R>0$, there exists $\overline{\epsilon}=\overline{\epsilon}(\overline{K},R)>0$
and $E>1$ such that for all $F\in\hat{\mathcal{I}}_{\delta}(I^{h}\times I^{v},\overline{\epsilon})$
the following properties hold for all $n\geq0$:

Assume that $J\subset B_{n}(k)$ is a connected closed $R$-regular
set and $k\leq\min\left(\overline{K},K_{n}\right)$. Then $J'\subset C_{n+k}(0)=A_{n+k}\cup W_{n+k}^{1}(0)\cup B_{n+k}$
is $R$-regular and 
\begin{equation}
l(J')\geq E\lambda^{k}l(J)\label{eq:B->C1 expansion}
\end{equation}
where $J'=\Phi_{n}^{k}\circ F_{n}(J)$.
\end{lemma}
\begin{proof}
Since $\overline{K}$ is fixed, we may assume that $\overline{\epsilon}=\overline{\epsilon}(\overline{K})>0$
is sufficiently small such that the properties in Corollary \ref{cor:Rescaling trick expansion}
hold for all $j\leq\overline{K}$.

Given a connected closed $R$-regular set $J\subset B_{n}(k)$ with
$k\leq\min\left(\overline{K},K_{n}\right)$ and $n\geq0$. For convenience,
let $G=\Phi_{n}^{k}\circ F_{n}$ and denote $G_{x}=\pi_{x}\circ G$.
We have $J'=G(J)$.

To prove (\ref{eq:B->C1 expansion}), assume the case that $J\subset B_{n}^{l}(k)$.
The other case $J\subset B_{n}^{r}(k)$ is similar. Let $(x_{1},y_{1}),(x_{2},y_{2})\in J$
such that $l(J)=x_{2}-x_{1}$. From Corollary \ref{cor:Rescaling trick expansion},
$G$ is defined on $[q^{l}(k-1)-d(k),q^{l}(k)+d(k)]\times I_{n}^{v}$
and $x_{1},x_{2}\in[q^{l}(k-1)-d(k),q^{l}(k)+d(k)]$. We can apply
the mean value theorem. There exists $\xi\in(x_{1},x_{2})$ and $\eta\in(y_{1},y_{2})$
such that 
\[
G_{x}(x_{2},y_{2})-G_{x}(x_{1},y_{2})=\frac{\partial G_{x}}{\partial x}(\xi,y_{2})(x_{2}-x_{1})
\]
and 
\[
G_{x}(x_{1},y_{2})-G_{x}(x_{1},y_{1})=\frac{\partial G_{x}}{\partial y}(x_{1},\eta)(y_{2}-y_{1}).
\]
By triangular inequality and $J$ is $R$-regular, we get
\begin{eqnarray}
l(J') & \geq & \left|G_{x}(x_{2},y_{2})-G_{x}(x_{1},y_{2})\right|-\left|G_{x}(x_{1},y_{2})-G_{x}(x_{1},y_{1})\right|\label{eq:B->C1 triangular}\\
 & \geq & \left|\frac{\partial G_{x}}{\partial x}(\xi,y_{2})\right|l(J)-\left|\frac{\partial G_{x}}{\partial y}(x_{1},\eta)\right|h(J)\nonumber \\
 & \geq & \left(\left|\frac{\partial G_{x}}{\partial x}(\xi,y_{2})\right|-\left|\frac{\partial G_{x}}{\partial y}(x_{1},\eta)\right|R\left\Vert \epsilon_{n}\right\Vert ^{-1/4}\right)l(J).\nonumber 
\end{eqnarray}

The first term $\frac{\partial G_{x}}{\partial x}(\xi,y_{2})$ can
be bounded by Corollary \ref{cor:Rescaling trick expansion}. That
is
\begin{equation}
E\lambda^{k}<\left|\frac{\partial G_{x}}{\partial x}(\xi,y_{2})\right|<E'\lambda^{k}\label{eq:B->C1 triangular 1}
\end{equation}
for some constants $E'>E>1$. 

To bound the second term $\frac{\partial G_{x}}{\partial y}(x_{1},\eta)$,
compute by the chain rule and Corollary \ref{cor:Rescaling trick expansion}.
We get
\begin{eqnarray*}
E'\lambda^{k} & > & \left|\frac{\partial G_{x}}{\partial x}(x_{1},\eta)\right|\\
 & = & \left|\frac{\partial\pi_{x}\circ\Phi_{n}^{k}}{\partial x}\circ F_{n}(x_{1},\eta)\frac{\partial h_{n}}{\partial x}(x_{1},\eta)+\frac{\partial\pi_{x}\circ\Phi_{n}^{k}}{\partial y}\circ F_{n}(x_{1},\eta)\right|\\
 & \geq & \left|\frac{\partial\pi_{x}\circ\Phi_{n}^{k}}{\partial x}\circ F_{n}(x_{1},\eta)\frac{\partial h_{n}}{\partial x}(x_{1},\eta)\right|-\left|\frac{\partial\pi_{x}\circ\Phi_{n}^{k}}{\partial y}\circ F_{n}(x_{1},\eta)\right|
\end{eqnarray*}
and 
\begin{eqnarray*}
\left|\frac{\partial G_{x}}{\partial y}(x_{1},\eta)\right| & = & \left|\frac{\partial\pi_{x}\circ\Phi_{n}^{k}}{\partial x}\circ F_{n}(x_{1},\eta)\frac{\partial h_{n}}{\partial y}(x_{1},\eta)\right|\\
 & = & \left|\frac{\partial\pi_{x}\circ\Phi_{n}^{k}}{\partial x}\circ F_{n}(x_{1},\eta)\frac{\partial h_{n}}{\partial x}(x_{1},\eta)\right|\left|\frac{\frac{\partial h_{n}}{\partial y}(x_{1},\eta)}{\frac{\partial h_{n}}{\partial x}(x_{1},\eta)}\right|\\
 & \leq & \left(E'\lambda^{k}+\left|\frac{\partial\pi_{x}\circ\Phi_{n}^{k}}{\partial y}\circ F_{n}(x_{1},\eta)\right|\right)\left(\left|f_{n}'(x_{1})\right|-\frac{1}{\delta}\left\Vert \epsilon_{n}\right\Vert \right)^{-1}\frac{1}{\delta}\left\Vert \epsilon_{n}\right\Vert .
\end{eqnarray*}
Assume that $\overline{\epsilon}=\overline{\epsilon}(\overline{K})$
is small enough such that $\left|\frac{\partial\pi_{x}\circ\Phi_{n}^{k}}{\partial y}(x,y)\right|<E'$
for all $k\leq\overline{K}$ and $n\geq0$. This is possible because
of (\ref{eq:limiting case Phi}), the definition of $\hat{\mathcal{I}}$,
and $\overline{K}$ is a fixed bounded number. Apply Lemma \ref{lem:Quadratic center}, we get 
\[
\left|\frac{\partial G_{x}}{\partial y}(x_{1},\eta)\right|\leq2E'\lambda^{k}\left(\frac{1}{a}\left|x_{1}-v_{n}\right|-\frac{1}{\delta}\left\Vert \epsilon_{n}\right\Vert \right)^{-1}\frac{1}{\delta}\left\Vert \epsilon_{n}\right\Vert .
\]
for some constant $a>1$. Also, by Proposition \ref{prop:Good and Bad Regions},
we obtain
\begin{equation}
\left|\frac{\partial G_{x}}{\partial y}(x_{1},\eta)\right|\leq2E'\lambda^{k}\left(\frac{c}{a}\left\Vert \epsilon_{n}\right\Vert ^{1/2}-\frac{1}{\delta}\left\Vert \epsilon_{n}\right\Vert \right)^{-1}\frac{1}{\delta}\left\Vert \epsilon_{n}\right\Vert \leq\frac{4E'a}{\delta c}\lambda^{k}\left\Vert \epsilon_{n}\right\Vert ^{1/2}\label{eq:B->C1 triangular 2'}
\end{equation}
for some constant $c>0$ when $\overline{\epsilon}$ is small enough. 

Combine (\ref{eq:B->C1 triangular}), (\ref{eq:B->C1 triangular 1}),
and (\ref{eq:B->C1 triangular 2'}), we obtain 
\begin{eqnarray}
l(J') & \geq & \left(E\lambda^{k}-\frac{4E'ac}{\delta}\lambda^{k}\left\Vert \epsilon_{n}\right\Vert ^{1/2}R\left\Vert \epsilon_{n}\right\Vert ^{-1/4}\right)l(J)\nonumber \\
 & \geq & \sqrt{E}\lambda^{k}l(J)\label{eq:B->C1 expansion-1}
\end{eqnarray}
when $\overline{\epsilon}$ is small enough.

To prove that $J'$ is $R$-regular, we apply (\ref{eq:B->C1 expansion-1})
and $h(J')=\left(\prod_{j=0}^{k(J)-1}\lambda_{j+n}\right)l(J)$. Assume
that $\overline{\epsilon}=\overline{\epsilon}(\overline{K})$ is small
enough such that $\prod_{j=0}^{i-1}\lambda_{j+n}\leq2\lambda^{i}$
for all $1\leq i\leq\overline{K}$ and $n\geq0$. Thus,
\[
\frac{h(J')}{l(J')}\leq\frac{2\lambda^{k}l(J)}{\sqrt{E}\lambda^{k}l(J)}=\frac{2}{\sqrt{E}}\leq R\left\Vert \epsilon_{n+k}\right\Vert ^{-1/4}
\]
when $\overline{\epsilon}=\overline{\epsilon}(R)$ is small enough.
\end{proof}

\section{\label{sec:Bad region}The Bad Region\index{bad region} and the
Thickness\index{thickness}}

In the good region, we showed the expansion argument holds by studying
the iteration of the horizontal endpoints. However, in the bad region,
the iteration of horizontal endpoints fails to estimate the change
rate of the horizontal size. In fact, the $x$-displacement of the
endpoints can shrink as small as possible by the vertical line argument
in Chapter \ref{sec:Good and Bad region}. Even worse, the case of
entering the bad region is unavoidable. The next lemma shows that
an infinitely renormalizable Hénon-like map must have a wandering
domain in the bad region if it has any wandering domain. 
\begin{lemma}
{*}Given $\delta>0$ and $I^{v}\supset I^{h}\Supset I$. There exists
$\overline{\epsilon}>0$ such that for all non-degenerate Hénon-like
maps $F\in\hat{\mathcal{I}}_{\delta}(I^{h}\times I^{v},\overline{\epsilon})$
the following property holds. If $F$ has a wandering domain in $D$
then $F$ has a wandering domain in the bad region of $B$ and a wandering
domain in the bad region of $C$.
\end{lemma}
\begin{proof}
Recall $K_{0}$ is the boundary of the good and bad region for $F_{0}=F$.
See Definition \ref{def:Good and bad region}. Let $j>K_{0}$.

If $F$ has a wandering domain in $D$, then $F_{j}$ also has a wandering
domain $J'$ in $D_{j}$ by Corollary \ref{cor:Wandering Domain/Renormalization}.
By iterating the wandering domain, we can assume without lose of generality
that $J'\subset F_{j}(D_{j})$. Set $J_{C}=\left(\Phi_{0}^{j}\right)^{-1}(J')$.
Then $J_{C}\subset C$ is a wandering domain of $F$.

Moreover, since $J'\subset F_{j}(D_{j})$, we have $J_{C}\subset\left(\Phi_{0}^{j}\right)^{-1}(F_{j}(D_{j}))\subset F(D)$.
Let $J_{B}=F^{-1}(J_{C})$. Then $J_{B}\subset B$ is a wandering
domain of $F$ in the bad region.
\end{proof}

The case of entering the bad region becomes the main difficulty for
proving the nonexistence of wandering domain. 

In this chapter, we will first introduce a new quantity ``thickness''\index{thickness}
that approximates the horizontal or vertical cross-section of a wandering
domain. When a sequence element $J_{n}$ enters the bad region, we
showed by the vertical line argument in Chapter \ref{sec:Good and Bad region}
that the next sequence element $J_{n+1}$ can turn so vertical that
the iteration of horizontal endpoints fails to estimate the horizontal
size of $J_{n+1}$. However, the horizontal size $l_{n+1}$ is not
zero because $J_{n+1}$ has area as shown in Figure \ref{fig:vertical argument-area}.
This is because the Hénon-like map is non-degenerated. The Jacobian
is not zero. Thus, thickness (horizontal cross-section) provides an
approximation for the horizontal size $l_{n+1}$ of $J_{n+1}$.

Before giving a precise definition for the thickness and rigorous
computation for its change rate, here we present a lax estimation
on the thickness by using an area argument to explain the relations
between the thickness, horizontal size, and vertical size in a closest
approach. 

Assume that we start from a square subset $J_{0}$ of a wandering
domain. Let $\left\{ J_{n}\right\} _{n=0}^{\infty}$ be the $J$-closest
approach, $a_{n}$ be the area of $J_{n}$, and $w_{n}$ be the thickness
of $J_{n}$. Assume that $J_{0},J_{1},\cdots,J_{n-1}$ stays in the
good region and $J_{n}$ enters the bad region. Since $J_{0}$ is
a square, we have $l_{0}=h_{0}$.

Some assumptions are made here to simplify the argument. The contribution
from the rescaling is neglected. When the wandering domain is $R$-regular,
assume that the horizontal size is comparable to the vertical size,
i.e. $l_{n}\sim h_{n}$. Also, assume that the thickness $w_{n}$
is determined by the horizontal cross-section which can be approximated
by $w_{n}\sim\frac{a_{n}}{h_{n}}$.

In the good region, the estimations in Chapter \ref{sec:Good region}
determines the relation of the horizontal size. Proposition \ref{prop:Good region estimates}
says that $l_{m+1}\sim El_{m}$ for all $m\leq n-1$ where $E>1$
is a constant.

However, for the wandering domain $J_{n+1}$, the horizontal size
$l_{n}$ fails to estimate $l_{n+1}$ because $J_{n}$ enters the
bad region. We need to use the thickness to approximate the horizontal
size. That is, $l_{n+1}\sim w_{n+1}$. The only known relation between
the horizontal size and the thickness prior entering the bad region
is $w_{0}=l_{0}$. This is because $J_{0}$ is a square. To relate
$l_{n+1}$ with $l_{0}$, we need to go back to study the change rate
of the thickness in each step.

We use the area to study the change rate of thickness. In each step,
the change rate of the area is determined by the Jacobian $\text{Jac}F_{0}\sim\left\Vert \epsilon_{0}\right\Vert $
of the Hénon-like map, i.e $a_{m+1}\sim\left\Vert \epsilon_{r(m)}\right\Vert a_{n}$.
We get 
\[
w_{m+1}\sim\frac{a_{m+1}}{h_{m+1}}\sim\left\Vert \epsilon_{r(m)}\right\Vert \frac{a_{m}}{l_{m}}\sim\left\Vert \epsilon_{r(m)}\right\Vert \frac{a_{m}}{h_{m}}\sim\left\Vert \epsilon_{r(m)}\right\Vert w_{m}.
\]
This allows us to relate $l_{n+1}$ with $l_{0}$ by 
\[
l_{n+1}\sim w_{n+1}\sim\left(\prod_{m=0}^{n}\left\Vert \epsilon_{r(m)}\right\Vert \right)w_{0}\sim\left(\prod_{m=0}^{n}\left\Vert \epsilon_{r(m)}\right\Vert \right)l_{0}.
\]
Consequently, the horizontal size becomes extremely small when the
$J$-closest approach leaves the bad region.

One can see two problems from the estimations.

One problem is the wandering domain $J_{n+1}$ fail to be $R$-regular
after the $J$-closest approach leaves the bad region. This is because
\[
\frac{h_{n+1}}{l_{n+1}}\sim\frac{l_{n}}{l_{n+1}}\sim\frac{E^{n}l_{0}}{\left(\prod_{m=0}^{n}\left\Vert \epsilon_{r(m)}\right\Vert \right)l_{0}}=E^{n}\left(\prod_{m=0}^{n}\left\Vert \epsilon_{r(m)}\right\Vert \right)^{-1}.
\]
Thus, the expansion argument does not work for the later sequence
element even if the $J$-closest approach does not enter the bad region
again. This problem will be resolved by introducing the largest square
subset.

Another problem is the strong contraction of horizontal size when
the wandering domain enters the bad region. We will show that this
strong contraction happens every time when the $J$-closest approach
enters the bad region. If the the $J$-closest approach enters the
bad region infinitely many time, then the horizontal size may fails
to tend to infinity because this strong contraction happens infinity
many times. This problem will be resolved in Section \ref{subsec:finite rows}
by proving the closest approach $J_{n}$ can only enter the bad region
at most finitely many times. 

Finally, after combining all of the ingredients in this article together,
we will show wandering domains do not exist in Section \ref{subsec:Nonexistence of wandering domain}.

\subsection{\label{subsec:Thickness}Thickness and largest square subset}

When the wandering domain $J_{n}$ enters the bad region, there are
two issues that stop us to proceed. First, the horizontal size of
$J_{n+1}$ cannot be estimated by the expansion argument in Chapter
\ref{sec:Good region}. Instead, it is determined by its horizontal
cross-section that is not comparable to the horizontal size of $J_{n}$.
Second, $J_{n+1}$ fail to be $R$-regular. The estimations for the
expansion rate of the horizontal size in Proposition \ref{prop:Good region estimates}
does not apply to the later steps $J_{n+1}\rightarrow J_{n+2}\rightarrow\cdots$
in the sequence. 

To resolve the two issues, we need the following:
\begin{enumerate}
\item A quantity to approximate the horizontal cross-section of a wandering
domain, called the thickness.
\item Keep track of the thickness in each step of the closest approach.
This will provides the information for the horizontal size when the
wandering domain enters the bad region.
\item A method to select a subset from the wandering domain $J_{n+1}$ that
makes the subset to be $R$-regular and has approximately the same
horizontal size as $J_{n+1}$. The subset will be defined to be a
largest square subset of $J_{n+1}$.
\end{enumerate}
In this section, we define thickness and largest square subset then
study the properties of these two objects in a closest approach. 

First, define
\begin{definition}[Square, Largest square subset, and Thickness]
\label{def:w}A set $I\subset\mathbb{R}^{2}$ is a square\index{square|textbf}
if $I=[x_{1},x_{2}]\times[y_{1},y_{2}]$ with $x_{2}-x_{1}=y_{2}-y_{1}$.
This means that $I$ is a closed square with horizontal and vertical
sides.

\nomenclature[w]{$w$}{Thickness}Assume that $J\subset\mathbb{R}^{2}$.
Define the thickness\index{thickness|textbf} of $J$ to be the quantity
$w(J)=\mbox{sup}\left\{ l(I)\right\} $ where the supremum is taken
over all square subsets $I\subset J$. 

A subset $I\subset J$ is a largest square subset\index{square!largest square subset|textbf}
of $J$ if $I$ is a square such that $l(I)=w(J)$. 

The definition is illustrated as in Figure \ref{fig:Comparison of lengths}.
\end{definition}
\begin{figure}
\begin{centering}
\includegraphics{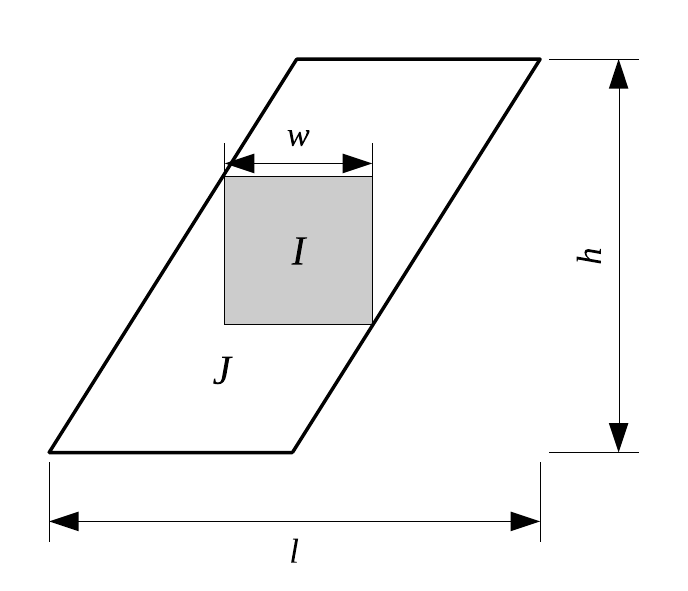}
\par\end{centering}
\caption{\label{fig:Comparison of lengths}Comparison of the horizontal size
$l$, the vertical size $h$, and the thickness $w$ for $J$. In
this picture, $I$ is a largest square subset of $J$.}
\end{figure}
\begin{lemma}
\index{thickness|textit}A largest square subset of a compact set
exists.
\end{lemma}
\begin{proof}
The lemma follows from compactness.
\end{proof}

To keep track of the thickness in each step, the following two lemmas
estimate the change rate of a square under iteration and rescaling.
\begin{lemma}
\label{lem:Largest square F_n contraction rate}\index{square|textit}Given
$\delta>0$ and $I^{v}\supset I^{h}\Supset I$. There exists $\overline{\epsilon}>0$
and $c>0$ such that for all $F\in\hat{\mathcal{I}}_{\delta}(I^{h}\times I^{v},\overline{\epsilon})$
the following property holds for all $n\geq0$:

If $I\subset D_{n}$ is a square, there exists a square $I'\subset F_{n}(I)$
such that 
\[
l(I')\geq c\frac{\left\Vert \epsilon_{n}\right\Vert }{\left|I_{n}^{v}\right|}l(I).
\]
\end{lemma}
\begin{proof}
The lemma is trivial when $F$ is degenerate. We assume that $F$
is non-degenerate. 

By the definition of $\hat{\mathcal{I}}$ and (\ref{eq:Lower bound for de/dy})
we have $\frac{\partial\epsilon_{n}}{\partial y}>0$ for all $n\geq0$.
Write $I=[x,a]\times[y_{1},y_{2}]$. Fixed $b>0$ to be sufficiently
small. Let $(x_{1}',x)=F_{n}(x,y_{2})$, $(x_{2}',x)=F_{n}(x,y_{1})$,
and $W=b(x_{2}'-x_{1}')=b\left[\epsilon_{n}(x,y_{2})-\epsilon_{n}(x,y_{1})\right]>0$.
Define $x'=\frac{x_{1}'+x_{2}'}{2}$ and $I'=\left[x'-\frac{1}{2}W,x'+\frac{1}{2}W\right]\times\left[x,x+W\right]$. 

To prove that $I'\subset F_{n}(I)$ for some $b>0$ sufficiently small,
it suffice to prove the inequality 
\begin{equation}
h_{n}(t,y_{2})<x'-\frac{1}{2}W<x'+\frac{1}{2}W<h_{n}(t,y_{1})\label{eq:F_n thickness cross-section}
\end{equation}
that corresponds to the four points on a horizontal cross section
at $y=t$ for $x\leq t\leq x+W$. See Figure \ref{fig:Largest square F_n cross section}.
\begin{figure}
\resizebox{\columnwidth}{!}{%
\begin{centering}
\includegraphics[scale=0.6]{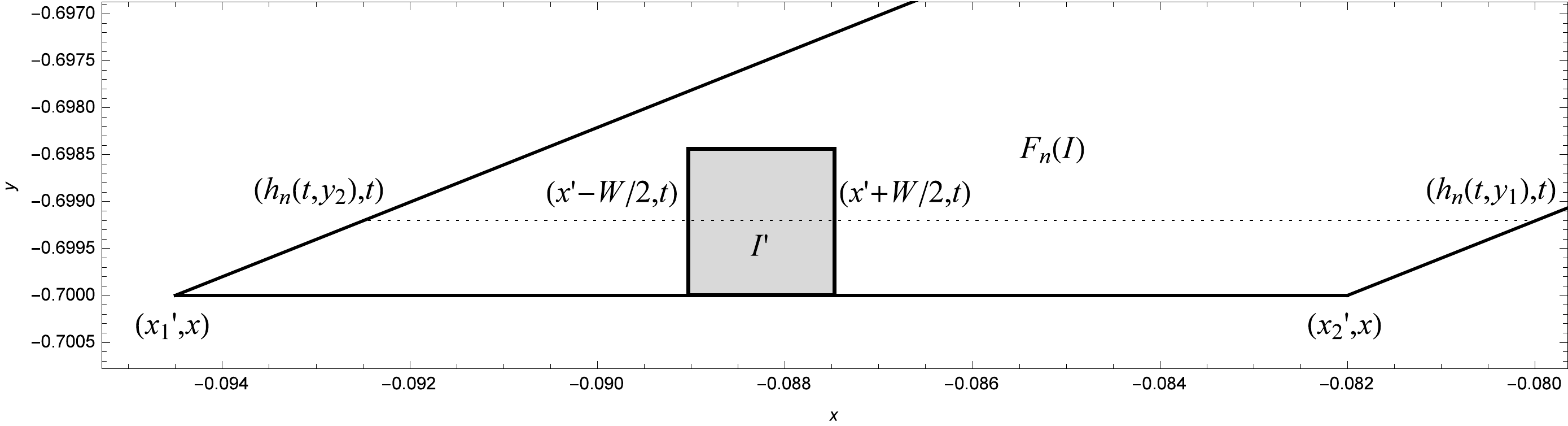}
\par\end{centering}
}

\caption{\label{fig:Largest square F_n cross section}Four points on the cross
section $y=t$.}
\end{figure}
 If this is true, then by the mean value theorem, there exists $\eta\in(y_{1},y_{2})$
such that 
\[
l(I')=W=b\frac{\partial\epsilon_{n}}{\partial y}(x,\eta)l(I).
\]
Also, by (\ref{eq:Lower bound for de/dy}), we obtain 
\[
l(I')\geq\frac{bc}{\left|I_{n}^{v}\right|}\left\Vert \epsilon_{n}\right\Vert l(I)
\]
since $F\in\hat{\mathcal{I}}_{\delta}(I^{h}\times I^{v},\overline{\epsilon})$
which proves the lemma.

First, we prove the left inequality of (\ref{eq:F_n thickness cross-section})
\[
h_{n}(t,y_{2})<x'-\frac{1}{2}W.
\]
By the mean value theorem and the compactness of the domain, there
exists $\xi\in(x,t)$ and $E>1$ such that 
\[
\left|h_{n}(t,y_{2})-x_{1}'\right|=\left|h_{n}(t,y_{2})-h_{n}(x,y_{2})\right|=\left|\frac{\partial h_{n}}{\partial x}(\xi,y_{2})\right|\left|t-x\right|\leq EW.
\]
We get
\begin{eqnarray*}
\left(x'-\frac{1}{2}W\right)-h_{n}(t,y_{2}) & = & \left[\left(x'-\frac{1}{2}W\right)-x_{1}'\right]-\left[h_{n}(t,y_{2})-x_{1}'\right]\\
 & \geq & \left(\frac{x_{2}'-x_{1}'}{2}-\frac{1}{2}W\right)-EW\\
 & = & \left[\frac{1}{2}-\left(\frac{1}{2}+E\right)b\right]\left(x_{2}'-x_{1}'\right)\\
 & > & 0
\end{eqnarray*}
when $b<\frac{1}{1+2E}$. Note that $b$ can chosen to be universal.
Thus, the left inequality is proved.

Similarly, we prove the right inequality of (\ref{eq:F_n thickness cross-section})
\[
x'+\frac{1}{2}W<h_{n}(t,y_{1}).
\]
By the mean value theorem, there exists $\xi\in(x,t)$ such that 
\[
\left|h_{n}(t,y_{1})-x_{2}'\right|=\left|h_{n}(t,y_{1})-h_{n}(x,y_{1})\right|=\left|\frac{\partial h_{n}}{\partial x}(\xi,y_{1})\right|\left|t-x\right|\leq EW.
\]
Similarly, we get 
\begin{eqnarray*}
h_{n}(t,y_{1})-\left(x'+\frac{1}{2}W\right) & = & \left[x_{2}'-\left(x'+\frac{1}{2}W\right)\right]-\left[x_{2}'-h_{n}(t,y_{1})\right]\\
 & \geq & \left(\frac{x_{2}'-x_{1}'}{2}-\frac{1}{2}W\right)-EW\\
 & = & \left[\frac{1}{2}-\left(\frac{1}{2}+E\right)b\right]\left(x_{2}'-x_{1}'\right)\\
 & > & 0.
\end{eqnarray*}
Thus, the right inequality is proved.
\end{proof}

\begin{lemma}
\label{lem:Largest square phi_n expansion rate}\index{square|textit}Given
$\delta>0$ and $I^{v}\supset I^{h}\Supset I$. There exists $\overline{\epsilon}>0$
such that for all $F\in\hat{\mathcal{I}}_{\delta}(I^{h}\times I^{v},\overline{\epsilon})$
the following property holds for all $n\geq0$:

If $I\subset C_{n}$ is a square, there exists a square $I'\subset\phi_{n}(I)$
such that
\[
l(I')=\lambda_{n}l(I).
\]
\end{lemma}
\begin{proof}
Let $I=[x_{1},x_{2}]\times[y_{1},y_{2}]$, $W=l(I)$, $x=\frac{1}{2}\left[h_{n}(x_{2},y_{1})+h_{n}(x_{1},y_{1})\right]$,
and $I''=[x-\frac{1}{2}W,x+\frac{1}{2}W]\times[y_{1},y_{2}]$. Then
$I''$ is a square with $l(I'')=l(I)$.

First we prove that $I''\subset H_{n}(I)$. It suffice to prove the
inequality 
\[
h_{n}(x_{2},t)<x-\frac{1}{2}W<x+\frac{1}{2}W<h_{n}(x_{1},t)
\]
that corresponds to the four points on a horizontal cross section
at $y=t$ for $y_{1}\leq t\leq y_{2}$. See Figure \ref{fig:Largest square H_n cross section}.
\begin{figure}
\begin{centering}
\includegraphics[scale=0.6]{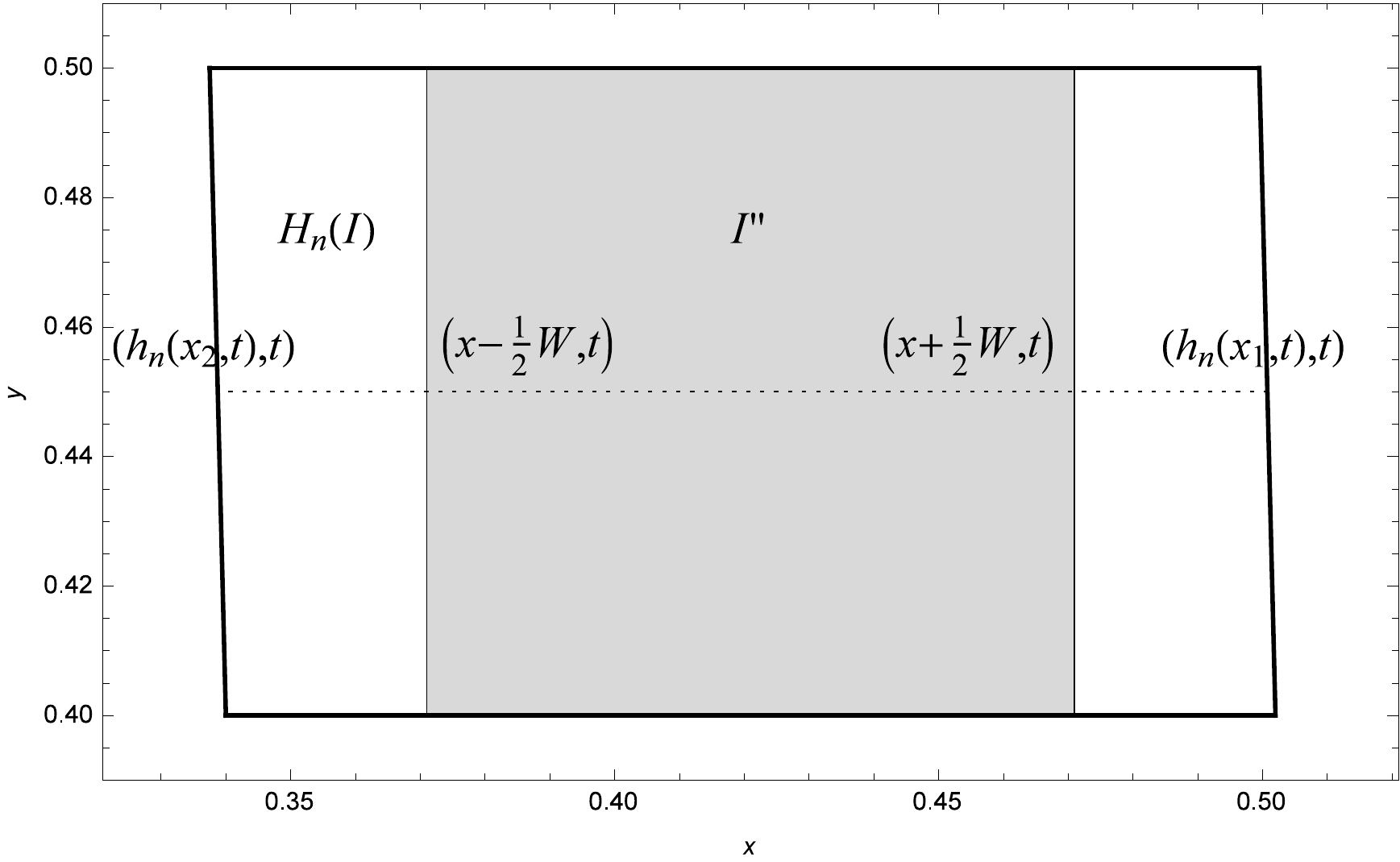}
\par\end{centering}
\caption{\label{fig:Largest square H_n cross section}Four points on the cross
section $y=t$.}
\end{figure}

To prove the left inequality, by the mean value theorem, there exists
$\xi\in(x_{1},x_{2})$ and $\eta\in(y_{1},t)$ such that
\[
h_{n}(x_{1},y_{1})-h_{n}(x_{2},y_{1})=\left|\frac{\partial h_{n}}{\partial x}(\xi,y_{1})\right|(x_{2}-x_{1})
\]
and 
\[
\epsilon_{n}(x_{2},t)-\epsilon_{n}(x_{2},y_{1})=\frac{\partial\epsilon_{n}}{\partial y}(x_{2},\eta)(t-y_{1}).
\]
By Lemma \ref{lem:Expanding rate on A and C}, there exists $E>1$
such that 
\begin{eqnarray*}
\left(x-\frac{1}{2}W\right)-h_{n}(x_{2},t) & = & \left[x-h_{n}(x_{2},y_{1})\right]-\left[\epsilon_{n}(x_{2},y_{1})-\epsilon_{n}(x_{2},t)\right]-\frac{1}{2}W\\
 & \geq & \frac{1}{2}\left|\frac{\partial h_{n}}{\partial x}(\xi,y_{1})\right|\left(x_{2}-x_{1}\right)-\left|\frac{\partial\epsilon_{n}}{\partial y}(x_{2},\eta)\right|\left(t-y_{1}\right)-\frac{1}{2}W\\
 & \geq & \left(\frac{E}{2}-\frac{1}{\delta}\left\Vert \epsilon_{n}\right\Vert -\frac{1}{2}\right)W\\
 & > & 0
\end{eqnarray*}
when $\overline{\epsilon}>0$ is sufficiently small. Thus, the left
inequality is proved.

Similarly, to prove the right inequality, by the mean value theorem,
there exists $\eta'\in(y_{1},t)$ such that 
\[
\epsilon_{n}(x_{1},t)-\epsilon_{n}(x_{1},y_{1})=\frac{\partial\epsilon_{n}}{\partial y}(x_{1},\eta')\left(t-y_{1}\right).
\]
Compute
\begin{eqnarray*}
h_{n}(x_{1},t)-\left(x+\frac{1}{2}W\right) & = & \left[h_{n}(x_{1},y_{1})-x\right]-\left[\epsilon_{n}(x_{1},t)-\epsilon_{n}(x_{1},y_{1})\right]-\frac{1}{2}W\\
 & \geq & \frac{1}{2}\left|\frac{\partial h_{n}}{\partial x}(\xi,y_{1})\right|(x_{2}-x_{1})-\left|\frac{\partial\epsilon_{n}}{\partial y}(x_{1},\eta')\right|(t-y_{1})-\frac{1}{2}W\\
 & \geq & \left(\frac{E}{2}-\frac{1}{\delta}\left\Vert \epsilon_{n}\right\Vert -\frac{1}{2}\right)W\\
 & > & 0.
\end{eqnarray*}
Thus, the right inequality is proved.

Finally, let $I'=\Lambda_{n}(I'')$. Then $I'\subset\phi_{n}(I)$
and 
\[
l(I')=\lambda_{n}l(I'')=\lambda_{n}l(I).
\]
\end{proof}

As before we abbreviate $w_{n}=w(J_{n})$ for a closest approach $\left\{ J_{n}\right\} _{n=0}^{\infty}$.
The next proposition allows us to estimate the contraction rate of
the thickness for a closest approach.
\begin{proposition}
\label{prop:w_n Bound}\index{thickness|textit}Given $\delta>0$
and $I^{v}\supset I^{h}\Supset I$. There exists $\overline{\epsilon}>0$
and $c>0$ such that for all $F\in\hat{\mathcal{I}}_{\delta}(I^{h}\times I^{v},\overline{\epsilon})$
the following property holds:

Assume that $J\subset A\cup B$ is a compact subset of a wandering
domain of $F$ and $\left\{ J_{n}\right\} _{n=0}^{\infty}$ is the
$J$-closest approach. Then
\[
w_{n+1}\geq c\frac{\left\Vert \epsilon_{r(n)}\right\Vert }{\left|I_{r(n)}^{v}\right|}w_{n}
\]
for all $n\geq0$.
\end{proposition}
\begin{proof}
Let $\overline{\epsilon}>0$ be small enough such that Lemma \ref{lem:Largest square F_n contraction rate}
and Lemma \ref{lem:Largest square phi_n expansion rate} holds. The
sets $\left\{ J_{n}\right\} _{n=0}^{\infty}$ are compact by the continuity
of Hénon-like maps and rescaling.

For the case that $J_{n}\subset A_{r(n)}$, let $I$ be a largest
square of $J_{n}$. By Proposition \ref{lem:Largest square F_n contraction rate},
there exists a square $I'\subset F_{r(n)}(I)\subset J_{n+1}$ such
that 
\[
l(I')\geq c\frac{\left\Vert \epsilon_{r(n)}\right\Vert }{\left|I_{r(n)}^{v}\right|}l(I).
\]
We get
\[
w_{n+1}\geq l(I')\geq c\frac{\left\Vert \epsilon_{r(n)}\right\Vert }{\left|I_{r(n)}^{v}\right|}l(I)=c\frac{\left\Vert \epsilon_{r(n)}\right\Vert }{\left|I_{r(n)}^{v}\right|}w_{n}.
\]

For the case that $J_{n}\subset B_{r(n)}$, let $I$ be a largest
square of $J_{n}$. By Proposition \ref{lem:Largest square F_n contraction rate},
there exists a square $I_{0}\subset F_{r(n)}(I)\subset F_{r(n)}(J_{n})$
such that
\[
l(I_{0})\geq c\frac{\left\Vert \epsilon_{r(n)}\right\Vert }{\left|I_{r(n)}^{v}\right|}l(I).
\]
Also by Proposition \ref{lem:Largest square phi_n expansion rate},
there exists a square $I_{j+1}\subset\phi_{r(n)+j}(I_{j})\subset\Phi_{r(n)}^{j}\circ F_{r(n)}(J_{n})$
such that 
\[
l(I_{j+1})=\lambda_{r(n)+j}l(I_{j})
\]
for all $0\leq j<k_{n}$. We get 
\[
w_{n+1}\geq l(I_{k_{n}})=\left(\prod_{j=0}^{k_{n}-1}\lambda_{r(n)+j}\right)l(I_{0})\geq c\frac{\left\Vert \epsilon_{r(n)}\right\Vert }{\left|I_{r(n)}^{v}\right|}l(I)=c\frac{\left\Vert \epsilon_{r(n)}\right\Vert }{\left|I_{r(n)}^{v}\right|}w_{n}.
\]
\end{proof}

\begin{remark}
The original proof was based on the area and horizontal cross-section
estimates briefly mentioned in the beginning of this chapter instead
of tracking the size of largest square subset. However, the area argument
is discarded by two reasons. First, to estimate the horizontal cross-section
of a set, we need to find the lower bound of $a/l$. This means that
we need to repeat the arguments in Chapter \ref{sec:Good region}
to find the upper bound for $l$ and the lower bound for $a$. This
makes the argument several times longer than the current one. Second,
to select a subset from the wandering domain after it enters the bad
region, the area approach makes it hard to find the upper bound of
$l$ for the subset.
\end{remark}
Since $\left\Vert \epsilon_{n}\right\Vert $ decreases super-exponentially
and $\left|I_{n}^{v}\right|$ increases exponentially, we can simplify
\begin{corollary}
\label{cor:w_n bound simplfied}Given $\delta>0$ and $I^{v}\supset I^{h}\Supset I$.
There exists $\overline{\epsilon}>0$ and $c>0$ such that for all
$F\in\hat{\mathcal{I}}_{\delta}(I^{h}\times I^{v},\overline{\epsilon})$
the following property holds:

Assume that $J\subset A\cup B$ is a compact subset of a wandering
domain of $F$ and $\left\{ J_{n}\right\} _{n=0}^{\infty}$ is the
$J$-closest approach. Then
\[
w_{n+1}\geq c\left\Vert \epsilon_{r(n)}\right\Vert ^{3/2}w_{n}
\]
for all $n\geq0$.
\end{corollary}

\subsection{\label{subsec:Double sequence}Double sequence}

Next, we study the number of times that a closest approach enters
the bad region by defining a double sequence (two-dimensional sequence/sequence
of two indices) of sets. The double sequence consists of rows. Each
row is a closest approach in the sense of Definition \ref{def: J_n Sequence of Wandering Domain}.
When the sequence first enters the bad region in a row, the horizontal
size of the next step is dominated by its thickness. Add a new row
by selecting a largest square subset then generate the closest approach
starting from the subset. Thus, each row in the double sequence corresponds
to enter the bad region once. 

The precise definition of the double sequence is as follows. Figure
\ref{fig:double sequence J (1)} illustrates the construction. 
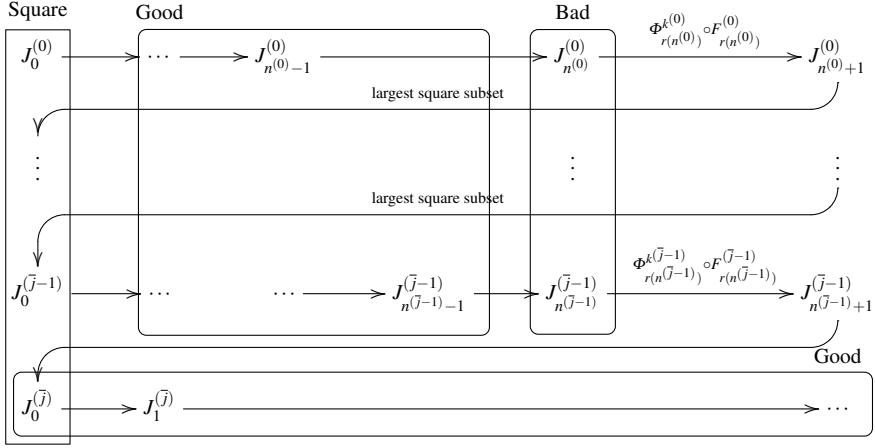
\begin{figure}
\begin{center}
$\xymatrix{J_{0}^{(0)}\ar[r]\save[]+<0cm,+0.6cm>*\txt{\mbox{Square}}\restore & \cdots\vphantom{J_{n^{(0)}-1}^{(0)}}\ar[r]\save[]+<0cm,+0.6cm>*\txt{\mbox{Good}}\restore & J_{n^{(0)}-1}^{(0)}\ar[rr] &  & J_{n^{(0)}}^{(0)}\ar[rr]^{\Phi_{r(n^{(0)})}^{k^{(0)}}\circ F_{r(n^{(0)})}^{(0)}}\save[]+<0cm,+0.6cm>*\txt{\mbox{Bad}}\restore &  & J_{n^{(0)}+1}^{(0)}\ar@{-<}`d[l]`[lllllld]_{\text{largest square subset}}[lllllld]\\
\vdots & \save"1,1"."4,1"*+[F]\frm{}\restore &  &  & \vdots &  & \vdots\ar`d[l]`[lllllld]_{\text{largest square subset}}[lllllld]\\
J_{0}^{(\overline{j}-1)}\ar[r] & \cdots\save"4,1"."4,7"*+[F:<3pt>]\frm{}\restore & \cdots\ar[r] & J_{n^{(\overline{j}-1)}-1}^{(\overline{j}-1)}\ar[r] & J_{n^{(\overline{j}-1)}}^{(\overline{j}-1)}\ar[rr]^{\Phi_{r(n^{(\overline{j}-1)})}^{k^{(\overline{j}-1)}}\circ F_{r(n^{(\overline{j}-1)})}^{(\overline{j}-1)}} & \mbox{\;\;\;\;\;\;\;} & J_{n^{(\overline{j}-1)}+1}^{(\overline{j}-1)}\ar`d[l]`[lllllld][lllllld]\\
J_{0}^{(\overline{j})}\ar[r] & J_{1}^{(\overline{j})}\ar[rrrrr]\save"1,2"."3,4"*+[F:<3pt>]\frm{}\restore &  &  & \save"1,5"."3,5"*+[F:<3pt>]\frm{}\restore &  & \cdots\save[]+<0cm,+0.7cm>*\txt<8pc>{\mbox{Good}}\restore
}
$
\par\end{center}

\caption{\label{fig:double sequence J (1)}Construction of a double sequence.}
\end{figure}
\begin{definition}[Double sequence, Row, and Time span in the good regions]
\label{def:J double sequence}\index{wandering domain!double sequence|see{double sequence}}Given
$\delta>0$ and $I^{v}\supset I^{h}\Supset I$. Assume that $\overline{\epsilon}>0$
be sufficiently small so that Proposition \ref{prop:Good and Bad Regions}
holds and $F\in\hat{\mathcal{I}}_{\delta}(I^{h}\times I^{v},\overline{\epsilon})$
is a non-degenerate open map. 

Given a square subset $J\subset A\cup B$ of a wandering domain for
$F$. Define $\left\{ J_{n}^{(j)}\right\} _{n\geq0,0\leq j\leq\overline{j}}$,
$\left\{ F_{n}^{(j)}=(f_{n}^{(j)}-\epsilon_{n}^{(j)},x)\right\} _{n\geq0,0\leq j\leq\overline{j}}$,
and $\left\{ n^{(j)}\right\} _{0\leq j\leq\overline{j}}$ for some
$\overline{j}\in\mathbb{N}\cup\{0,\infty\}$\footnote{For the case $\overline{j}=\infty$, this means that the sequence
is defined for all finite positive integers $j$.} by induction on $j$ such that the following properties hold.

\begin{enumerate}
\item For $j=0$, set $J_{0}^{(0)}=J$ and $F_{0}^{(0)}=F$. 
\item The super-script $j$ is called \index{row|textbf}row. The initial
set $J_{0}^{(j)}$ for each row $j$ is a square in $A(F_{0}^{(j)})\cup B(F_{0}^{(j)})$.
\item Each row $j$ is a $J_{0}^{(j)}$-closest approach. Precisely, if
$J_{0}^{(j)}$ and $F_{0}^{(j)}$ are defined, set $F_{n}^{(j)}=R^{n}F_{0}^{(j)}$
and $K_{n}^{(j)}$ be the boundary of good and bad region for $F_{n}^{(j)}$.
Let $\left\{ J_{n}^{(j)}\right\} _{n=0}^{\infty}$ and $\left\{ r^{(j)}(n)\right\} _{n=0}^{\infty}$
be the $J_{0}^{(j)}$-closest approach. See Definition \ref{def: J_n Sequence of Wandering Domain}
and Definition \ref{def:Good and bad region}.
\item For a row $j$, if there exists some $n\geq0$ such that $k_{n}^{(j)}>K_{r(n)}^{(j)}$,
set $n^{(j)}$ to be the smallest integer with this property. The
set $J_{n^{(j)}}^{(j)}$ is the first set in row $j$ that enters
the bad regions. The nonnegative integer $n^{(j)}$\nomenclature[n^j]{$n^{(j)}$}{Time span in the good region for row $j$ in a double sequence of wandering domain}
is called the time span in the good regions\index{time span in the good regions|textbf}
for row $j$. Otherwise, if the row never enters the bad region, set
$n^{(j)}=\infty$ and $\overline{j}=j$ and the construction stops.
\item If $n^{(j)}<\infty$, construct a new row $j+1$ by defining $J_{0}^{(j+1)}$
to be a largest square subset\index{square!largest square subset}
of $J_{n^{(j)}+1}^{(j)}$ and set $F_{0}^{(j+1)}=F_{r^{(j)}(n^{(j)}+1)}^{(j)}$.
\item If the procedure never stop, i.e. enters the bad region infinitely
many times, set $\overline{j}=\infty$. 
\end{enumerate}
The two dimensional sequence $\left\{ J_{n}^{(j)}\right\} _{n\geq0,0\leq j\leq\overline{j}}$
is called a double sequence\index{double sequence|textbf} generated
by $J$ or a $J$-double sequence. The integer $\overline{j}$ is
the number of rows\index{double sequence!number of rows|textbf} for
the double sequence (enters the bad region $\overline{j}$ times). 
\end{definition}

To be consistent and avoid confusion, the superscript is assigned
for the row and the subscript is assigned for the renormalization
level or the index of sequence element in the closest approach. For
example, abbreviate $A_{n}^{(j)}=A(F_{n}^{(j)})$, $B_{n}^{(j)}=B(F_{n}^{(j)})$,
$C_{n}^{(j)}=C(F_{n}^{(j)})$, $D_{n}^{(j)}=D(F_{n}^{(j)})$, $l_{n}^{(j)}=l(J_{n}^{(j)})$,
$h_{n}^{(j)}=h(J_{n}^{(j)})$, $w_{n}^{(j)}=w(J_{n}^{(j)})$, and
$k_{n}^{(j)}=k(J_{n}^{(j)})$ as before. 

In the following, we abbreviate $r^{(j)}(n)=r(n)$ when the context
is clear, for example $F_{r(n^{(j)}+1)}^{(j)}=F_{r^{(j)}(n^{(j)}+1)}^{(j)}$.
Also, write $\epsilon^{(j)}=\epsilon_{r(n^{(j)})}^{(j)}$, $K^{(j)}=K_{r(n^{(j)})}^{(j)}$,
and $k^{(j)}=k_{n^{(j)}}^{(j)}$. For convenience, let $m^{(j)}=n^{(j)}+1$.
\begin{example}
\begin{figure}
\begin{centering}
\includegraphics[bb=0bp 0bp 482bp 227bp,clip,scale=0.7]{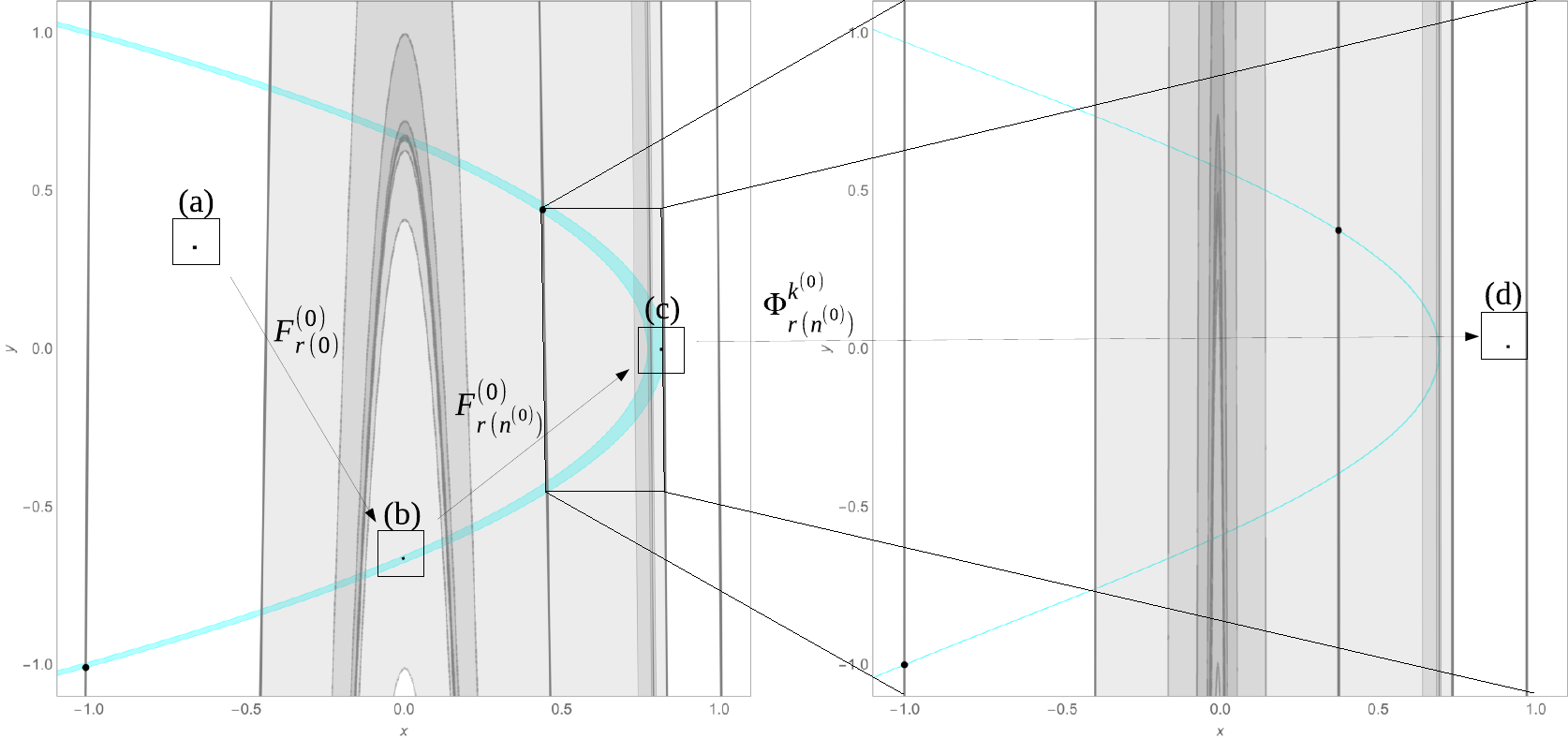}
\par\end{centering}
\begin{centering}
\subfloat[$J_{0}^{(0)}$]{\includegraphics[clip,scale=0.75]{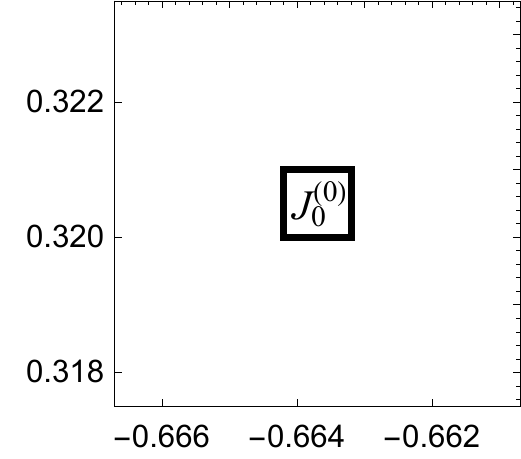}

}\subfloat[$J_{n^{(0)}}^{(0)}$]{\includegraphics[scale=0.75]{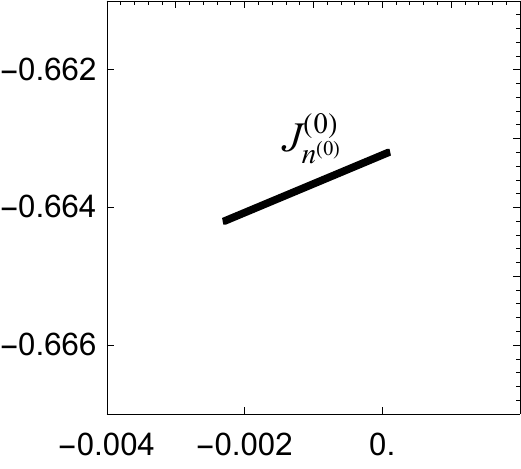}

}\subfloat[$F_{n^{(0)}}^{(0)}\left(J_{n^{(0)}}^{(0)}\right)$]{\includegraphics[scale=0.75]{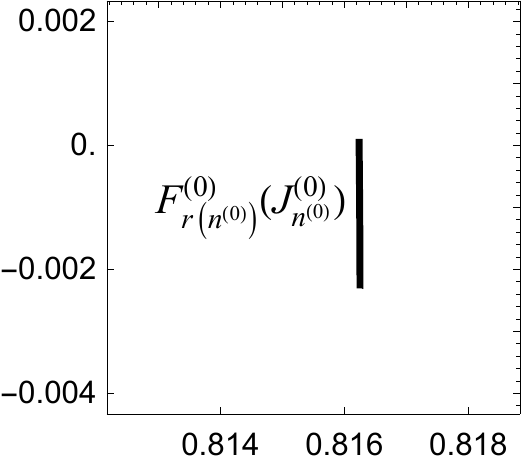}

}
\par\end{centering}
\centering{}\subfloat[\label{fig:double sequence example-bad}$J_{n^{(0)}+1}^{(0)}$]{\includegraphics[scale=0.75]{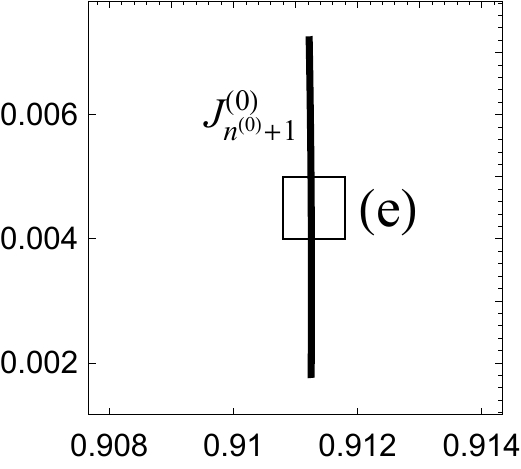}

}\subfloat[\label{fig:double sequence example-square}Zoom in (d)]{\includegraphics[scale=0.75]{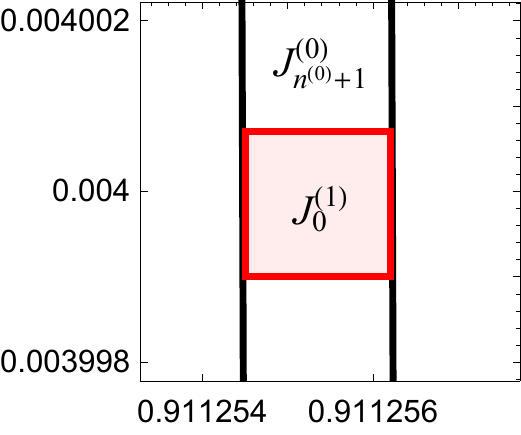}

}\caption{\label{fig:double sequence example}Construction of a double sequence.
The left and right are the graphs for $F_{0}^{(0)}$ and $F_{1}^{(0)}=F_{0}^{(1)}$
respectively. The arrows indicate the iteration and rescaling in the
construction of the double sequence. The sub-figures (a), (b), (c),
and (d) are the zoomed double sequence elements. The scale of (a),
(b), (c), and (d) are chosen to be the same for the reader to compare
the change of the horizontal size.}
\end{figure}
Figure \ref{fig:double sequence example} gives an example of constructing
a double sequence. 

In this example, we choose the same Hénon-like map as in Example \ref{exa:Sequence of wandering domain}.
Select an initial square set $J_{0}^{(0)}=\left[-0.6642,-0.6632\right]\times\left[0.320,0.321\right]\subset A_{0}^{(0)}$
and set $r^{(0)}(0)=0$. 

By the construction of the closest approach, $J_{1}^{(0)}=F_{r(0)}^{(0)}(J_{0}^{(0)})$
and $r^{(0)}(1)=r^{(0)}(0)=0$. From the figure, we see that $J_{1}^{(0)}\subset B_{r(0)}(1)$.
In this example, $\epsilon$ is chosen to be so large that $C_{0}^{r}(1)$
intersects the image $F_{0}(D_{0})$. Thus $K_{r(1)}^{(0)}=K_{0}^{(0)}=0$
and $J_{1}^{(0)}$ lies in the bad region. Set $n^{(0)}=1$. 

By the construction, $J_{n^{(0)}+1}^{(0)}=\Phi_{r(n^{(0)})}^{k^{(0)}}\circ F_{r(n^{(0)})}^{(0)}(J_{n^{(0)}}^{(0)})=\phi_{0}^{(0)}\circ F_{0}^{(0)}(J_{1}^{(0)})$.
The double sequence in this example is chosen in purpose to demonstrate
the set $J_{n^{(0)}+1}^{(0)}$ turns so vertical that the thickness
dominates the horizontal size as in Figure \ref{fig:double sequence example-bad}.
Select a largest square subset $J_{0}^{(1)}$ from $J_{n^{(0)}+1}^{(0)}$
as in Figure \ref{fig:double sequence example-square}. Set $F_{0}^{(1)}=F_{n^{(0)}+1}^{(0)}$. 

The procedure is repeated until the sequence does not enter the bad
region again.
\end{example}
Next, we study the relation between horizontal size and thickness
in a double sequence. For each row $j$, the first set $J_{0}^{(j)}$
is a square so $l_{0}^{(j)}=w_{0}^{(j)}$. When the row stays in the
good region ($n<n^{(j)}$), the next horizontal size can be estimated
by expansion argument $l_{n+1}^{(j)}\geq El_{n}^{(j)}$ (Proposition
\ref{prop:Good region estimates}). When the row first enters the
bad region $n=n^{(j)}$, the expansion argument fails. The vertical
line argument in Chapter \ref{sec:Good and Bad region} shows that
the only way to estimate the horizontal size $l_{n^{(j)}+1}^{(j)}$
is to use the thickness $w_{n^{(j)}+1}^{(j)}$. That is, $l_{n^{(j)}+1}^{(j)}\geq w_{n^{(j)}+1}^{(j)}$.
Proposition \ref{prop:w_n Bound} provides the relation between $l_{n^{(j)}+1}^{(j)}$
and $l_{0}^{(j)}$ by using the thickness. Finally, the horizontal
size $l_{0}^{(j+1)}$ and thickness $w_{0}^{(j+1)}$ of the first
set $J_{0}^{(j+1)}$ in the next row $j+1$ is obtained by the thickness
$w_{n^{(j)}+1}^{(j)}$ by definition.

From the discussion, the horizontal size of any set in the double
sequence can be estimated as follows.
\begin{proposition}
\label{prop:l recurrent relation}\index{horizontal size|textit}\index{time span in the good regions|textit}Given
$\delta>0$ and $I^{v}\supset I^{h}\Supset I$. There exists $\overline{\epsilon}>0$
and $E>1$ such that for all non-degenerate open maps $F\in\hat{\mathcal{I}}_{\delta}(I^{h}\times I^{v},\overline{\epsilon})$
the following property holds:

Let $J\subset A\cup B$ be a square subset of a wandering domain of
$F$ and $\left\{ J_{n}^{(j)}\right\} _{n\geq0,0\leq j\leq\overline{j}}$
be a $J$-double sequence. Then 
\begin{enumerate}
\item $\ln l_{0}^{(j+1)}\geq2m^{(j)}\ln\left\Vert \epsilon^{(j)}\right\Vert +\ln l_{0}^{(j)}$
for all $0\leq j\leq\overline{j}-1$ and 
\item $l_{n+1}^{(j)}\geq El_{n}^{(j)}$ for all $n<n^{(j)}$ and all $0\leq j\leq\overline{j}$.
\end{enumerate}
\end{proposition}
\begin{proof}
Let $\overline{\epsilon}>0$ be small enough such that Proposition
\ref{prop:Good region estimates} and Corollary \ref{cor:w_n bound simplfied}
hold. 

With the help Corollary \ref{cor:w_n bound simplfied}, we are able
to compare $l_{0}^{(j+1)}$ with $l_{0}^{(j)}$ by using the thickness.
That is
\begin{align*}
l_{0}^{(j+1)} & =w_{n^{(j)}+1}^{(j)}\\
 & \geq\left(\prod_{n=0}^{n^{(j)}}c\left\Vert \epsilon_{r(n)}^{(j)}\right\Vert ^{3/2}\right)w_{0}^{(j)}\geq\left(c^{\frac{2}{3}}\left\Vert \epsilon_{r(n^{(j)})}^{(j)}\right\Vert \right)^{\frac{3}{2}(n^{(j)}+1)}w_{0}^{(j)}\\
 & =\left(c^{\frac{2}{3}}\left\Vert \epsilon^{(j)}\right\Vert \right)^{\frac{3}{2}m^{(j)}}l_{0}^{(j)}
\end{align*}
where $c>0$ is a constant. Apply nature logarithm to both sides,
we get 
\begin{eqnarray*}
\ln l_{0}^{(j+1)} & \geq & \frac{3}{2}m^{(j)}\left(\ln\left\Vert \epsilon^{(j)}\right\Vert +\frac{2}{3}\ln c\right)+\ln l_{0}^{(j)}\\
 & \geq & 2m^{(j)}\ln\left\Vert \epsilon^{(j)}\right\Vert +\ln l_{0}^{(j)}.
\end{eqnarray*}
Here we assume that $\overline{\epsilon}$ is small enough so that
$\frac{2}{3}\ln c\geq\frac{1}{3}\ln\left\Vert \epsilon^{(j)}\right\Vert $
for all $0\leq j\leq\overline{j}-1$ to assimilate the constants.

The second inequality follows directly from Proposition \ref{prop:Good region estimates},
the definition of $n^{(j)}$, and a square is $R$-regular when $\overline{\epsilon}$
is small enough.
\end{proof}

The next proposition provides the relation of the perturbation $\epsilon$
between two rows.
\begin{proposition}
\label{prop:e in two chains}\index{row|textit}Given $\delta>0$
and $I^{v}\supset I^{h}\Supset I$. There exists $\overline{\epsilon}>0$
and $\alpha>0$ (universal) such that for all non-degenerate open
maps $F\in\hat{\mathcal{I}}_{\delta}(I^{h}\times I^{v},\overline{\epsilon})$
we have
\begin{equation}
\left\Vert \epsilon^{(j+1)}\right\Vert \leq\left\Vert \epsilon^{(j)}\right\Vert ^{\left\Vert \epsilon^{(j)}\right\Vert ^{-2\alpha}}\label{eq:e in two chains}
\end{equation}
for all $0\leq j\leq\overline{j}-1$.
\end{proposition}
\begin{proof}
By definition and Proposition \ref{prop:Hyperbolicity of the Renormalization Operator},
we have 
\[
\left\Vert \epsilon^{(j+1)}\right\Vert =\left\Vert \epsilon_{r(n^{(j+1)})}^{(j+1)}\right\Vert \leq\left\Vert \epsilon_{0}^{(j+1)}\right\Vert =\left\Vert \epsilon_{r(n^{(j)}+1)}^{(j)}\right\Vert \leq\left(c\left\Vert \epsilon_{r(n^{(j)})}^{(j)}\right\Vert \right)^{2^{k^{(j)}}}=\left(c\left\Vert \epsilon^{(j)}\right\Vert \right)^{2^{k^{(j)}}}
\]
for some constant $c>0$. Apply logarithm to the both side, we get
\begin{equation}
\ln\left\Vert \epsilon^{(j+1)}\right\Vert \leq2^{k^{(j)}}\left(\ln\left\Vert \epsilon^{(j)}\right\Vert +\ln c\right)\leq2^{k^{(j)}-1}\ln\left\Vert \epsilon^{(j)}\right\Vert .\label{eq:e in two chains 1}
\end{equation}
Here we assume that $\overline{\epsilon}>0$ is small enough such
that 
\[
-\frac{1}{2}\ln\left\Vert \epsilon^{(j)}\right\Vert >\ln c
\]
for all $j\geq0$.

Since $J_{n^{(j)}}^{(j)}$ enters the bad region, we have $k^{(j)}>K^{(j)}$.
By Proposition \ref{prop:Good and Bad Regions} and the change base
formula, we get
\begin{equation}
2^{k^{(j)}}>2^{K^{(j)}}=\left(\lambda^{K^{(j)}}\right)^{\frac{\ln2}{\ln\lambda}}\geq c'\left(\frac{1}{\left\Vert \epsilon^{(j)}\right\Vert }\right)^{\frac{\ln2}{2\ln\lambda}}\label{eq:e in two chains 2}
\end{equation}
for some constant $c'>0$. Let $\alpha=\frac{\ln2}{6\ln\lambda}>0$.
Combine (\ref{eq:e in two chains 1}) and (\ref{eq:e in two chains 2}),
we obtain 
\[
\ln\left\Vert \epsilon^{(j+1)}\right\Vert \leq\frac{c'}{2}\left(\frac{1}{\left\Vert \epsilon^{(j)}\right\Vert }\right)^{3\alpha}\ln\left\Vert \epsilon^{(j)}\right\Vert <\left(\frac{1}{\left\Vert \epsilon^{(j)}\right\Vert }\right)^{2\alpha}\ln\left\Vert \epsilon^{(j)}\right\Vert .
\]
Note that $\ln\left\Vert \epsilon^{(j)}\right\Vert <0$. Here we also
assume that $\overline{\epsilon}$ is small enough such that 
\[
\frac{c'}{2}\left(\frac{1}{\left\Vert \epsilon^{(j)}\right\Vert }\right)^{\alpha}\geq\frac{c'}{2}\left(\frac{1}{\left\Vert \epsilon\right\Vert }\right)^{\alpha}>1
\]
for all $j\geq0$. This proves the proposition.
\end{proof}

\subsection{\label{subsec:finite rows}Closest approach cannot enter the bad
region infinitely many times}

A strong contraction on the horizontal size occurs each time when
the double sequence (or closest approach) enters the bad region as
proved in Proposition \ref{prop:l recurrent relation}. The contraction
produces an obstruction to the expansion argument. This section will
resolve the problem by proving the double sequence can have at most
finitely many rows.

Although entering the bad region produces an obstruction to the expansion
argument, it also provides a restriction to the sequence element $J_{n^{(j)}}^{(j)}$:
its horizontal size $l_{n^{(j)}}^{(j)}$ cannot exceed the size of
bad region (Proposition \ref{prop:Good and Bad Regions}). The Two
Row Lemma, which is the final key toward the proof, studies the interaction
between the obstruction and restriction between two consecutive rows
as illustrated in Figure \ref{fig:two rows}.

\begin{figure}
\begin{center}
$\xymatrix{\text{row }j & w_{0}^{(j)}\ar[r]^{\left\Vert \epsilon^{(j)}\right\Vert }\save"1,2"."1,6"!C*+[]+[]\frmu\restore{} & w_{1}^{(j)}\ar[r]^{\left\Vert \epsilon^{(j)}\right\Vert } & \cdots\ar[r]^{\left\Vert \epsilon^{(j)}\right\Vert }\save[]+<+1cm,+1cm>*\txt<8pc>{\mbox{contraction \ensuremath{\left\Vert \epsilon^{(j)}\right\Vert ^{n^{(j)}+1}}}}\restore & w_{n^{(j)}}^{(j)}\ar[r]^{\left\Vert \epsilon^{(j)}\right\Vert } & w_{n^{(j)}+1}^{(j)}\ar@{->}`d[l]`[lllld][lllld]\\
\text{row }j+1 & l_{0}^{(j+1)}\ar[r]^{E}\save"2,2"."2,5"!C*\frmd\restore{} & l_{1}^{(j+1)}\ar[r]^{E} & \cdots\ar[r]^{E}\save[]+<-1cm,-0.6cm>*\txt<8pc>{\mbox{expansion \ensuremath{E^{n^{(j+1)}}}}}\restore & l_{n^{(j+1)}}^{(j+1)} & <\sqrt{\left\Vert \epsilon^{(j+1)}\right\Vert }\sim\left\Vert \epsilon^{(j)}\right\Vert ^{\left\Vert \epsilon^{(j)}\right\Vert ^{-\alpha}}
}
$
\par\end{center}

\caption{\label{fig:two rows}Relations of horizontal size and thickness in
two rows $j$ and $j+1$.}
\end{figure}
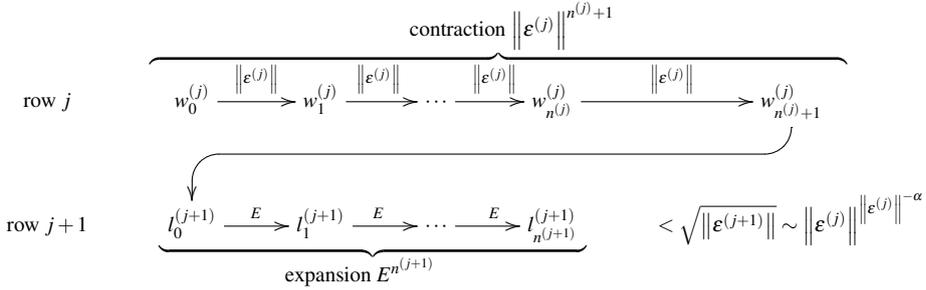

Assume the two rows $j$ and $j+1$ both enters the bad region. 

On row $j+1$, the sequence element $J_{n^{(j+1)}}^{(j+1)}$ enters
the bad region. The size of bad region provides the restriction to
the horizontal size $l_{n^{(j+1)}}^{(j+1)}$. For the sequence elements
prior to $J_{n^{(j+1)}}^{(j+1)}$ on the same row, the horizontal
size expand. This means that the initial sequence element $J_{0}^{(j+1)}$
is restricted by both the size of bad region and the amount of expansion
on row $j+1$. 

On row $j$, the thickness determines the horizontal size of the next
row $j+1$. The contraction of thickness produces the contraction
of horizontal size from $l_{0}^{(j)}$ to $l_{0}^{(j+1)}$. This cause
the obstruction toward the expansion argument.

The following lemma summarize the discussion.
\begin{lemma}[Two Row Lemma]
\label{lem:n recurrent relation}\index{two row lemma}\index{row|textit}\index{horizontal size|textit}Given
$\delta>0$ and $I^{v}\supset I^{h}\Supset I$. There exists $\overline{\epsilon}>0$,
$E>1$, $\alpha>0$ (universal) such that for all non-degenerate open
maps $F\in\hat{\mathcal{I}}_{\delta}(I^{h}\times I^{v},\overline{\epsilon})$
the following property holds:

Let $J\subset A\cup B$ be a square subset of a wandering domain of
$F$ and $\left\{ J_{n}^{(j)}\right\} _{n\geq0,0\leq j\leq\overline{j}}$
be a $J$-double sequence. Then the time span in the good regions\index{time span in the good regions|textit}
$n^{(j)}=m^{(j)}-1$ for row $j$ is bounded below by
\begin{equation}
m^{(j)}>\frac{\ln E}{-2\ln\left\Vert \epsilon^{(j)}\right\Vert }m^{(j+1)}+\left(\frac{1}{\left\Vert \epsilon^{(j)}\right\Vert }\right)^{\alpha}+\frac{1}{-2\ln\left\Vert \epsilon^{(j)}\right\Vert }\ln l_{0}^{(j)}\label{eq:n recurrent relation}
\end{equation}
for all $0\leq j\leq\overline{j}-2$.
\end{lemma}
\begin{proof}
The idea of the proof comes from Figure \ref{fig:two rows}.

On row $j+1$, $J_{n^{(j+1)}}^{(j+1)}$ is in the bad region since
$j+1\leq\overline{j}-1$. The size of $J_{n^{(j+1)}}^{(j+1)}$ cannot
exceed the size of bad region. Let $z_{1},z_{2}\in J_{n^{(j+1)}}^{(j+1)}$
be such that $\left|\pi_{x}z_{2}-\pi_{x}z_{1}\right|=l_{n^{(j+1)}}^{(j+1)}$.
Apply Proposition \ref{prop:Good and Bad Regions} to bound the horizontal
size. We get
\[
l_{n^{(j+1)}}^{(j+1)}\leq\left|\pi_{x}z_{2}-v_{n^{(j+1)}}^{(j+1)}\right|+\left|\pi_{x}z_{1}-v_{n^{(j+1)}}^{(j+1)}\right|\leq2c\sqrt{\left\Vert \epsilon_{n^{(j+1)}}^{(j+1)}\right\Vert }=2c\sqrt{\left\Vert \epsilon^{(j+1)}\right\Vert }
\]
for some constant $c>0$. 

Also, the horizontal size expands on row $j+1$. Proposition \ref{prop:l recurrent relation}
yields 
\[
E^{n^{(j+1)}}l_{0}^{(j+1)}\leq l_{n^{(j+1)}}^{(j+1)}\leq2c\sqrt{\left\Vert \epsilon^{(j+1)}\right\Vert }.
\]
Apply natural logarithm to the both sides, we get 
\begin{eqnarray*}
\ln l_{0}^{(j+1)} & < & -n^{(j+1)}\ln E+\frac{1}{2}\ln\left\Vert \epsilon^{(j+1)}\right\Vert +\ln2c\\
 & = & -m^{(j+1)}\ln E+\frac{1}{2}\ln\left\Vert \epsilon^{(j+1)}\right\Vert +\left(\ln E+\ln2c\right).
\end{eqnarray*}

On row $j$, the thickness contracts. Proposition \ref{prop:l recurrent relation}
provides the contraction as
\begin{eqnarray*}
2m^{(j)}\ln\left\Vert \epsilon^{(j)}\right\Vert  & \leq & \ln l_{0}^{(j+1)}-\ln l_{0}^{(j)}\\
 & < & -m^{(j+1)}\ln E+\frac{1}{2}\ln\left\Vert \epsilon^{(j+1)}\right\Vert +\left(\ln E+\ln2c\right)-\ln l_{0}^{(j)}.
\end{eqnarray*}
Since $\ln\left\Vert \epsilon^{(j)}\right\Vert <0$, we solved 
\[
m^{(j)}>\frac{\ln E}{-2\ln\left\Vert \epsilon^{(j)}\right\Vert }m^{(j+1)}+\frac{1}{4}\frac{\ln\left\Vert \epsilon^{(j+1)}\right\Vert }{\ln\left\Vert \epsilon^{(j)}\right\Vert }+\frac{\ln E+\ln2c}{2\ln\left\Vert \epsilon^{(j)}\right\Vert }+\frac{\ln l_{0}^{(j)}}{-2\ln\left\Vert \epsilon^{(j)}\right\Vert }.
\]

To simplify the second term, apply Proposition \ref{prop:e in two chains}.
We obtain
\begin{eqnarray*}
m^{(j)} & > & \frac{\ln E}{-2\ln\left\Vert \epsilon^{(j)}\right\Vert }m^{(j+1)}+\frac{1}{4}\left(\frac{1}{\left\Vert \epsilon^{(j)}\right\Vert }\right)^{2\alpha}+\frac{\ln E+\ln2c}{2\ln\left\Vert \epsilon^{(j)}\right\Vert }+\frac{1}{-2\ln\left\Vert \epsilon^{(j)}\right\Vert }\ln l_{0}^{(j)}\\
 & = & \frac{\ln E}{-2\ln\left\Vert \epsilon^{(j)}\right\Vert }m^{(j+1)}+\left(\frac{1}{\left\Vert \epsilon^{(j)}\right\Vert }\right)^{\alpha}\left[\frac{1}{4}\left(\frac{1}{\left\Vert \epsilon^{(j)}\right\Vert }\right)^{\alpha}+\frac{\ln E+\ln2c}{2\ln\left\Vert \epsilon^{(j)}\right\Vert }\left\Vert \epsilon^{(j)}\right\Vert ^{\alpha}\right]\\
 &  & +\frac{1}{-2\ln\left\Vert \epsilon^{(j)}\right\Vert }\ln l_{0}^{(j)}\\
 & > & \frac{\ln E}{-2\ln\left\Vert \epsilon^{(j)}\right\Vert }m^{(j+1)}+\left(\frac{1}{\left\Vert \epsilon^{(j)}\right\Vert }\right)^{\alpha}+\frac{1}{-2\ln\left\Vert \epsilon^{(j)}\right\Vert }\ln l_{0}^{(j)}.
\end{eqnarray*}
Here we assume that $\overline{\epsilon}$ is sufficiently small
such that 
\[
\frac{1}{4}\left(\frac{1}{\left\Vert \epsilon^{(j)}\right\Vert }\right)^{\alpha}+\frac{\ln E+\ln2c}{2\ln\left\Vert \epsilon^{(j)}\right\Vert }\left\Vert \epsilon^{(j)}\right\Vert ^{\alpha}>1
\]
for all $j\geq0$ to assimilate the constants.
\end{proof}

If the double sequence has infinite rows, then the obstruction and
restriction both happens infinitely many times. For this to happen,
the obstruction must beats (or balance with) the restriction. However,
it is not possible to compare the contraction of the horizontal size
with the size of bad region directly because the time span in the
good regions also interacts with the obstruction and restriction as
(\ref{eq:n recurrent relation}) shows. So we turn to analyze the
relation between the time span in the good regions versus the the
number of rows in a double sequence.

If the double sequence enters the bad region twice, we apply the Two
Row Lemma to row $0$ and $1$. The restriction from the bad region
says the the horizontal size is bounded by the size of bad region
$\left\Vert \epsilon^{(1)}\right\Vert $. At this moment, there are
no information for the expansion on row $1$. So the restriction comes
only from the size of bad region. To balance the obstruction with
the restriction, the total contraction $\left\Vert \epsilon^{(0)}\right\Vert ^{m^{(0)}}$
on row $0$ must have at least the same order as the size of bad region.
Thus, the time span in the good regions $m^{(0)}$ must be large ($\approx\left\Vert \epsilon^{(0)}\right\Vert ^{-1}$)
by Proposition \ref{prop:e in two chains} because the contraction
and the size of the bad region come from the perturbation on two different
rows.

If a double sequence enters the bad region three times, we apply the
Two Row Lemma twice. First, we apply the lemma to row $1$ and $2$.
The previous paragraph says that $m^{(1)}$ is large ($\approx\left\Vert \epsilon^{(1)}\right\Vert ^{-1}$).
Then, apply the Two Row Lemma again to row $0$ and $1$. Unlike the
previous paragraph, now the expansion on row $1$ is determined when
the Two Row Lemma is applied to row $1$ and $2$. So the restriction
comes from both the size of bad region $\left\Vert \epsilon^{(1)}\right\Vert $
and the expansion of horizontal size $E^{m^{(1)}}\sim E^{\left\Vert \epsilon^{(1)}\right\Vert ^{-1}}$.
To balance the obstruction with the restriction, the total contraction
$\left\Vert \epsilon^{(0)}\right\Vert ^{m^{(0)}}$ on row $0$ must
have at least the same order as $E^{-\left\Vert \epsilon^{(1)}\right\Vert ^{-1}}\left\Vert \epsilon^{(1)}\right\Vert $.
This yields a larger estimate (compare to the previous paragraph)
for the time span in the good regions $m^{(0)}$ because of the expansion.

If the double sequence enters the bad region infinite times, we start
from any arbitrary row $j+k+1$ then apply the Two Row Lemma recurrently
to the rows $j$, $j+1$, $\cdots$, and $j+k+1$ in reverse order.
It is important that the contribution of obstruction and restriction
comes from the perturbation in two different rows as illustrated in
Figure \ref{fig:two rows}. With the help from Proposition \ref{prop:e in two chains},
the contribution from different rows makes the time span in the good
regions increases each time when the Two Row Lemma is applied. This
gives the following lemma
\begin{lemma}
\label{lem:n lower bound}\index{row|textit}\index{time span in the good regions|textit}Given
$\delta>0$ and $I^{v}\supset I^{h}\Supset I$. There exists $\overline{\epsilon}>0$
such that for all non-degenerate open maps $F\in\hat{\mathcal{I}}_{\delta}(I^{h}\times I^{v},\overline{\epsilon})$
the following property holds:

Let $J\subset A\cup B$ be a square subset of a wandering domain of
$F$ and $\left\{ J_{n}^{(j)}\right\} _{n\geq0,0\leq j\leq\overline{j}}$
be a $J$-double sequence. Then the time span in the good regions
$n^{(j)}$ for row $j$ is bounded below by
\begin{equation}
m^{(j)}=n^{(j)}+1>\frac{2^{k}}{\left\Vert \epsilon^{(j)}\right\Vert ^{\alpha}}+\frac{1}{-2\ln\left\Vert \epsilon^{(j)}\right\Vert }\ln l_{0}^{(j)}\label{eq:m lower bound}
\end{equation}
for all $j$ and $k$ with $0\leq j\leq\overline{j}-2$ and $0\leq k\leq(\overline{j}-2)-j$
where $\alpha>0$ is a universal constant.

In particular for the case $j=0$ 
\begin{equation}
m^{(0)}=n^{(0)}+1>\frac{2^{k}}{\left\Vert \epsilon^{(0)}\right\Vert ^{\alpha}}+\frac{1}{-2\ln\left\Vert \epsilon^{(0)}\right\Vert }\ln l_{0}^{(0)}\label{eq:m lower bound for row 0}
\end{equation}
for all $0\leq k\leq\overline{j}-2$.
\end{lemma}
\begin{proof}
We prove (\ref{eq:m lower bound}) holds for all $0\leq j\leq\overline{j}-k-2$
by induction on $k\leq\overline{j}-2$. Let $\overline{\epsilon}$
be small enough such that Proposition \ref{prop:l recurrent relation},
Proposition \ref{prop:e in two chains}, and Lemma \ref{lem:n recurrent relation}
hold.

For the base case $k=0$. Apply (\ref{eq:n recurrent relation}),
we have
\begin{eqnarray*}
m^{(j)} & > & \frac{\ln E}{-2\ln\left\Vert \epsilon^{(j)}\right\Vert }m^{(j+1)}+\left(\frac{1}{\left\Vert \epsilon^{(j)}\right\Vert }\right)^{\alpha}+\frac{1}{-2\ln\left\Vert \epsilon^{(j)}\right\Vert }\ln l_{0}^{(j)}\\
 & > & \frac{1}{\left\Vert \epsilon^{(j)}\right\Vert ^{\alpha}}+\frac{1}{-2\ln\left\Vert \epsilon^{(j)}\right\Vert }\ln l_{0}^{(j)}
\end{eqnarray*}
for all $j$ with $0\leq j\leq\overline{j}-2$.

Assume that there exists $k$ with $1\leq k\leq\overline{j}-2$ such
that (\ref{eq:m lower bound}) holds for all $j$ with $0\leq j\leq\overline{j}-k-2$.
If $k+1\leq\overline{j}-2$ and $0\leq j<\overline{j}-(k+1)-2$, then
$k\leq\overline{j}-2$ and $1\le j+1\leq\overline{j}-k-2$. The induction
hypothesis yields
\begin{eqnarray}
m^{(j+1)} & > & \frac{2^{k}}{\left\Vert \epsilon^{(j+1)}\right\Vert ^{\alpha}}+\frac{1}{-2\ln\left\Vert \epsilon^{(j+1)}\right\Vert }\ln l_{0}^{(j+1)}.\label{eq:n-induction hypothesis}
\end{eqnarray}
Substitute (\ref{eq:n-induction hypothesis}) into (\ref{eq:n recurrent relation}),
we get
\begin{align}
m^{(j)}> & \frac{\ln E}{-2\ln\left\Vert \epsilon^{(j)}\right\Vert }\frac{2^{k}}{\left\Vert \epsilon^{(j+1)}\right\Vert ^{\alpha}}+\frac{\ln E}{-2\ln\left\Vert \epsilon^{(j)}\right\Vert }\frac{1}{-2\ln\left\Vert \epsilon^{(j+1)}\right\Vert }\ln l_{0}^{(j+1)}\nonumber \\
 & +\frac{1}{-2\ln\left\Vert \epsilon^{(j)}\right\Vert }\ln l_{0}^{(j)}.\label{eq:n-substitute}
\end{align}

For the first term of (\ref{eq:n-substitute}), we have 
\[
\ln\frac{1}{\left\Vert \epsilon^{(j)}\right\Vert }<\frac{1}{\left\Vert \epsilon^{(j)}\right\Vert }.
\]
Together with (\ref{eq:e in two chains}), we get 
\[
\frac{\ln E}{-2\ln\left\Vert \epsilon^{(j)}\right\Vert }\frac{2^{k}}{\left\Vert \epsilon^{(j+1)}\right\Vert ^{\alpha}}>2^{k}\left[\frac{\ln E}{2}\left(\frac{1}{\left\Vert \epsilon^{(j)}\right\Vert }\right)^{\alpha\left\Vert \epsilon^{(j)}\right\Vert ^{-2\alpha}-1}\right]>2^{k+2}\left(\frac{1}{\left\Vert \epsilon^{(j)}\right\Vert }\right)^{\alpha}.
\]
Here, we assume that $\overline{\epsilon}$ is small enough such that
\[
\frac{\ln E}{8}>\left\Vert \epsilon^{(j)}\right\Vert 
\]
and 
\[
\alpha\left\Vert \epsilon^{(j)}\right\Vert ^{-2\alpha}-2>\alpha
\]
for all $j\geq0$.

For the second term of (\ref{eq:n-substitute}), apply Proposition
\ref{prop:l recurrent relation}. We get 
\[
\frac{\ln E}{-2\ln\left\Vert \epsilon^{(j)}\right\Vert }\frac{1}{-2\ln\left\Vert \epsilon^{(j+1)}\right\Vert }\ln l_{0}^{(j+1)}>\frac{\ln E}{2\ln\left\Vert \epsilon^{(j+1)}\right\Vert }m^{(j)}+\frac{\ln E}{-2\ln\left\Vert \epsilon^{(j)}\right\Vert }\frac{1}{-2\ln\left\Vert \epsilon^{(j+1)}\right\Vert }\ln l_{0}^{(j)}.
\]

Combine the results to (\ref{eq:n-substitute}), we obtain 
\[
m^{(j)}>2^{k+2}\left(\frac{1}{\left\Vert \epsilon^{(j)}\right\Vert }\right)^{\alpha}+\frac{\ln E}{2\ln\left\Vert \epsilon^{(j+1)}\right\Vert }m^{(j)}+\frac{1}{-2\ln\left\Vert \epsilon^{(j)}\right\Vert }\left(1+\frac{\ln E}{-2\ln\left\Vert \epsilon^{(j+1)}\right\Vert }\right)\ln l_{0}^{(j)}.
\]
Then
\begin{align*}
\left(1+\frac{\ln E}{-2\ln\left\Vert \epsilon^{(j+1)}\right\Vert }\right)m^{(j)}> & 2^{k+2}\left(\frac{1}{\left\Vert \epsilon^{(j)}\right\Vert }\right)^{\alpha}\\
 & +\frac{1}{-2\ln\left\Vert \epsilon^{(j)}\right\Vert }\left(1+\frac{\ln E}{-2\ln\left\Vert \epsilon^{(j+1)}\right\Vert }\right)\ln l_{0}^{(j)}.
\end{align*}
Solve for $m^{(j)}$, we get
\begin{eqnarray*}
m^{(j)} & > & 2^{k+2}\left(1+\frac{\ln E}{-2\ln\left\Vert \epsilon^{(j+1)}\right\Vert }\right)^{-1}\left(\frac{1}{\left\Vert \epsilon^{(j)}\right\Vert }\right)^{\alpha}+\frac{1}{-2\ln\left\Vert \epsilon^{(j)}\right\Vert }\ln l_{0}^{(j)}.
\end{eqnarray*}

To simplify the inequality, we assume that $\overline{\epsilon}$
is small enough such that 
\[
\frac{\ln E}{-2\ln\left\Vert \epsilon^{(j+1)}\right\Vert }\leq\frac{\ln E}{-2\ln\overline{\epsilon}}<1
\]
for all $j\geq0$. Therefore, 
\[
m^{(j)}>\frac{2^{k+1}}{\left\Vert \epsilon^{(j)}\right\Vert ^{\alpha}}+\frac{1}{-2\ln\left\Vert \epsilon^{(j)}\right\Vert }\ln l_{0}^{(j)}
\]
and the lemma is proved by induction.
\end{proof}

The lemma shows that the restriction beats the obstruction because
(\ref{eq:m lower bound for row 0}) approaches infinity as the total
number of rows in a double sequence increases. This proves
\begin{proposition}
\label{prop:Finite number of chain}\index{row|textit}\index{double sequence!number of rows|textit}Given
$\delta>0$ and $I^{v}\supset I^{h}\Supset I$. There exists $\overline{\epsilon}>0$
such that for all non-degenerate open maps $F\in\hat{\mathcal{I}}_{\delta}(I^{h}\times I^{v},\overline{\epsilon})$
the following property holds:

Let $J\subset A\cup B$ be a square subset of a wandering domain of
$F$ and $\left\{ J_{n}^{(j)}\right\} _{n\geq0,0\leq j\leq\overline{j}}$
be a $J$-double sequence. Then the number of rows $\overline{j}$
for the double sequence is finite.
\end{proposition}

\subsection{\label{subsec:Nonexistence of wandering domain}Nonexistence of wandering
domain}

Finally, the main theorem is concluded as follows.
\begin{theorem}
\label{thm:nonexistence of wandering domain}\index{wandering domain!nonexistence|textit}Given
$\delta>0$ and $I^{v}\supset I^{h}\Supset I$. There exists $\overline{\epsilon}>0$
such that every non-degenerate open Hénon-like map $F\in\mathcal{I}_{\delta}(I^{h}\times I^{v},\overline{\epsilon})$
does not have wandering domains.
\end{theorem}
\begin{proof}
Assume that $\overline{\epsilon}>0$ is small enough such that Proposition
\ref{prop:Hyperbolicity of the Renormalization Operator} holds and
$F\in\mathcal{I}_{\delta}(I^{h}\times I^{v},\overline{\epsilon})$.
There exists $0<\delta_{R}<\delta$ and $I\Subset I_{R}^{h}\subset I^{h}$
such that $F_{n}\in\mathcal{H}_{\delta_{R}}(I_{R}^{h}\times I_{n}^{v},\overline{\epsilon})$
for all $n\geq0$.

Prove by contradiction. Assume that $F$ has a wandering domain. Let
$\overline{\epsilon}'>0$ be small enough such that Proposition \ref{prop:l recurrent relation}
and Proposition \ref{prop:Finite number of chain} holds for $\delta_{R}$
and $I_{R}^{h}\times I_{R}^{h}$. By Proposition \ref{prop:Hyperbolicity of the Renormalization Operator},
there exists $N\geq0$ such that $F_{N}\in\hat{\mathcal{I}}_{\delta_{R}}(I_{R}^{h}\times I_{R}^{h},\overline{\epsilon}')$.
Set $\hat{F}=F_{N}|_{I_{R}^{h}\times I_{R}^{h}}$.

By Corollary \ref{cor:Wandering Domain/Renormalization}, $F_{N}$
has a wandering domain $J$ in $D(F_{N})\subset I^{h}(F_{N})\times I_{N}^{v}$.
If $J\subset B(F_{N})$, then $J\subset I_{R}^{h}\times I_{N}^{v}$
and so $F^{2}(J)\subset B(F_{N})\cap\left(I_{R}^{h}\times I_{R}^{h}\right)$.
If $J\subset A(F_{N})$, there exists $n>0$ such that $F^{n}(J)\subset B(F_{N})$
by Proposition \ref{prop:Dynamics of the partition on D}. If $J\subset C(F_{N})$,
then $F(J)\subset B(F_{N})$. Without lose of generality, we may assume
that $J\subset B(F_{N})\cap\left(I_{R}^{h}\times I_{R}^{h}\right)$.
Hence, $J\subset B(\hat{F})$ is a wandering domain of the restriction
$\hat{F}$.

Let $\hat{J}$ be a nonempty square subset of $J$ and $\left\{ J_{n}^{(j)}\right\} _{n\geq0,0\leq j\leq\overline{j}}$
be a $\hat{J}$-double sequence. By Proposition \ref{prop:Finite number of chain},
$\overline{j}$ is finite. Then the second property of Proposition
\ref{prop:l recurrent relation} implies that 
\[
\lim_{n\rightarrow\infty}l_{n}^{(\overline{j})}=\infty
\]
which is a contraction. Therefore, $F$ does not have wandering domains.
\end{proof}

\begin{remark}
\label{rem:Absence of wandering domain for CLM-renormalizable}The
result for Theorem \ref{thm:nonexistence of wandering domain} also
applies to infinitely CLM-renormalizable maps if all levels of renormalization
are defined on a sufficiently large domain. This is because of the
hyperbolicity of the Hénon renormalization operator.

Assume that $F$ is a strongly dissipative infinitely CLM-renormalizable
Hénon-like map. By the hyperbolicity of the renormalization operator
\cite[Theorem 4.1]{de2005renormalization}, there exists $N\geq0$
such that $F_{n}$ is sufficiently close to the fixed point $G$ for
all $n\geq N$. This means that $F_{n}$ is renormalizable for all
$n\geq N$ and hence $F_{N}$ is infinitely renormalizable in the
sense of this article. Thus, we can apply the theorem to $F_{N}$
to conclude $F$ does not have wandering domains.
\end{remark}
As a consequence, the absence of wandering domains provides the information
of the topology as follows.
\begin{corollary}
Given $\delta>0$ and $I^{v}\supset I^{h}\Supset I$. There exists
$\overline{\epsilon}>0$ such that for any non-degenerate open map
$F\in\mathcal{I}_{\delta}(I^{h}\times I^{v},\overline{\epsilon})$,
the union of the stable manifolds for the period doubling periodic
points is dense in the domain.
\end{corollary}

\appendix

\bibliographystyle{spmpsci}
\bibliography{Wandering_Domain}

\addcontentsline{toc}{section}{\refname}

\settowidth{\nomlabelwidth}{$\mathcal{U}_{\delta}$}
\printnomenclature{}

\addcontentsline{toc}{section}{Nomenclature}

\printindex{}
\end{document}